\documentclass[]{amsart}

\usepackage{amssymb}
\usepackage{amsfonts}
\usepackage{amsmath}
\usepackage{amsthm}
\usepackage{mathtools}
\usepackage{graphicx}
\usepackage{mathrsfs}
\usepackage{dsfont}
\usepackage{amscd}
\usepackage{multirow}
\usepackage[all]{xy}
\normalfont
\usepackage[T1]{fontenc}
\usepackage{calligra}
\usepackage{verbatim}
\usepackage[usenames,dvipsnames]{color}
\usepackage[colorlinks=true,linkcolor=Blue,citecolor=Violet]{hyperref}
\usepackage{enumerate}
\usepackage{slashed}
\usepackage{nicematrix}
\usepackage{microtype}

\newcommand{\N}{{\mathds{N}}}
\newcommand{\Z}{{\mathds{Z}}}

\newcommand{\R}{{\mathds{R}}}
\newcommand{\C}{{\mathds{C}}}
\newcommand{\T}{{\mathds{T}}}

\newcommand{\D}{{\mathfrak{D}}}
\newcommand{\A}{{\mathfrak{A}}}
\newcommand{\B}{{\mathfrak{B}}}

\newcommand{\Lip}[1][L]{{\mathsf{#1}}}
\newcommand{\TLip}{{\mathsf{T}}}
\newcommand{\SLip}{{\mathsf{S}}}
\newcommand{\Hilbert}[1][H]{{\mathscr{#1}}}

\newcommand{\dpropinquity}[1]{{\mathsf{\Lambda}^\ast_{#1}}}

\newcommand{\dmetpropinquity}[1]{{\mathsf{\Lambda}^{\ast\mathsf{met}}_{#1}}}

\newcommand{\spectralpropinquity}[1]{{\mathsf{\Lambda}^{\mathsf{spec}}_{#1}}}
\newcommand{\oppropinquity}[1]{{\mathsf{\Lambda}^{\mathsf{op}}_{#1}}}

\newcommand{\Kantorovich}[1]{{\mathsf{mk}_{#1}}}
\newcommand{\KantorovichJ}[1]{{\mathsf{MK}_{#1}}}
\newcommand{\KantorovichMod}[1]{{\mathsf{k}_{#1}}}

\newcommand{\KantorovichAlt}[1]{{\mathrm{mk^{\mathrm{alt}}_{#1}}}}

\newcommand{\Haus}[1]{{\mathsf{Haus}\!\left[{#1}\right]\,}}
\newcommand{\LHaus}[1]{{\overrightarrow{\mathsf{Haus}}\!\left[{#1}\right]\,}}

\newcommand{\StateSpace}{{\mathscr{S}}}
\newcommand{\ModStateSpace}{{\widetilde{\mathscr{S}}}}

\newcommand{\MongeKant}{{Mon\-ge-Kan\-to\-ro\-vich metric}}

\newcommand{\gMVB}{metrical C*-correspondence}

\newcommand{\mcc}[3]{{\mathrm{metCor}\left({#1},{#2},{#3}\right)}}

\newcommand{\qcms}{quantum compact metric space}

\newcommand{\unit}{1}

\newcommand{\sa}[1]{{\mathfrak{sa}\left({#1}\right)}}

\newcommand{\UIsoMono}[2]{{\mathsf{UIso}_{#1}\left({#2}\right)}}

\newcommand{\inner}[3]{{\left<{#1},{#2}\right>_{#3}}}

\newcommand{\dom}[1]{{\operatorname*{dom}\left({#1}\right)}}

\newcommand{\norm}[2]{\left\|{#1}\right\|_{#2}}

\newcommand{\targetsettunnel}[3]{{\mathfrak{t}_{#1}\left({#2}\middle\vert{#3}\right)}}

\newcommand{\multiplicity}[2]{{\mathrm{multiplicity}\left({#1}\middle\vert{#2}\right)}}

\newcommand{\CDN}[1][DN]{{\mathsf{#1}}}
\newcommand{\TDN}{{\mathsf{TN}}}

\newcommand{\worknote}[1]{}
\newcommand{\opnorm}[3]{{\left|\mkern-1.5mu\left|\mkern-1.5mu\left| {#1} \right|\mkern-1.5mu\right|\mkern-1.5mu\right|_{#3}^{#2}}}

\newcommand{\tunnelsep}[2]{{\mathsf{sep}\left({#2}\middle\vert{#1}\right)}}
\newcommand{\tunnelreach}[2]{{\rho\left({#1}\middle\vert{#2}\right)}}
\newcommand{\tunnelhalfreach}[2]{{\overrightarrow{\mathrm{\rho}}\left({#1}\middle\vert{#2}\right)}}

\newcommand{\tunnelmodreach}[2]{{\rho\left({#1}\middle\vert{#2}\right)}}

\newcommand{\tunnelmagnitude}[2]{{\mu\left({#1}\middle\vert{#2}\right)}}
\newcommand{\tunnelmodmagnitude}[2]{{\mu\left({#1}\middle\vert{#2}\right)}}
\newcommand{\tunnelmodsymmagnitude}[2]{{\mu_{\text{sym}}\left({#1}\middle\vert{#2}\right)}}

\newcommand{\tunnelextent}[1]{{\chi\left({#1}\right)}}

\newcommand{\tunneldispersion}[2]{{\mathsf{dis}\left({#2}\middle\vert {#1} \right)}}

\newcommand{\alg}[1]{{\mathfrak{#1}}}

\newcommand{\module}[1]{{\mathscr{#1}}}

\newcommand{\resolvent}[2]{\mathcal{R}\left({#1} ; {#2} \right)}
\newcommand{\spectrum}[1]{\mathrm{Sp}\left({#1}\right)}

\newcommand{\supp}[1]{\mathrm{supp}\left({#1}\right)}
\newcommand{\tunnel}[6]{{ {#1}: {#2} \xleftarrow{#3} {#4} \xrightarrow{#5} {#6} }}

\theoremstyle{plain}
\newtheorem{theorem}{Theorem}[section]

\newtheorem{corollary}[theorem]{Corollary}

\newtheorem{lemma}[theorem]{Lemma}
\newtheorem{proposition}[theorem]{Proposition}

\newtheorem{theorem-definition}[theorem]{Theorem-Definition}

\theoremstyle{definition}
\newtheorem{definition}[theorem]{Definition}

\newtheorem{example}[theorem]{Example}

\newtheorem{convention}[theorem]{Convention}
\newtheorem{hypothesis}[theorem]{Hypothesis}

\theoremstyle{remark}

\newtheorem{remark}[theorem]{Remark}

\newtheorem{notation}[theorem]{Notation}

\renewcommand{\geq}{\geqslant}
\renewcommand{\leq}{\leqslant}

\newcommand{\Siep}{{Sierpi{\'n}ski}}
\newcommand{\SiepG}{{{\Siep} gasket}}

\newcommand{\Dirac}[1][D]{{\slashed{#1}}}

\numberwithin{equation}{section}

\allowdisplaybreaks[4]

\begin{document}

\title[]{Continuity of the Spectrum of Dirac Operators of Spectral Triples for the Spectral Propinquity}
\author{Fr\'{e}d\'{e}ric Latr\'{e}moli\`{e}re}
\email{frederic@math.du.edu}
\urladdr{http://www.math.du.edu/\symbol{126}frederic}
\address{Department of Mathematics \\ University of Denver \\ Denver CO 80208}

\date{\today}

\begin{abstract}
  The spectral propinquity is a distance, up to unitary equivalence, on the class of metric spectral triples. We prove in this paper that if a sequence of metric spectral triples converges for the propinquity, then the spectra of the Dirac operators for these triples do converge to the spectrum of the Dirac operator at the limit. We obtain this result by first proving that, in an appropriate sense induced by some natural metric, the bounded, continuous functional calculi defined by the Dirac operators also converge. As an application of our work, we see, in particular, that action functionals of a wide class of metric spectral triples are continuous for the spectral propinquity, which clearly connects convergence for the spectral propinquity with the applications of noncommutative geometry to mathematical physics. This fact is a consequence of results on the continuity of the multiplicity of eigenvalues of Dirac operators. In particular, we formalize convergence of adjoinable operators of different C*-correspondences, endowed with appropriate quantum metric data.
\end{abstract}

\thanks{The corresponding author states that there is no conflict of interest.}
\thanks{This manuscript has no associated data.}
\maketitle

\let\oldtocsection=\tocsection
\let\oldtocsubsection=\tocsubsection
\let\oldtocsubsubsection=\tocsubsubsection
 
\renewcommand{\tocsection}[2]{\hspace{0em}\oldtocsection{#1}{#2}}
\renewcommand{\tocsubsection}[2]{\hspace{2em}\oldtocsubsection{#1}{#2}}
\renewcommand{\tocsubsubsection}[2]{\hspace{2.5em}\oldtocsubsubsection{#1}{#2}}

\tableofcontents

\section{Introduction}

Spectral triples \cite{Connes,Connes89}, far-reaching generalizations of Dirac operators on spin bundles of compact Riemannian spin manifolds, have emerged as the preferred means to encode geometric information of quantum spaces. Spectral triples have been constructed to describe a geometry over such quantum spaces as quantum tori \cite{Connes80}, quantum spheres \cite{violette02}, quantum groups \cite{Dabrowski05}, group C*-algebras \cite{Connes89}, C*-crossed-products \cite {Hawkins13} --- i.e. quantum orbit spaces --- among many others. Spectral triples also provide a framework for the study of singular geometries, such as the geometry of fractals \cite{Lapidus08,Lapidus14} and orbifolds \cite{Harju16}. Connes discovered, moreover, that spectral triples can be used to describe mathematical physics models, by means of a functional associated to them, called the spectral action \cite{Connes}. In essence, noncommutative geometry offers the very exciting possibility, as discussed for instance in \cite{Marcolli18}, to encode quantum physics into geometry, in the same spirit as general relativity encodes gravity in Lorentzian geometry.

In particular, matrix models for quantum field theories have become an interesting tool for the study of certain mathematical physics questions, related to quantum gravity, quantum cosmology, particle physics and string theory. The asymptotic behavior of such models, as the dimension of the underlying algebras grows to infinity, is of central interest \cite{Kimura01,Schreivogl13,Barrett15,Connes97,Seiberg99}. If the physics or the geometry of a quantum space is described via a spectral triple over a finite dimensional algebra, then a crucial problem is to find a formal means to understand what a notion of limit for spectral triples ought to be. Such a notion could prove interesting for the general study of what it means to approximate a Riemannian geometry, or any of the far-reaching generalization of such a geometry given by spectral triples, with another geometry. Using spectral triples allows us to expand the class of possible spaces that can be approximated, as well as what spaces may be used as approximations, such as matrix algebras, endowed with the geometry from a spectral triple, approximating a sphere \cite{Rieffel15} or a quantum torus \cite{Latremoliere13c,Latremoliere21a}.

In \cite{Latremoliere18g}, we used our research in noncommutative metric geometry \cite{Latremoliere13,Latremoliere13b,Latremoliere14,Latremoliere15,Latremoliere16c,Latremoliere18b,Latremoliere18d} to introduce the \emph{spectral propinquity}, a distance, up to unitary equivalence, on a class of spectral triples, which include many of the standard examples. We then proved in \cite{Latremoliere21a} that a matrix model from mathematical physics, called the fuzzy torus \cite{Connes97,Junge16,Schreivogl13,Kimura01,Barrett15}, endowed with a natural spectral triple, does indeed converge, as its dimension grows to infinity, to a spectral triple over the classical torus; moreover we study generalizations of these models whose limits are general quantum tori. In \cite{Latremoliere20a}, we also employed our distance to prove that spectral triples constructed over certain fractals \cite{Lapidus08}, including the {\SiepG}, are limits of simple spectral triples over finite graphs.

Naturally, for the spectral propinquity to be helpful to the project of studying geometry and physics by means of approximations, we must understand how some of the most fundamental properties of spectral triples behave under convergence for the spectral propinquity. These properties include the induced Connes' metric on the state space of the underlying algebra, the spectrum of the associated Dirac operator, the bounded continuous functional calculus for these Dirac operators, and even the spectral action functionals defined by these spectral triples \cite{Connes96,Connes96b}. This is the topic of the present article: we establish that, indeed, all the properties in this list enjoy some form of continuity with respect to the spectral propinquity, at least under very natural assumptions. It thus becomes possible to compute the spectrum of the Dirac operator of a spectral triple constructed as a limit of a sequence of other spectral triples for the propinquity, as a limit of the spectra of the Dirac operators in our sequence. It is also possible, under reasonable assumptions, to show that the spectral action functionals of the limit spectral triple are indeed limits of the action functionals of the approximating spectral triples.

The study of the continuity of spectra of various operators of geometric importance, such as the Laplacians, with respect to the Gromov-Hausdorff distance, is an important area of inquiry in classical Riemannian geometry, and a tool in trying to relate the spectrum of these operators with geometric invariants \cite{Lott02}. Our present work show how to define a general setting for the study of this continuity of the spectrum, thanks to the spectral propinquity, and opens the possibility to extend work on convergence of the spectra of collapsing sequences of manifolds, as seen for instance in \cite{Fukaya87}, \cite{Lott02b} \cite{Roos20}, \cite{Lott14}, not only to smooth limits, but to singular limits, from orbifolds to fractals, via noncommutative geometry methods, and of course, to entirely noncommutative spaces. Other framework based on functional analysis have been proposed \cite{Kuwae03}; however, our approach is based on defining an actual metric, up to unitary equivalence, between spectral triples (the spectral data being encoded differently in \cite{Kuwae03}), and is applicable to noncommutative geometry as well. We hope to explore these issues in subsequent publications.

\medskip

This paper begins with a review of the context and basic definitions and concepts from noncommutative metric geometry, which we will involve in this work. We then introduce a notion of convergence for operators defined on certain C*-correspondences, as a basic tool to formalize our later results. Next, we provide equivalent formulations of the spectral propinquity, which will prove helpful to its study. Putting these two parts together, we prove that the bounded continuous functional calculus associated with Dirac operators of convergent spectral triples do indeed converge to the bounded, continuous functional calculus of the Dirac operator at the limit. In turn, this result about functional calculus allows us to prove the core result about a form of continuity of the spectra of Dirac operators for the spectral propinquity. The results thus far are all quite general. If we add some very reasonable assumptions, we then also prove the continuity of the multiplicities of the eigenvalues of the Dirac operators for the spectral propinquity, which, in turns, implies the continuity of action functionals for the spectral propinquity.

\bigskip

Our story thus begins with the introduction by A. Connes, in 1985, at the Coll{\`e}ge de France, of the fundamental idea of the spectral approach to noncommutative geometry, and its application to mathematical physics: the spectral triple \cite{Connes}. There are many variations on the definition of a spectral triple in the literature, though they all share the following common properties.

\begin{definition}[{\cite{Connes89}}]
  A \emph{spectral triple} $(\A,\Hilbert,\Dirac)$ is a triple consisting of a unital C*-algebra $\A$, given with an implicit unital *-representation on a Hilbert space $\Hilbert$, together with a self-adjoint operator $\Dirac$, defined on a dense subspace $\dom{\Dirac}$ of $\Hilbert$, such that
  \begin{itemize}
  \item $D + i$ has a compact inverse,
  \item the space
    \begin{equation*}
      \A_{\Dirac} \coloneqq \left\{ a \in \sa{\A} : a\cdot\dom{\Dirac}\subseteq\dom{\Dirac}\text{, }[\Dirac,a]\text{ is bounded} \right\}
    \end{equation*}
    is dense in $\A$.
  \end{itemize}

  The operator $\Dirac$ is called the \emph{Dirac operator} of the spectral triple $(\A,\Hilbert,\Dirac)$.
\end{definition}

A spectral triple $(\A,\Hilbert,\Dirac)$ induces, in particular, a (possibly $\infty$-valued) pseudo-metric on the state space of the underlying C*-algebra $\A$, called the \emph{Connes' distance}, defined, between any two states $\varphi,\psi$ of $\A$, by
\begin{equation*}
  \Kantorovich{\Dirac}(\varphi,\psi) = \sup\left\{ |\varphi(a)-\psi(a)| : a \in \A_{\Dirac}, \opnorm{[\Dirac,a]}{}{\Hilbert} \leq 1 \right\} \text.
\end{equation*}
Connes proved in \cite{Connes89} that, when $\A = C(M)$ is the C*-algebra of $\C$-valued continuous over a connected, compact spin Riemannian manifold $M$, and when $\Dirac$ is the Dirac operator acting on the square integrable sections of the spinor bundle over $M$, the Connes' metric $\Kantorovich{\Dirac}$ restricts to the usual path metric over $M$ induced by its Riemannian metric. More generally, Connes' distance suggests a way to study the metric properties of noncommutative spaces, as it makes sense whether $\A$ is commutative or not.

\bigskip

Connes' distance, and the observation that a spectral triple contains some metric information, led to a new chapter in our story: the study of quantum metric spaces. Rieffel, starting from Connes' work, introduced a first, very general, notion of compact quantum metric space \cite{Rieffel98a,Rieffel99,Rieffel00} (the situation is more complex when dealing with the locally compact setting, see \cite{Latremoliere05b,Latremoliere12b}). It is helpful to understand this notion by observing that Connes' metric, in the case of a Riemannian manifold and its Dirac operator, is a special case of the \emph{\MongeKant}. In \cite{Kantorovich40,Kantorovich58}, Kantorovich defines a distance between any two regular Borel probability measures $\mu$ and $\nu$ over a compact metric space $(X,d)$, by setting:
\begin{equation*}
  \Kantorovich{\Lip}(\mu,\nu) = \left\{ \left| \int_X f\, d\mu - \int_X f\, d\nu \right| : f \in C(X), \Lip(f) \leq 1 \right \} \text,
\end{equation*}
where, for all $f \in C(X)$, we define the \emph{Lipschitz constant} $\Lip(f)$ of $f$ by:
\begin{equation*}
  \Lip(f) = \sup\left\{ \frac{|f(x)-f(y)|}{d(x,y)} : x,y \in X, x\neq y \right\} \text,
\end{equation*}
allowing $\infty$.

The {\MongeKant} $\Kantorovich{\Lip}$ enjoys many helpful properties. Among the most important of these properties, $\Kantorovich{\Lip}$ metrizes the weak* topology on the space of Borel regular probability measures over $X$ --- which makes it a useful tool, for instance, in probability theory. For our purpose, an easy but very relevant observation is that the restriction of $\Kantorovich{\Lip}$ to $X$ (identified as the space of Dirac measures) is indeed the original metric $d$. Thus, $\Lip$, which is a seminorm defined on some dense subspace of $C(X)$, encodes all the metric information given by the metric $d$ on $X$.

Therefore, it becomes natural, following the standard approach to noncommutative geometry, to package the key properties of a Lipschitz seminorms which do not refer to points or the original space, and still make sense if we replace $C(X)$ by some noncommutative algebra. What constitute a key property is not immediately clear, and in fact, is really seen in hindsight, as a property needed to develop an interesting theory of {\qcms} shared by many interesting examples. For our purpose, an interesting theory of {\qcms s} will be one where we can develop a notion of convergence, analogue to the Gromov-Hausdorff convergence \cite{Gromov} --- which was also a motivation for Rieffel's original work \cite{Rieffel00}. The following definition seems to have surfaced as the appropriate one to advance our project.

\begin{notation}
  If $\A$ is a unital C*-algebra, the space $\{a\in\A : a=a^\ast \}$ of the self-adjoint elements in $\A$ is denoted by $\sa{\A}$, and the state space of $\A$ is denoted by $\StateSpace(\A)$. If $a\in \A$, we then write
  \begin{equation*}
    \Re a = \frac{a+a^\ast}{2} \in \sa{\A} \text{ and } \Im a = \frac{a-a^\ast}{2i} \in \sa{\A} \text.
  \end{equation*}

  Moreover, for any normed vector space $E$, the norm of $E$ is denoted by $\norm{\cdot}{E}$ unless stated otherwise.
\end{notation}

\begin{notation}
  If $T : E\rightarrow F$ is a continuous linear operator from a normed vector space $E$ to a normed vector space $F$, then the norm of $T$ is denoted by $\opnorm{T}{E}{F}$; if $E=F$, then we simply write $\opnorm{T}{}{E}$.
\end{notation}

\begin{convention}\label{dom-convention}
  If $N$ is a seminorm defined on some subspace $\dom{N}$ of a normed vector space $E$, then we define $N(x) = \infty$ whenever $x\in E\setminus\dom{N}$. With this convention, $\dom{N} = \left\{ x \in E : N(x) < \infty \right\}$.
\end{convention}

\begin{definition}\label{qcms-def}
  An \emph{$(F,K)$-\qcms} $(\A,\Lip)$, where $F\geq 1$ and $K\geq 0$, is an ordered pair of a unital C*-algebra $\A$, and a seminorm $\Lip$ defined on a dense subspace $\dom{\Lip}$ of the space $\sa{\A}$ of self-adjoint elements of $\A$, such that:
  \begin{itemize}
  \item $\Lip(\unit_\A) = 0$,
  \item the \emph{\MongeKant} $\Kantorovich{\Lip}$, defined between any two states $\varphi,\psi \in \StateSpace(\A)$ of $\A$, by
    \begin{equation*}
      \Kantorovich{\Lip}(\varphi,\psi) = \sup\left\{ |\varphi(a)-\psi(a)| : a \in \dom{\Lip}, \Lip(a) \leq 1 \right\} \text,
    \end{equation*}
    is a metric on $\StateSpace(\A)$ which induces the weak* topology restricted to $\StateSpace(\A)$,
  \item the following \emph{$(F,K)$-Leibniz inequality} holds for all $a,b \in \dom{\Lip}$:
    \begin{equation*}
      \max\left\{ \Lip(\Re(ab)), \Lip(\Im(ab)) \right\} \leq F\left(\Lip(a)\norm{b}{\A} + \norm{a}{\A}\Lip(b)\right) + K \Lip(a) \Lip(b) \text,
    \end{equation*}
  \item $\left\{ a \in \dom{\Lip} : \Lip(a) \leq 1 \right\}$ is closed in $\sa{\A}$.
  \end{itemize}
\end{definition}

As an immediate consequence of the fact that $\Kantorovich{\Lip}$ is indeed a metric, with the notation of Definition (\ref{qcms-def}), we note that $\Lip(a) = 0$ implies $a\in \R\unit_\A$ (otherwise, there exists $a\notin\R\unit_\A$ such that $\Lip(a)=0$, and thus there exist two states $\varphi,\psi \in \StateSpace(\A)$ such that $\varphi(a)-\psi(a)\neq 0$; for all $t\in \R$, we also have $\Lip(ta) = 0$, yet $\lim_{t\rightarrow\infty}|\varphi(ta)-\psi(ta)| = \infty$ so $\Kantorovich{\Lip}(\varphi,\psi) = \infty$, which is not allowed). Consequently:
\begin{equation*}
  \left\{ a \in \dom{\Lip} : \Lip(a) = 0 \right\} = \R\unit_\A \text.
\end{equation*}
We also note that the $(F,K)$-Leibniz property, together with Convention (\ref{dom-convention}), implies that $\dom{\Lip}$ is a Jordan-Lie subalgebra of $\sa{\A}$.

We also note that the $(F,K)$-Leibniz relation in Definition (\ref{qcms-def}) is a special case of the more general definition of a {\qcms} in, for instance, \cite{Latremoliere15}, but this level of generality will suffice here.

The class of {\qcms s} can be made into a category, where $\pi:(\A,\Lip_\A) \rightarrow (\B,\Lip_\B)$ is a \emph{Lipschitz morphism} when $\pi:\A\rightarrow\B$ is a unital morphism such that $\pi(\dom{\Lip_\A}) \subseteq \dom{\Lip_\B}$ \cite{Rieffel00,Latremoliere16b}.

\bigskip

There are many examples of {\qcms s}, including quantum tori \cite{Rieffel98a,Rieffel02}, AF algebras \cite{Latremoliere15}, Podle{\'s} spheres \cite{Kaad18}, noncommutative solenoids \cite{Latremoliere16}, hyperbolic \cite{Ozawa05} and nilpotent \cite{Rieffel15b} group C*-algebras, among others. In some cases, such as for AF algebras in \cite{Latremoliere15d}, it is indeed helpful to relax the Leibniz inequality, as we did in Definition (\ref{qcms-def}) with the introduction of the real numbers $F$ and $K$.

\bigskip

Among all examples of {\qcms s}, we shall focus in this paper on those which arise from spectral triples, using Connes' distance. In general, Connes' distance may or may not meet the conditions needed to construct a {\qcms}, but many important examples do. We thus introduce the following definition.

\begin{definition}
  A spectral triple $(\A,\Hilbert,\Dirac)$ is \emph{metric} when, setting:
  \begin{equation*}
    \dom{\Lip} = \left\{ a \in \sa{\A} : a\cdot\dom{\Dirac}\subseteq\dom{\Dirac} \text{, } [\Dirac,a] \text{ is bounded } \right\}
  \end{equation*}
  and
  \begin{equation*}
    \forall a \in \dom{\Lip} \quad \Lip(a) = \opnorm{[\Dirac,a]}{}{\Hilbert} \text,
  \end{equation*}
  then $(\A,\Lip)$ is a {\qcms}. 
\end{definition}

In \cite[Proposition 1.10]{Latremoliere18g}, we prove that a spectral triple $(\A,\Hilbert,\Dirac)$ is metric if, and only if, its associated Connes' distance induces the weak* topology on $\StateSpace(\A)$.

Examples of metric spectral triples include, once again, quantum tori \cite{Rieffel98a,Latremoliere16b}, hyperbolic \cite{Ozawa05} and nilpotent \cite{Rieffel15b} group C*-algebras, certain C*-crossed-products \cite{Hawkins13}, certain fractals \cite{Lapidus08,Lapidus14}, Podle{\'s} spheres \cite{Kaad18}, curved quantum tori \cite{Latremoliere15c}, among others. Certain examples motivated  in physics can be found in \cite{Latremoliere21a}.

\bigskip

Our main contribution to the study of {\qcms s} is the introduction of the \emph{Gromov-Hausdorff propinquity}, a noncommutative analogue of the Gromov-Hausdorff distance \cite{Gromov,Gromov81}, which enables the study of convergence of quantum metric spaces --- for instance, the approximation of quantum tori by fuzzy tori \cite{Latremoliere13c}, of the two-sphere by matrix algebras \cite{Rieffel15}, or the continuity of the Effr{\"o}s-Shen AF algebras in their irrational parameters \cite{Latremoliere15d}. Of course, various notions of approximations are common in C*-algebra theory --- such as continuous fields, deformations, nuclearity, inductive limits. The propinquity opens the possibility to study limits and approximations of {\qcms s} in terms of an actual distance function, so that our notion of approximations are topological in nature. The propinquity is the basis for the metric studied in this paper, the \emph{spectral propinquity}, defined on the class of metric spectral triples.

\bigskip

Indeed, a metric spectral triple defines an object which encompasses both a {\qcms} and a special type of module over it. In the next chapter of our story \cite{Latremoliere16c,Latremoliere18d}, we extended the propinquity from {\qcms s} to certain modules over them. This class of objects is central to this work. More specifically, we introduce the notion of a metrical C*-correspondence. We begin by recalling what a C*-correspondence is.

\begin{definition}
  A right \emph{pre-Hilbert $C^\ast$-module} $\module{M}$ over a C*-algebra $\B$ is a right $\B$-module, together with a map $\inner{\cdot}{\cdot}{\module{M}} : \module{M}\times\module{M} \rightarrow \B$ such that:
  \begin{itemize}
  \item $\inner{\omega}{x\eta+y\zeta}{\module{M}} = x\inner{\omega}{\eta}{\module{M}} + y\inner{\omega}{\zeta}{\module{M}}$ for all $\omega,\eta,\zeta\in\module{M}$ and for all $x,y \in \C$,
  \item $\inner{\omega}{\eta\cdot b}{\module{M}} = \inner{\omega}{\eta}{\module{M}} b$ for all $b\in\B$ and $\omega,\eta \in \module{M}$,
  \item $\inner{\omega}{\eta}{\module{M}} = \left(\inner{\eta}{\omega}{\module{M}}\right)^\ast$ for all $\omega,\eta\in\module{M}$,
  \item $\inner{\omega}{\omega}{\module{M}}\geq 0$ (as a self-adjoint operator in $\B$) for all $\omega\in\module{M}$,
  \item for all $\omega\in\module{M}$, we have $\inner{\omega}{\omega}{\module{M}} = 0$ if, and only if, $\omega = 0$.
  \end{itemize}
\end{definition}

If $\module{M}$ is a right $\B$-Hilbert $C^\ast$-module, then the function
\begin{equation*}
  \omega\in\module{M} \mapsto \norm{\omega}{\module{M}} \coloneqq \sqrt{\norm{\inner{\omega}{\omega}{\module{M}}}{\B}}
\end{equation*}
is indeed a norm on $\module{M}$ \cite{Lance95}.

\begin{definition}
  A \emph{right $\B$-Hilbert $C^\ast$-module} $\module{M}$ over a C*-algebra $\B$ is a right pre-Hilbert $C^\ast$-module over $\B$ which is also complete with respect to the norm $\norm{\cdot}{\module{M}}$.
\end{definition}

In general, if $T$ is a $\B$-linear operator on a right Hilbert $\B$-module $\module{M}$, then $T$ is called \emph{adjoinable} when there exists a $\B$-linear operator $T^\ast$ on $\module{M}$ (called the adjoint of $T$) such that $\inner{T(\omega)}{\eta}{\module{M}} = \inner{\omega}{T^\ast(\eta)}{\module{M}}$ for all $\omega,\eta\in\module{M}$. Whenever it exists, the adjoint is unique. Notably, adjoinable operators are always continuous linear operators for the Hilbert $C^\ast$-module norm, and the set of all adjoinable operators of a Hilbert $C^\ast$-module is a C*-algebra.

\begin{definition}
  A \emph{$\A$-$\B$-C*-correspondence} $(\module{M},\A,\B)$ is given by a right $\B$-Hilbert $C^\ast$-module $\module{M}$ and a *-morphism from $\A$ to the C*-algebra of adjoinable operator of the Hilbert $\B$-module $\module{M}$.
\end{definition}
For our purpose, all our C*-algebras and all our *-morphisms will be unital. Our notion of a metrical C*-correspondence is now given as follows.

\begin{definition}[{\cite[Definition 2.2]{Latremoliere18g}}]\label{mcc-def}
  A \emph{(F,K,G,H)-metrical C*-correspondence} $(\module{M},\CDN,\A,\Lip,\B,\SLip)$, where $F,G \geq 1$, $H\geq 2$, and $K>0$, is given by two {$(F,K)$-\qcms s} $(\A,\Lip)$ and $(\B,\SLip)$, an $\A$-$\B$ C*-correspondence $(\module{M},\A,\B)$, and a norm $\CDN$ defined on a dense $\C$-subspace $\dom{\TDN}$ of $\module{M}$, such that
  \begin{enumerate}
  \item $\forall \omega\in\dom{\CDN} \quad \CDN(\omega)\geq\norm{\omega}{\module{M}}\coloneqq\sqrt{\norm{\inner{\omega}{\omega}{\module{M}}}{\B}}$,
  \item $\{ \omega\in\dom{\CDN} : \CDN(\omega)\leq 1\}$ is compact in $(\module{M},\norm{\cdot}{\module{M}})$,
  \item for all $a\in \dom{\Lip}$ and $\omega\in \dom{\TDN}$, the following \emph{modular Leibniz} inequality holds:
    \begin{equation*}
      \CDN(a\omega)\leq G(\norm{a}{\A}+\Lip(a))\CDN(\omega) \text,
    \end{equation*}
  \item for all $\omega,\eta\in\dom{\CDN}$, the following \emph{inner Leibniz} inequality holds:
    \begin{equation*}
      \SLip(\inner{\omega}{\eta}{\module{M}}) \leq H \CDN(\omega) \CDN(\eta) \text.
    \end{equation*}
  \end{enumerate}
\end{definition}

\begin{convention}\label{FKGH-convention}
  In this work, we \emph{fix} $F\geq 1$, $K \geq 0$, $H\geq 2$ and $G\geq 1$ all throughout the paper. \emph{All} {\qcms s} will be assumed to be in the class of $(F,K)$-{\qcms s} and all metrical C*-correspondences will be assumed to be in the class of $(F,K,G,H)$-metrical C*-correspondences, unless otherwise specified.
\end{convention}

In particular, we proved in \cite{Latremoliere18g} that metric spectral triples naturally give rise to {\gMVB s}.

\begin{theorem}[{\cite[Theorem 2.7]{Latremoliere18g}}]\label{mcc-thm}
  If $(\A,\Hilbert,\Dirac)$ is a metric spectral triple, if we set:
\begin{equation*}
  \forall \xi \in \dom{\Dirac} \quad \CDN_{\Dirac}(\xi) \coloneqq \norm{\xi}{\Hilbert} + \norm{\Dirac\xi}{\Hilbert} \text,
\end{equation*}
and if $\dom{\Lip_{\Dirac}} = \left\{ a\in\sa{\A} : a\dom{\Dirac}\subseteq\dom{\Dirac}\text{, }[\Dirac,a] \text{ is bounded } \right\}$ and
\begin{equation*}
  \forall a \in \dom{\Lip_{\Dirac}} \quad \Lip_{\Dirac}(a) \coloneqq \opnorm{[\Dirac,a]}{}{\Hilbert} \text,
\end{equation*}
then 
\begin{equation*}
  \mcc{\A}{\Hilbert}{\Dirac} \coloneqq \left(\Hilbert,\CDN_{\Dirac},\A,\Lip_{\Dirac},\C,0\right)
\end{equation*}
is a {\gMVB}(with $F = 1, K = 0, G = 1, H =2 $).
\end{theorem}

In fact, while canonical, $\mcc{\A}{\Hilbert}{\Dirac}$ is not the only means to associate a {\gMVB} to a metric spectral triple. One of the key point, as we shall see, in our construction of the spectral propinquity, is that it is based on such concepts as {\gMVB s} and actions on them, and thus, we certainly can choose a different way to define a distance between metric spectral triples by defining a different map from metrical spectral triples to {\gMVB s}. For instance, we note the following, in passing.

\begin{theorem}
  If $(\A,\Hilbert,\Dirac)$ is a metric spectral triple with $\Dirac$ injective, if $\Lip(a) \coloneqq \opnorm{[\Dirac,a]}{}{\Hilbert}$ for all $a\in\dom{\Lip} \coloneqq \{ a\in\sa{\A} : a\dom{\Dirac}\subseteq\dom{\Dirac},[\Dirac,a]\text{ bounded} \}$, and if
  \begin{equation*}
    \forall \xi \in \dom{\Dirac} \quad \CDN(\xi) \coloneqq \lambda^{-1} \norm{\Dirac\xi}{\Hilbert}
  \end{equation*}
  where $\lambda \coloneqq \min\{ |\mu| : \mu \in \spectrum{\Dirac} \}$, then $(\Hilbert,\CDN,\A,\Lip,\C,0)$ is a {\gMVB} with $F = 1, K = 0, H = 2, G = 1$.
\end{theorem}

\begin{remark}
  The set $\{ |\mu| : \mu \in \spectrum{\Dirac} \}$ indeed has a smallest element since $\Dirac$ has compact resolvent.
\end{remark}

\begin{proof}
  By assumption, $(\A,\Lip)$ is a {\qcms} with $F = 1$, $K=0$.
  
  By construction, the spectrum of $|\lambda|^{-1}\Dirac$ only contains eigenvalues in $(-\infty,1]\cup[1,\infty)$, and thus $\CDN(\xi) \geq \norm{\xi}{\Hilbert}$ for all $\xi \in \dom{\Dirac}$. Since $\Dirac$ has compact resolvent, we also note that
  \begin{equation*}
    \left\{ \xi \in \dom{\Dirac} : \CDN(\xi)\leq 1 \right\}
    = \left\{ \Dirac^{-1} \eta : \eta\in\Hilbert,\norm{\eta}{\Hilbert}\leq |\lambda| \right\}
  \end{equation*}
  is totally bounded. As proven in \cite{Rieffel00}, this set is also closed, as $\Dirac$ is self-adjoint. Thus, as $\Hilbert$ is complete, the closed unit ball of $\CDN$ is compact.

  Of course, $0 \leq \CDN(x \xi) = |x|\CDN(\xi)$ for all $x\in\C$, $\xi\in\Hilbert$, so the inner Leibniz inequality holds for $G = 1$.

  Last, for all $a\in \dom{\Lip}$, we note that
  \begin{align*}
    \CDN(a\xi)
    &= \norm{\Dirac a \xi}{\Hilbert} \\
    &= \norm{[\Dirac,a]\xi + a\Dirac\xi}{\Hilbert} \\
    &\leq \Lip(a)\norm{\xi}{\Hilbert} + \norm{a}{\A} \norm{\Dirac\xi}{\Hilbert} \\
    &\leq (\Lip(a) + \norm{a}{\A}) \CDN(\xi) \text.
  \end{align*}

  Thus, $(\Hilbert,\CDN,\A,\Lip,\C,0)$ is a {\gMVB}.
\end{proof}

\begin{remark}
  As the reader may check, our proofs in this work do not depend on the exact form of the D-norms for {\gMVB s} associated with metric spectral triples. What really matters is that \emph{the domain of the D-norms is the domain of the Dirac operator}. 
\end{remark}

\bigskip

We defined in \cite{Latremoliere16c,Latremoliere18g} an analogue of the Gromov-Hausdorff distance on the class of {\gMVB s}, which, in particular, induces a first pseudo-distance on metric spectral triples, via Theorem (\ref{mcc-thm}). We begin the presentation of this metric with a few key concepts.

\medskip

A morphism between metrical C*-correspondences is given by a triple of linear maps which satisfy a long but very natural list of algebraic and analytic properties.

\begin{definition}
  For each $j \in \{1,2\}$, let
  \begin{equation*}
    \mathds{M}_j \coloneqq \left( \module{M}_j,\CDN_j,\A_j,\Lip_j,\B_j,\SLip_j \right)
  \end{equation*}
  be a metrical C*-correspondence. 

  A \emph{Lipschitz morphism} $(\Pi,\pi,\theta)$ from $\mathds{M}_1$ to $\mathds{M}_2$ is a given by:
  \begin{enumerate}
  \item a continuous $\C$-linear map $\Pi : \module{M}_1 \rightarrow\module{M}_2$,
  \item a unital *-morphism $\pi : \A_1 \rightarrow \A_2$,
  \item a unital *-morphism $\theta: \B_1 \rightarrow \B_2$,
  \end{enumerate}
  such that
  \begin{enumerate}
  \item $\forall a \in \A \quad \forall \omega \in \module{M}_1 \quad \Pi(a\omega) = \pi(a)\Pi(\omega)$,
  \item $\forall b \in \B \quad \forall \omega \in \module{M}_2 \quad \Pi(\omega\cdot b) = \Pi(\omega)\theta(b)$,
  \item $\forall \omega,\eta\in \module{M}_1 \quad \theta(\inner{\omega}{\eta}{\module{M}_1}) = \inner{\Pi(\omega)}{\Pi(\eta)}{\module{M}_2}$,
  \item $\pi(\dom{\Lip_1}) \subseteq \dom{\Lip_2}$,
  \item $\theta(\dom{\SLip_1}) \subseteq \dom{\SLip_2}$,
  \item $\Pi(\dom{\CDN_1})\subseteq\dom{\CDN_2}$.
  \end{enumerate}
\end{definition}

\begin{definition}
  For each $j \in \{1,2\}$, let
  \begin{equation*}
    \mathds{M}_j = \left( \module{M}_j,\CDN_j,\A_j,\Lip_j,\B_j,\SLip_j \right)
  \end{equation*}
  be a metrical C*-correspondence.

  A \emph{modular quantum isometry} $(\Pi,\pi,\theta)$ from $\mathds{M}_1$ to $\mathds{M}_2$ is a Lipschitz morphism $(\Pi,\pi,\theta) : \mathds{M}_1 \rightarrow \mathds{M}_2$ such that
  \begin{enumerate}
  \item $\forall a \in \dom{\Lip_2} \quad \Lip_2(a) = \inf\left\{ \Lip_1(d) : d\in\dom{\Lip_1},\pi(d) = a \right\}$,
  \item $\forall b \in \dom{\SLip_2} \quad \SLip_2(b) = \inf\left\{ \SLip_1(d) : d\in\dom{\SLip_1}, \theta(d) = b \right\}$,
  \item $\forall \omega \in \dom{\CDN_2} \quad \CDN_2(\omega)=\inf\left\{ \CDN_1(\eta) : \eta\in\dom{\CDN_1}, \theta(\eta) = \omega \right\}$.
  \end{enumerate}
\end{definition}

Since metrical C*-correspondences morphisms $(\Pi,\pi,\theta)$ consists, by definition, of three isometric maps $\Pi$, $\pi$ and $\theta$, they all have closed ranges in their respective codomains; when $(\Pi,\pi,\theta)$ is a quantum isometry, our definition implies that these ranges contain certain dense subspaces, and thus, $\Pi$, $\pi$ and $\theta$ are in fact surjective by definition.

The definition of a distance between {\gMVB s}, called the \emph{metrical propinquity}, relies on a notion of isometric embedding called a \emph{tunnel}, and defined as follows.

\begin{definition}[{\cite[Definition 2.19]{Latremoliere18g}}]\label{tunnel-def}
  Let $\mathds{M}_1$ and $\mathds{M}_2$ be two metrical C*-correspondences. A \emph{(metrical) tunnel} $\tau = (\mathds{J},\Pi_1,\Pi_2)$ is a triple given by a metrical C*-correspondence $\mathds{J}$, and for each $j\in\{1,2\}$, a modular quantum isometry $\Pi_j : \mathds{J}\mapsto \mathds{M}_j$.
\end{definition}

\begin{notation}
  We also may use the notation
  \begin{equation*}
    \tunnel{\tau}{\mathds{M}_1}{\Pi_1}{\mathds{P}}{\Pi_2}{\mathds{M}_2}
  \end{equation*}
  to define a tunnel $\tau \coloneqq (\mathds{P},\Pi_1,\Pi_2)$ from a {\gMVB} $\mathds{M}_1$ to a {\gMVB} $\mathds{M}_2$.
\end{notation}

\begin{remark}
  It is important to note that our tunnels involve $(F,K,G,H)$-C*-metrical correspondences only (as per Convention (\ref{FKGH-convention})). We will dispense with calling our tunnels $(F,K,G,H)$-tunnels, to keep our notation simple, but it should be stressed that \emph{fixing} $(F,K,G,H)$ and staying within the class of $(F,K,G,H)$-C*-metrical correspondences is important to obtain a metric from tunnels.
\end{remark}

Tunnels can be ``reversed'', simply by exchanging their domain and codomains, as follows.
\begin{notation}
  If $\tunnel{\tau}{\mathds{M}_1}{\Pi_1}{\mathds{P}}{\Pi_2}{\mathds{M}_2}$ is a tunnel from $\mathds{M}_1$ to $\mathds{M}_2$, then
  \begin{equation*}
    \tunnel{\tau^{-1}}{\mathds{M}_2}{\Pi_2}{\mathds{P}}{\Pi_1}{\mathds{M}_1}
  \end{equation*}
  is a tunnel from $\mathds{M}_2$ to $\mathds{M}_1$.
\end{notation}

To each tunnel, we can associate a number, which quantifies how far two {\gMVB s} are from each others, as quantified by this particular tunnel. This number is computed using the Hausdorff distance induced by the {\MongeKant} on the hyperspace of all closed subsets of the state space of a {\qcms}.

\begin{definition}
  Let $d$ be a distance function over a set $E$. For any nonempty subset $B \subseteq E$ of $E$, and for any $x\in E$, we define:
  \begin{equation*}
    d(x,B) \coloneqq \inf\{ d(x,y) : y \in B \}\text.
  \end{equation*}
  For any two nonempty subsets $A,B \subseteq E$ of $E$, we then define
  \begin{equation*}
    \LHaus{d}(A\rightarrow B) \coloneqq \sup_{x \in A} d(x,B)\text.
  \end{equation*}
  
  The \emph{Hausdorff distance} \cite{Hausdorff} between any two closed nonempty subsets $A$ and $B$ of a metric space $(X,d)$ is defined by
  \begin{equation*}
    \Haus{d}(A,B) \coloneqq \max\{ \LHaus{d}(A\rightarrow B), \LHaus{d}(B\rightarrow A) \} \text.
  \end{equation*}
\end{definition}

\begin{notation}
  If $X$ is in fact a normed vector space, and $d$ is the distance induced by the norm $\norm{\cdot}{X}$ of $X$, then we may write $\Haus{\norm{\cdot}{X}}$, or even $\Haus{X}$, for $\Haus{d}$.
\end{notation}

\begin{definition}[{\cite[Definition 1.17, Definition 2.21]{Latremoliere18g}}]
  Let $\mathds{M}_j = (\module{M}_j,\CDN_j,\A_j,\Lip_j,\B_j,\SLip_j)$ be a metrical C*-correspondence, for each $j \in \{1,2\}$. Let $\tau = (\mathds{P},(\Pi_1,\pi_1,\theta_1),(\Pi_2,\pi_2,\theta_2))$ be a metrical tunnel form $\mathds{M}_1$ to $\mathds{M}_2$, with $\mathds{P} = (\module{P},\TDN,\D,\Lip_\D,\alg{E},\Lip_{\alg{E}})$.
  
  The \emph{extent} $\tunnelextent{\tau}$ of a metrical tunnel $\tau$ is
  \begin{multline*}
    \tunnelextent{\tau} \coloneqq
    \max_{j\in\{1,2\}} \max\Big\{ \Haus{\Kantorovich{\Lip_\D}}\left( \left\{ \varphi\circ\pi_j : \varphi \in \StateSpace(\A_j)\right\}, \StateSpace(\D) \right), \\
    \Haus{\Kantorovich{\SLip_{\alg{E}}}}\left( \left\{\psi\circ\theta_j : \psi\in\StateSpace(\B_j)\right\}, \StateSpace(\alg{E})  \right)\Big\}
  \end{multline*}
\end{definition}

We define our distance between {\gMVB s} in \cite{Latremoliere18d} following the model proposed by Edwards \cite{Edwards75} and Gromov \cite{Gromov81}, as follows (we refer to \cite{Latremoliere18g} as we follow the notations of that paper more closely in the present work).

\begin{definition}[{\cite[Definition 2.22]{Latremoliere18g}}]\label{metpropinquity-def}
  The \emph{metrical propinquity} $\dmetpropinquity{}(\mathds{M}_1,\mathds{M}_2)$ between any two metrical C*-correspondences $\mathds{M}_1$ and $\mathds{M}_2$ is given by the real number:
  \begin{equation*}
    \dmetpropinquity{}(\mathds{M}_1,\mathds{M}_2) \coloneqq \inf\left\{ \tunnelextent{\tau} : \text{$\tau$ is a tunnel from $\mathds{M}_1$ to $\mathds{M}_2$ } \right\} \text.
  \end{equation*}
\end{definition}

We prove that the metrical propinquity enjoys some welcomed properties. Some of these properties were established first for a stronger variant of the metrical propinquity in \cite{Latremoliere16c}.

\begin{theorem}[{\cite[Theorem 4.9]{Latremoliere18d}}]
  The metrical propinquity $\dmetpropinquity{}$ is a complete metric, up to a full quantum isometry, on the class of metrical C*-correspondences.
\end{theorem}

\begin{remark}
  The metrical propinquity, as defined above, should properly be called the $(F,K,G,H)$-metrical propinquity, denoted by $\dmetpropinquity{F,K,G,H}$, as it is defined on the class of $(F,K,G,H)$-metrical C*-correspondences (i.e., C*-correspondences with fixed Leibniz properties). However, following our Convention (\ref{FKGH-convention}), we will omit this index and terminology, with the understanding that we decided to restrict ourselves to this class of metrical C*-correspondences from the beginning. 
\end{remark}

\begin{remark}
  Since {\gMVB s} defined via metric spectral triples, using Theorem (\ref{mcc-thm}), are all of the same type ($F=1,K=0,G=1,H=2$), it may appear that, for the present paper, there is no need for the larger generality allowed by introducing these numbers. However, this is not the case. While the {\gMVB s} induced by metric spectral triples are all of the same type, tunnels may well involve {\gMVB s} with relaxed Leibniz properties.
\end{remark}

\begin{remark}
  The Gromov-Hausdorff propinquity $\dpropinquity{}$ defined in \cite{Latremoliere13,Latremoliere13b,Latremoliere14,Latremoliere15} on the class of {\qcms s} is a complete metric, up to full quantum isometry, on the class of {\qcms s}, whose restriction to the class of classical compact metric spaces $(X,d)$ (via the embedding which sends $(X,d)$ to $(C(X),\Lip)$ with $\Lip$ the Lipschitz seminorm induced by $d$) metrizes the same topology as the Gromov-Hausdorff distance.

\medskip
  
If $(\A,\Lip)$ is a {\qcms}, we can regard $\A$ as a Hilbert $C^\ast$-module over itself, with $\inner{a}{b}{\A} = a^\ast b$ for all $a,b \in \A$, and we can set
\begin{equation*}
  \forall a \in \A \quad \CDN(a) \coloneqq \max\left\{ \Lip(\Re a), \Lip(\Im a) \right\} \text,
\end{equation*}
thus associating to any {\qcms} the metrical C*-correspondence $(\A,\CDN,\C,0,\A,\Lip)$. The injection thus constructed from the class of {\qcms s} to the class of metrical C*-correspondences, is a homeomorphism onto its range, when the class of {\qcms s} is endowed with the dual propinquity \cite{Latremoliere13b,Latremoliere14}, and the class of metrical C*-correspondences is metrized with the metrical propinquity. In the present work, however, we focus on the class of metrical C*-correspondences.
\end{remark}

\bigskip

If we apply the metrical propinquity to the {\gMVB s} associated with metric spectral triples via Theorem (\ref{mcc-thm}), then we obtain a pseudo-metric on spectral triples. This metric, informally, captures the metric information of the spectral triple. We note, in passing, that tunnels are not required to be constructed using spectral triples --- in fact, this flexibility is very helpful. We thus have some preliminary notion of convergence for spectral triples.

\bigskip

However, distance zero between spectral triples is weaker than what we want. Indeed, the natural equivalence between metric spectral triples is given as follows.
\begin{definition}\label{unitary-eq-def}
  Two spectral triples $(\A,\Hilbert,\Dirac)$ and $(\B,\mathscr{J},\slashed{\nabla})$ are \emph{unitarily equivalent} when there exists a unitary $U : \Hilbert \rightarrow \mathscr{J}$, a *-isomorphism $\pi : \A \rightarrow \B$ such that $U\dom{\Dirac} = \dom{\slashed{\nabla}}$ and
  \begin{equation*}
    U \Dirac U^\ast = \slashed{\nabla} \text{ over $\dom{\slashed{\nabla}}$, and } \forall a \in \A \quad \pi(a) = U a U^\ast \text,
  \end{equation*}
  having identified $\A$ and $\B$ as C*-algebras of operators on $\Hilbert$ and $\mathscr{J}$, respectively, via their implicit *-representations.
\end{definition}

We proved in \cite{Latremoliere18g} that, if $(\A,\Hilbert,\Dirac)$ and $(\B,\mathscr{J},\slashed{\nabla})$ are two unitarily equivalent metric spectral triples, then, using the notations of Definition (\ref{unitary-eq-def}), the map $\pi$ is, in fact, a full quantum isometry, when $\A$ and $\B$ are endowed with the quantum metrics induced by $\Dirac$ and $\slashed{\nabla}$, respectively.

\bigskip

In order to strengthen the metrical propinquity, we add one more ingredient. In yet another chapter of our story \cite{Latremoliere18b,Latremoliere18c,Latremoliere18g}, we introduced covariant versions of the propinquity. Now, if $(\A,\Hilbert,\Dirac)$ is a spectral triple, then it induces a natural quantum dynamics, i.e. an action of $\R$ on $\Hilbert$  by unitaries, defined for each $t \in \R$ by $U^t \coloneqq \exp(it\Dirac)$.

The \emph{spectral propinquity} is thus the covariant version of the metrical propinquity, applied to the {\gMVB s} induced by metric spectral triples \emph{and} the associated unitary actions, restricted to the monoid $[0,\infty)$. We will present it in details in our next section. Convergence, in the sense of the spectral propinquity, is the main matter of study of this paper. The spectral propinquity is, indeed, a distance on the class of metric spectral triples, up to unitary equivalence; we also proved some nontrivial examples of convergence for this metric \cite{Latremoliere20a,Latremoliere21a}. The covariant version of the propinquity is best explained by introducing first a useful notion of convergence of operators over {\gMVB s}, itself of prime interest for this work, and new. This is the first matter which we will address.

\section{Convergence for family of Operators on C*-correspondences}

In this section, we define a distance, up to unitary equivalence, between families of operators on Hilbert spaces. Our purpose with this distance is to formalize the convergence of the bounded functional calculus for spectral triples under convergence for the spectral propinquity, which is the core result of this paper. While, in this work, we will work with families of operators on Hilbert spaces associated with spectral triples, the fact that we use metrical tunnel means that the computation of the various distances will take place within the framework of metrical C*-correspondences, which we thus employ in this section.

\medskip

Our first task in extending the construction of the propinquity to families of operators on Hilbert $C^\ast$-modules, including actions of monoids and groups as needed for the spectral propinquity, is to extend the {\MongeKant} of a {\qcms} to the dual of a Hilbert $C^\ast$-module with a D-norm.

\begin{definition}[{\cite[Notation 3.8]{Latremoliere18g}}]\label{Kantorovich-mod-def}
  If $(\module{M},\TDN,\A,\Lip_\A,\B,\Lip_\B)$ is a metrical C*-correspondence, then for any $\C$-valued continuous linear functional $\varphi,\psi \in \module{M}^\ast$ (where $\module{M}^\ast$ is the $\C$-dual of $\module{M}$), we set:
  \begin{equation*}
    \Kantorovich{\TDN}(\varphi,\psi) \coloneqq \sup\left\{ |\varphi(\xi)-\psi(\xi)| : \xi \in \dom{\TDN}, \TDN(\xi)\leq 1 \right\} \text.
  \end{equation*}
\end{definition}

Since the unit ball of the D-norm $\TDN$ is compact, a standard argument shows that  $\Kantorovich{\TDN}$ induces the weak* topology on bounded subsets of $\module{M}^\ast$.

\begin{remark}
  The metric of Definition (\ref{Kantorovich-mod-def}) is denoted by $\KantorovichAlt{\TDN}$ in \cite{Latremoliere18g}.
\end{remark}

\medskip

We then extend, in the simplest way, the {\MongeKant} of Definition (\ref{Kantorovich-mod-def}), to a distance between families of continuous linear functionals indexed by a fixed set. We only take the distance between families of functionals indexed by the same set.

\begin{definition}\label{Kantorovich-family-def}
  Let $(\module{M},\TDN,\A,\Lip_\A,\B,\Lip_\B)$ be a metrical C*-correspondence. Let $J$ be a nonempty set. For any two families $(\varphi_j)_{j \in J}, (\psi_j)_{j\in J} \in \left(\module{M}^\ast\right)^J$ of continuous $\C$-linear functionals of $\mathds{M}$, we set:
  \begin{equation*}
    \KantorovichJ{\TDN}((\varphi_j)_{j\in J},(\psi_j)_{j\in J}) \coloneqq \sup \left\{  \Kantorovich{\TDN}(\varphi_j,\psi_j) : j \in J \right\} \text.
  \end{equation*}
\end{definition}

It is immediate to check that $\KantorovichJ{\TDN}$ induces a metric on $\left(\module{M}^\ast\right)^J$ for any fixed nonempty set $J$. Its topology is stronger than the product of the weak* topologies, and equal to it when $J$ is finite, though this will not be of prime concern here.

\medskip

Now, if we want to assign a distance between any two families $(a_j)_{j\in J}$ and $(b_j)_{j\in J}$ of adjoinable $\B$-linear operators on a right $\B$-Hilbert $C^\ast$-module $\module{M}$, then a natural choice would be $\sup_{j \in J} \opnorm{a_j-b_j}{}{\module{M}}$; however, this choice does not generalize well if we consider, instead, families of operators on different modules. Inspired from our previous work, another idea opens up when working with families of operators on a {\gMVB} $\mathds{M}$, with $D$-norm $\TDN$. We begin by identifying a subset of the dual $\module{M}^\ast$ which plays a role akin to the state space of a C*-algebra for our purpose. For the present work, we will use the notion of \emph{pseudo-states} for a metrical C*-correspondence, as a natural generalization of states of a {\qcms}. As seen in \cite[Proposition 3.11]{Latremoliere18g}, the following set is weak* compact (though not convex), and the weak* topology is metrized by the {\MongeKant} from Definition (\ref{Kantorovich-mod-def}).

\begin{definition}[{\cite[Notation 3.9]{Latremoliere18g}}]\label{pseudo-state-def}
  If $\mathds{M} \coloneqq (\module{M},\TDN,\A,\Lip_\A,\B,\Lip_\B)$ is a {\gMVB}, then a continuous linear functional $\varphi \in \module{M}^\ast$ is a \emph{pseudo-state} of $\mathds{M}$ when there exist $\mu \in \StateSpace(\B)$ and $\omega \in \module{M}$ with $\TDN(\omega)\leq 1$ such that $\varphi$ is given by:
  \begin{equation*}
    \varphi :\xi \in \module{M} \longmapsto \mu\left(\inner{\omega}{\xi}{\module{M}}\right) \text.
  \end{equation*}

  The set of all pseudo-states of $\mathds{M}$ is denoted by $\ModStateSpace(\mathds{M})$.
\end{definition}

\begin{remark}
  The set of pseudo-states in Definition (\ref{pseudo-state-def}) is not convex in general.
\end{remark}

In this setup, given two families $(a_j)_{j\in J}$ and $(b_j)_{j\in J}$ of adjoinable $\B$-linear operators on $\module{M}$, we can take the Hausdorff distance between the sets $\{ (\varphi\circ a_j)_{j\in J} : \varphi \in \ModStateSpace(\mathds{M}) \}$ and $\{ (\psi\circ b_j)_{j \in J} : \psi\in\ModStateSpace(\mathds{M}) \}$ for the distance $\KantorovichJ{\TDN}$ of Definition (\ref{Kantorovich-family-def}). The benefit of this approach is that it generalizes as follows to families of operators acting on different {\gMVB s}.

\begin{notation}
  If $\module{M}$ is a right $\B$-Hilbert $C^\ast$-module over a C*-algebra $\B$, then the C*-algebra of all adjoinable $\B$-linear continuous endomorphisms of $\module{M}$ is denoted by $\alg{L}_\B(\module{M})$, or even $\alg{L}(\module{M})$ when no confusion may arise. If $\mathds{A} = (\module{M},\A,\Lip_\A,\B,\Lip_\B)$ is an $\A$-$\B$ metrical C*-correspondence, then $\alg{L}_\B(\module{M})$ is denoted by $\alg{L}(\mathds{A})$. The set $\alg{L}(\mathds{A})$ only depends on the $\B$-Hilbert $C^\ast$-module structure, and not on $\A$; by definition, there is a unital *-morphism from $\A$ to $\alg{L}(\mathds{A})$.
\end{notation}

\begin{definition}\label{separation-def}
  Let $\mathds{A}$ and $\mathds{B}$ be two {\gMVB s}. Let
  \begin{equation*}
    \tunnel{\tau}{\mathds{A}}{(\Pi_{\mathds{A}},\pi_{\mathds{A}},\theta_{\mathds{A}})}{\mathds{P}}{(\Pi_{\mathds{B}},\pi_{\mathds{B}},\theta_{\mathds{B}})}{\mathds{B}}
  \end{equation*}
  be a metrical tunnel from $\mathds{A}$ to $\mathds{B}$. Let $\TDN$ be the D-norm of the metrical C*-correspondence $\mathds{P}$. 
  
  Let $J$ be a nonempty set. If $A\coloneqq (a_j)_{j \in J}$ is a family of operators in $\alg{L}(\mathds{A})$, and $B\coloneqq (b_j)_{j \in J}$ is a family of operators in $\alg{L}(\mathds{B})$, then we define the \emph{separation} of $A$ and $B$ according to $\tau$ by: 
  \begin{multline*}
    \tunnelsep{\tau}{A, B} \coloneqq \Haus{\KantorovichJ{\TDN}}\Big(\left\{ (\varphi\circ a_j\circ\Pi_{\mathds{A}})_{j \in J} : \varphi \in \ModStateSpace(\mathds{A}) \right\}, \\ \left\{ (\psi\circ b_j\circ\Pi_{\mathds{B}})_{j \in J} : \psi \in \ModStateSpace(\mathds{B}) \right\} \Big) \text.
  \end{multline*}
  The \emph{dispersion} of $A$ and $B$ according to $\tau$ is
  \begin{equation*}
    \tunneldispersion{\tau}{A,B} \coloneqq \max\{\tunnelextent{\tau},\tunnelsep{\tau}{A,B} \}\text.
  \end{equation*}
\end{definition}

We therefore obtain a natural way to discuss the convergence of families of adjoinable operators on metrical C*-correspondences, in the spirit of the Gromov-Hausdorff distance and the propinquity.

\begin{definition}\label{oppropinquity-def}
  Let $\mathds{A}$ and $\mathds{B}$ be two {\gMVB s}. Let $J$ be a nonempty set. If $A\coloneqq (a_j)_{j\in J}$ is a family of operators in $\alg{L}(\mathds{A})$, and $B\coloneqq (b_j)_{j\in J}$ is a family of operators in $\alg{L}(\mathds{B})$, then we define the \emph{operational propinquity} between these families as:
  \begin{equation*}
    \oppropinquity{}(A,B) \coloneqq \inf\left\{ \tunneldispersion{\tau}{A,B} : \tau\text{ is a metrical tunnel from $\mathds{A}$ to $\mathds{B}$}  \right\} \text,
  \end{equation*}
  with $\tunneldispersion{\cdot}{\cdot}$ defined in Definition (\ref{separation-def}).
\end{definition}

The operational propinquity is certainly a pseudo-metric on families of operators indexed by a fixed index set, and under mild conditions, it is in fact a metric up to a natural equivalence relation, as seen in the following theorem, whose proof is a direct application of the techniques in \cite{Latremoliere18g}. For these techniques to give us an actual metric, we require that all of our families include the identity operator, always found at the same index. This  implies that we can actually find the restriction to the unit balls of the D-norms of the graph of all the operators in our families  --- thus, the coincidence property in the following theorem follows from the graphs of various operators having distance zero for an appropriate notion of the Gromov-Hausdorff distance.

\begin{theorem}
  The operational propinquity $\oppropinquity{}$ is a pseudo-metric on families of operators indexed by the same set. Moreover, if we fix an index set $J$ and we fix some $j_0 \in J$, and if we set $\mathcal{F}(J)$ to be the class:
  \begin{equation*}
    \mathcal{F}(J) \coloneqq \left\{ (a_j)_{j \in J} \in \alg{L}(\mathds{A})^J : \mathds{A} \text{ is a \gMVB, } a_{j_0} = \mathrm{id}_\A \right\}
  \end{equation*}
  then the restriction of $\oppropinquity{}$ to $\mathcal{F}(J)$ is a metric up to the following equivalence relation. For any {\gMVB s} $\mathds{A}$ and $\mathds{B}$, and for any family $A\coloneqq (a_j)_{j\in J}$ of operators in $\alg{L}(\mathds{A})$, and any family $B\coloneqq (b_j)_{j \in J}$ of operators in $\alg{L}(\mathds{B})$, both families $A$ and $B$ being indexed by $J$, we have $\oppropinquity{}(A,B) = 0$ if, and only if, there exists a full modular quantum isometry $(\Pi,\pi,\theta)$ from $\mathds{A}$ to $\mathds{B}$ such that, for all $j \in J$, we have $\Pi\circ a_j = b_j \circ \Pi$.
\end{theorem}

\begin{proof}
  The proof follows exactly \cite[Theorem 3.23]{Latremoliere18b}, with the simplification that we only need the same indices.
\end{proof}

Our focus will be on metrical C*-correspondences constructed from metric spectral triples. If $(\A_1,\Hilbert_1,\Dirac
_1)$ is a metric spectral triple, we recall from Theorem (\ref{mcc-thm}) that the associated metrical C*-correspondence is defined as:
\begin{equation*}
  \mcc{\A_1}{\Hilbert_1}{\Dirac_1} \coloneqq \left(\Hilbert_1,\CDN_1,\A_1,\Lip_1,\C,0\right)
\end{equation*}
where $\Lip_1(a) = \opnorm{[\Dirac_1,a]}{}{\Hilbert_1}$ for all $a\in \sa{\A_1}$ such that $a\,\dom{\Dirac_1}\subseteq\dom{\Dirac_1}$ and $[\Dirac_1,a]$ is bounded, and $\CDN_1$ is the graph norm $\xi\in\dom{\Dirac_1}\mapsto\CDN_1(\xi)\coloneqq \norm{\xi}{\Hilbert_1}+\norm{\Dirac_1\xi}{\Hilbert_1}$ of $\Dirac_1$.

We then see that $\ModStateSpace(\mcc{\A_1}{\Hilbert_1}{\Dirac_1})$ is given by definition as the set of maps of the form $\varphi(\inner{\xi}{\cdot}{\Hilbert_1})$ where $\varphi$ is a state of $\C$ --- so it is is the identity on $\C$ --- and $\xi \in \dom{\Dirac_1}$ with $\CDN_{1}(\xi) = \norm{\xi}{\Hilbert} + \norm{\Dirac_1\xi}{\Hilbert_1} \leq 1$. Let now $(\A_2,\Hilbert_2,\Dirac_2)$ be another metric spectral triple, and set $\CDN_2$ to be the graph norm of $\Dirac_2$. We will work with the following convention throughout this paper.
\begin{convention}
  If $(\A_1,\Hilbert_1,\Dirac_1)$ and $(\A_2,\Hilbert_2,\Dirac_2)$ are two metric spectral triples, a \emph{tunnel} from $(\A_1,\Hilbert_1,\Dirac_1)$ to $(\A_2,\Hilbert_2,\Dirac_2)$ is a metrical tunnel from $\mcc{\A_1}{\Hilbert_1}{\Dirac_1}$ to $\mcc{\A_2}{\Hilbert_2}{\Dirac_2}$. We will use the notation:
  \begin{equation*}
    \tunnel{\tau}{(\A_1,\Hilbert_1,\Dirac_1)}{(\Pi_1,\pi_1,\theta_1)}{\mathds{P}}{(\Pi_2,\pi_2,\theta_2)}{(\A_2,\Hilbert_2,\Dirac_2)}
  \end{equation*}
  for a tunnel $\tau$ between the metric spectral triples $(\A_1,\Hilbert_1,\Dirac_1)$ and $(\A_2,\Hilbert_2,\Dirac_2)$; importantly, in this notation, $\mathds{P}$ is a metrical C*-correspondence which may well not come from a spectral triple, and the morphisms $(\Pi_1,\pi_1,\theta_1)$ and $(\Pi_2,\pi_2,\theta_2)$ are quantum isometries from $\mathds{P}$, respectively, onto $\mcc{\A_1}{\Hilbert_1}{\Dirac_1}$ and $\mcc{\A_2}{\Hilbert_2}{\Dirac_2}$.
\end{convention}

Let $A\coloneqq (a_j)_{j\in J}$ be a family of bounded operators on $\Hilbert_1$, and let $B\coloneqq (b_j)_{j \in J}$ be a family of bounded operators on $\Hilbert_2$. If
\begin{equation*}
  \tunnel{\tau}{(\A_1,\Hilbert_1,\Dirac_1)}{(\Pi_1,\pi_1,\theta_1)}{\mathds{P}}{(\Pi_2,\pi_2,\theta_2)}{(\A_2,\Hilbert_2,\Dirac_2)}
  \end{equation*}
  is some tunnel from $(\A_1,\Hilbert_1,\Dirac_1)$ to $(\A_2,\Hilbert_2,\Dirac_2)$, where the D-norm of $\mathds{P}$ is denoted by $\TDN$, then the separation between $A$ and $B$ becomes:
\begin{multline*}
  \tunnelsep{\tau}{A,B} = \Haus{\KantorovichJ{\TDN}}\Big(\left\{ (\inner{a_j^\ast\xi}{\Pi_1(\cdot)}{})_{j \in J} : \xi \in \dom{\CDN_{1}}, \CDN_{1}(\xi)\leq 1 \right\}, \\ \left\{ (\inner{b_j^\ast\xi}{\Pi_2(\cdot)}{})_{j \in J} : \xi \in \dom{\CDN_{2}}, \CDN_{2}(\xi)\leq 1 \right\} \Big) \text.
\end{multline*}

This expression can be unwound to give that $\tunnelsep{\tau}{A,B}$ is the maximum of the two numbers computed as follows, for $s=1,k=2$ and $s=2,k=1$:
\begin{equation*}
  \sup_{\substack{\xi \in \dom{\Dirac_s} \\ \CDN_s(\xi)\leq 1}} \inf_{\substack{\eta\in\dom{\Dirac_k} \\ \CDN_k(\eta)\leq 1}}\sup_{j\in J}\sup_{\TDN(\omega)\leq 1}\left|\inner{a_j^\ast\xi}{\Pi_s(\omega)}{\Hilbert_s} - \inner{b_j^\ast\eta}{\Pi_k(\omega)}{\Hilbert_k}\right| \text.
\end{equation*}

\bigskip

The spectral propinquity is constructed in \cite{Latremoliere18g} as a distance between two families of unitaries obtained as the one-parameter groups whose generators are the Dirac operators of spectral triples, but the construction is a little different from the operational propinquity because the index sets have additional structures --- metric and algebraic, in the form of a proper metric monoid of indices, which we wish to account for. In fact, the spectral propinquity is defined in \cite{Latremoliere18g} as a special case of the covariant modular propinquity, which even allows taking the distance between actions from \emph{different} proper monoids, and thus involved a notion of \emph{almost isometric isomorphism} between monoids. The formalism we have used so far can shade some light on our construction of the covariant propinquity.

For our present purpose, we focus on the case needed for the spectral propinquity. In this case, the only proper monoid involved is $[0,\infty)$, and the tool used to define a distance between proper monoids in \cite{Latremoliere18b} reduces to the following notion.

\begin{definition}[{\cite[Definition 2.5]{Latremoliere18b}}]\label{iso-iso-def}
  Let $\varepsilon > 0$. A \emph{$\varepsilon$-iso-iso} $(\xi_1,\xi_2)$ of $[0,\infty)$ is an ordered pair of maps $\xi_1,\xi_2 : [0,\infty) \rightarrow [0,\infty)$ such that $\xi_1(0) = \xi_2(0) = 0$ and, for all $x,y,z \in \left[0,\frac{1}{\varepsilon}\right]$, and for all $\{j,k\} = \{1,2\}$,
  \begin{equation*}
    \left| |\xi_j(x)+\xi_j(y)-z| - |x+y-\xi_k(z)| \right| < \varepsilon \text.
  \end{equation*}

  The set of all $\varepsilon$-iso-iso of $[0,\infty)$ is denoted by $\UIsoMono{\varepsilon}{[0,\infty)}$.
\end{definition}

The idea behind $\varepsilon$-iso-isos is that they are a pair of functions, each almost an isometric isomorphism, and each almost an inverse of the other (though neither need to have an actual inverse), where ``almost'' is quantified using $\varepsilon$. Thus, using iso-isos introduces an asymmetry, but also enables one to ask for a more refined comparison between two families of operators by taking advantage of the structure of their monoid index sets. In the language of this section, we introduced in \cite{Latremoliere18g} the following definition, which has obvious similarities with Definition (\ref{separation-def}), but accounts for the asymmetry introduced by working with iso-isos by using $\LHaus{\KantorovichJ{\TDN}}$, rather than using $\Haus{\KantorovichJ{\TDN}}$, and involves the natural action of $[0,\infty)$ by unitaries induced by the Dirac operators of spectral triples. Intuitively, the following definition aims at quantifying how far the orbits of certain vectors for these actions differ as time passes, with allowance for some distortion in the passage of time given by the iso-isos.

\begin{definition}[{\cite[Definition 3.19]{Latremoliere18g}}]\label{reach-def}
  Let $(\A_1,\Hilbert_1,\Dirac_1)$ and $(\A_2,\Hilbert_2,\Dirac_2)$ be two metric spectral triples, and let
  \begin{equation*}
    \tunnel{\tau}{(\A_1,\Hilbert_1,\Dirac_1)}{(\Pi_1,\pi_1,\theta_1)}{\mathds{P}}{(\Pi_2,\pi_2,\theta_2)}{(\A_2,\Hilbert_2,\Dirac_2)}
  \end{equation*}
  be a tunnel from $(\A_1,\Hilbert_1,\Dirac_1)$ to $(\A_2,\Hilbert_2,\Dirac_2)$. For each $j\in \{1,2\}$, let $\CDN_j$ be the graph norm of $\Dirac_j$, and let $U_j : t \in [0,\infty) \mapsto \exp( i t \Dirac_j )$.
  
  For all $\varepsilon > 0$ and for all $(\varsigma_1,\varsigma_2) \in \UIsoMono{\varepsilon}{[0,\infty)}$, we define $\tunnelhalfreach{\tau,\varsigma_1}{\varepsilon}$ by:
  \begin{multline*}
    \LHaus{\Kantorovich{\TDN}}\Big( \left\{ \left(\inner{U_1^{-t} \xi}{\Pi_1(\cdot)}{\Hilbert_1}\right)_{t\in[0,\frac{1}{\varepsilon}]} : \xi\in\dom{\Dirac_1},\CDN_1(\xi)\leq 1 \right\}, \\
    \left\{ \left(\inner{U_2^{-\varsigma_1(t)}\xi}{\Pi_2(\cdot)}{\Hilbert_2}\right)_{t \in [0, \frac{1}{\varepsilon}]} : \xi\in\dom{\Dirac_2},\CDN_2(\xi)\leq 1 \right\} \Big) \text.
  \end{multline*}

  The \emph{$\varepsilon$-reach} $\tunnelreach{\tau,\varsigma_1,\varsigma_2}{\varepsilon}$  of $(\tau,\varsigma_1,\varsigma_2)$ is the nonnegative real number:
  \begin{equation}\label{reach-eq}
    \tunnelreach{\tau,\varsigma_1,\varsigma_2}{\varepsilon} \coloneqq \max\left\{ \tunnelhalfreach{\tau,\varsigma_1}{\varepsilon}, \tunnelhalfreach{\tau^{-1},\varsigma_2}{\varepsilon} \right\} \text.
  \end{equation}

  The \emph{$\varepsilon$-magnitude} $\tunnelmagnitude{\tau,\chi_2,\chi_2}{\varepsilon}$ of $(\tau,\chi_1,\chi_2)$ is $\max\left\{\tunnelextent{\tau}, \tunnelreach{\tau,\chi_1,\chi_2}{\varepsilon}\right\}$.
\end{definition}

We use the notation of Definition (\ref{reach-def}) for the following two remarks.
\begin{remark}
  For clarity, we note that $\tunnelhalfreach{\tau^{-1},\varsigma_2}{\varepsilon}$ is given by:
  \begin{multline*}
    \LHaus{\Kantorovich{\TDN}}\Big( \left\{ \left(\inner{U_2^{-t} \xi}{\Pi_2(\cdot)}{\Hilbert_2}\right)_{t\in[0,\frac{1}{\varepsilon}]} : \xi\in\dom{\Dirac_2},\CDN_2(\xi)\leq 1 \right\}, \\
    \left\{ \left(\inner{U_1^{-\varsigma_2(t)}\xi}{\Pi_1(\cdot)}{\Hilbert_1}\right)_{t \in [0, \frac{1}{\varepsilon}]} : \xi\in\dom{\Dirac_1},\CDN_1(\xi)\leq 1 \right\} \Big) \text.
  \end{multline*}
\end{remark}

\begin{remark}
  Expression (\ref{reach-eq}) was given ``unwound'' in \cite{Latremoliere18g}, in the following form:
  \begin{multline*}
    \tunnelreach{\tau,\varsigma_1,\varsigma_2}{\varepsilon} = \max_{\{j,k\}=\{1,2\}} \; \sup_{\substack{\xi \in \dom{\Dirac_j}\\ \CDN_j(\xi)\leq 1}} \; \inf_{\substack{\eta\in\dom{\Dirac_k}\\ \CDN_k(\eta)\leq 1}} \; \sup_{0\leq t \leq \frac{1}{\varepsilon}} \\ \sup_{\substack{\omega\in\dom{\TDN} \\ \TDN(\omega)\leq 1}} \left| \inner{U_j^{-t}\xi}{\Pi_j(\omega)}{\Hilbert_j} - \inner{U_k^{-\varsigma_k(t)}\eta}{\Pi_k(\omega)}{\Hilbert_k}  \right| \text.
  \end{multline*}
\end{remark}

The object of this paper is the study of some properties of the following metric, introduced in \cite{Latremoliere18g}.
\begin{definition}[{\cite[Definition 4.2]{Latremoliere18g}}]\label{spectral-propinquity-def}
  The \emph{spectral propinquity} $\spectralpropinquity{}((\A_1,\Hilbert_1,\Dirac_1),(\A_2,\Hilbert_2,\Dirac_2))$ between two metric spectral triples $(\A_1,\Hilbert_1,\Dirac_1)$ and $(\A_2,\Hilbert_2,\Dirac_2)$ is:
  \begin{multline*}
    \inf\Big\{\frac{\sqrt{2}}{2},\varepsilon > 0 : \exists \tau \text{ tunnel from $(\A_1,\Hilbert_1,\Dirac_1)$ to $(\A_2,\Hilbert_2,\Dirac_2)$, } \\ \exists (\varsigma,\xi) \in \UIsoMono{\varepsilon}{[0,\infty)} \quad  \tunnelmagnitude{\tau,\varsigma,\xi}{\varepsilon} < \varepsilon \Big\} \text.
  \end{multline*}
\end{definition}

Our first step is to establish an easier formulation of the spectral propinquity, by removing the iso-isos, which are useful in general for the covariant modular propinquity, but turn out to play no role in the spectral propinquity.

\section{Equivalent Formulations of the Spectral Propinquity}

We take advantage of both the specific geometry of $[0,\infty)$ (and $\R$), and the rather specific nature of the actions by unitaries, to derive a simpler version of the spectral propinquity. First, we note that $\varepsilon$-iso-iso of $[0,\infty)$ are $\varepsilon$-close to the identity map on $\left[0,\frac{1}{\varepsilon}\right]$.

\begin{lemma}\label{iso-iso-lemma}
  If $(\varsigma,\varkappa) \in \UIsoMono{\varepsilon}{[0,\infty)}$, then
  \begin{equation*}
    \forall t \in \left[0,\frac{1}{\varepsilon}\right] \quad \max\left\{ |\varkappa(t)-t|, |\varsigma(t)-t| \right\} < \varepsilon \text.
  \end{equation*}
\end{lemma}

\begin{proof}
  For all $x,y,z \in \left[ 0, \frac{1}{\varepsilon} \right]$, by Definition (\ref{iso-iso-def}) of an $\varepsilon$-iso-iso, we have:
  \begin{equation}\label{iso-iso-removal-eq}
    \left| |\varsigma(x)+\varsigma(y) - z| - |(x+y) - \varkappa(z)| \right| < \varepsilon \text,
  \end{equation}
  and $\varsigma(0) = \varkappa(0) = 0$.
  
  Setting $z = y = 0$ and $x = t \in \left[0,\frac{1}{\varepsilon}\right]$ in Equation (\ref{iso-iso-removal-eq}), we get
  \begin{equation*}
    \left| |\varsigma(t)| -|t| \right| < \varepsilon \text.
  \end{equation*}
  
  Since $\varsigma(t) \geq 0$ for all $t \geq 0$, we conclude that
  \begin{equation*}
    \forall t \in \left[0,\frac{1}{\varepsilon}\right] \quad |\varsigma(t) - t| < \varepsilon \text.
  \end{equation*}
  The argument for $\varkappa$ is similar.
\end{proof}

\bigskip

Let us now work with two metric spectral triples, and a tunnel between them. For this section, it is very helpful to group some recurrent notations within one hypothesis.

\begin{hypothesis}\label{working-hyp}
  For each $j\in\{1,2\}$, let $(\A_j,\Hilbert_j,\Dirac_j)$ be a metric spectral triple. We write, as in Theorem (\ref{mcc-thm}),
  \begin{equation*}
    \mcc{\A_j}{\Hilbert_j}{\Dirac_j} \coloneqq \left(\Hilbert_j,\CDN_j,\A_j,\Lip_j,\C,0\right)\text,
  \end{equation*}
  where, for all $\xi \in \dom{\Dirac_j}$, we have
  \begin{equation*}
    \CDN_j(\xi) \coloneqq \norm{\xi}{\Hilbert_j} + \norm{\Dirac_j \xi}{\Hilbert_j}\text,
  \end{equation*}
  and, for all $a\in\dom{\Lip_j}$, where
  \begin{equation*}
    \dom{\Lip_j} \coloneqq \left\{ b \in \sa{\A_j} : a \cdot \dom{\Dirac_j}\subseteq \dom{\Dirac_j}\text{ and }[\Dirac_j,a] \text{ bounded}\right\} \text,
  \end{equation*}
  we set
  \begin{equation*}
    \Lip_j(a) \coloneqq \opnorm{[\Dirac_j,a]}{}{\Hilbert_j}\text.
  \end{equation*}
  We recall from Convention (\ref{dom-convention}) that if $\xi\notin\dom{\Dirac_j}$, then $\CDN_j(\xi) = \infty$, and if $a\notin\dom{\Lip_j}$, then $\Lip_j(a) = \infty$.
  For each $t\in \R$, we define
  \begin{equation*}
    U_j^t = \exp(i t \Dirac_j) \text.
  \end{equation*}
  Of course, $U_j^t$ is an action of $\R$ on $\Hilbert_j$ by unitaries, and by isometries of $\CDN_j$. Let
  \begin{equation*}
    \tunnel{\tau}{(\A_1,\Hilbert_1,\Dirac_1)}{(\Pi_1,\pi_1,\theta_1)}{\mathds{P}}{(\Pi_2,\pi_2,\theta_2)}{(\A_2,\Hilbert_2,\Dirac_2)}
  \end{equation*}
  be a tunnel from $(\A_1,\Hilbert_1,\Dirac_1)$ to $(\A_2,\Hilbert_2,\Dirac_2)$, where $\mathds{P} = (\mathscr{J},\TDN,\D,\Lip_\D,\alg{E},\Lip_{\alg{E}})$ is a {\gMVB}.
  Last, let $\varepsilon > 0$ be fixed, and let $\varsigma = (\varsigma_1,\varsigma_2) \in \UIsoMono{\varepsilon}{[0,\infty)}$.
\end{hypothesis}

\medskip

\begin{lemma}\label{iso-iso-2-lemma}
  If we assume Hypothesis (\ref{working-hyp}), then for each $j\in\{1,2\}$, we conclude:
  \begin{equation*}
    \sup\left\{ \norm{U_j^{\varsigma_j(t)}\xi - U_j^t\xi}{\Hilbert_j} : \xi\in\dom{\Dirac_j}, \CDN_j(\xi)\leq 1; t \in \left[0,\frac{1}{\varepsilon}\right] \right\} \leq \varepsilon \text.
  \end{equation*}
\end{lemma}

\begin{proof}
  Let $\xi \in \dom{\Dirac_j}$ with $\CDN_j(\xi)\leq 1$. Let $t \in \left[0,\frac{1}{\varepsilon}\right]$. Since $U_j$ is a group morphism, we get, for all $s\in\R$,
\begin{align*}
  \frac{1}{s}(U_j^{t+s}\xi-U_j^t\xi)
  &= U_j^t\left( \frac{1}{s}(U_j^s-1)\xi \right) \\
  &\xrightarrow{s\rightarrow 0} U_j^t \Dirac_j \xi \text.
\end{align*}
Thus, in particular, $\norm{\left(\frac{d U_j^s \xi}{ds}\right)_{s=t}}{\Hilbert_j} = \norm{\Dirac_j\xi}{\Hilbert_j} \leq \CDN_j(\xi)$.

We therefore have, with $x= \min\{t,\varsigma_j(t)\}$ and $y = \max\{t,\varsigma_j(t)\}$.

\begin{align*}
  \norm{U_j^{\varsigma_j(t)}\xi - U_j^t\xi}{\Hilbert_j}
  &= \norm{\int_t^{\varsigma_j(t)} \left(\frac{d U_j^s \xi}{ds}\right)_{s=w} \, dw}{\Hilbert_j} \\
  &\leq \int_x^y \norm{\left(\frac{dU_j^s \xi}{ds}\right)_{s=w}}{\Hilbert_j} \, dw \\
  &\leq \int_x^y \CDN_j(\xi) \, dw \\
  &\leq \int_x^y \, dw \text{ since $\CDN_j(\xi)\leq 1$,}\\
  &\leq y-x = |\varsigma_j(t)-t| \\
  &\leq \varepsilon \text{ by Lemma (\ref{iso-iso-lemma}).}
\end{align*}
 
This concludes our proof.
\end{proof}

\begin{lemma}\label{reach-correction-lemma}
  If we assume Hypothesis (\ref{working-hyp}), then we conclude:
  \begin{equation*}
    \tunnelsep{\tau}{(U_1^t)_{t \in \left[0, \frac{1}{\varepsilon}\right]}, (U_2^t)_{t \in \left[0, \frac{1}{\varepsilon}\right]}} \leq \varepsilon + \tunnelmodreach{\tau,\varsigma_1,\varsigma_2}{\varepsilon} \text.
  \end{equation*}
\end{lemma}

\begin{proof}
  Let $\{j,k\} = \{1,2\}$. Let $\xi \in \dom{\Dirac_j}$ with $\CDN_j(\xi)\leq 1$. By definition of the $\varepsilon$-reach $\tunnelmodreach{\tau,\varsigma_1,\varsigma_2}{\varepsilon}$, there exists $\eta \in \dom{\Dirac_k}$ such that $\CDN_k(\eta)\leq 1$, and
  \begin{equation*}
    \sup_{\substack{\omega\in\dom{\TDN} \\ \TDN(\omega)\leq 1}} \sup_{t\in \left[0,\frac{1}{\varepsilon}\right]} \left| \inner{\xi}{U_j^t\Pi_j(\omega)}{\Hilbert_j} - \inner{\eta}{U_k^{\varsigma_j(t)}\Pi_k(\omega)}{\Hilbert_k} \right| \leq \tunnelmodreach{\tau,\varsigma_1,\varsigma_2}{\varepsilon} \text.
  \end{equation*}

  Let $\omega\in\dom{\TDN}$ with $\TDN(\omega)\leq 1$. Using Cauchy-Schwarz's inequality and the fact that $\CDN_k(\Pi_k(\omega)) \leq \TDN(\omega)) \leq 1$, we note that
  \begin{align*}
    \Bigg| \inner{\xi}{U_j^t\Pi_j(\omega)}{\Hilbert_j} &- \inner{\eta}{U_k^{t}\Pi_k(\omega)}{\Hilbert_k}\Bigg| \\
                                                       &\leq \left|\inner{\xi}{U_j^t\Pi_j(\omega)}{\Hilbert_j} - \inner{\eta}{U_k^{\varsigma_j(t)}\Pi_k(\omega)}{\Hilbert_k}\right| \\
    &\quad + \left|\inner{\eta}{\left(U_k^t - U_k^{\varsigma_j(t)}\right)\Pi_k(\omega)}{\Hilbert_k}\right| \\
    &\leq \tunnelmodreach{\tau,\varsigma}{\varepsilon} +  \norm{\eta}{\Hilbert_k} \norm{(U_k^t-U_k^{\varsigma_j(t)})\Pi_k(\omega)}{\Hilbert_k}  \\
    &\leq \tunnelmodreach{\tau,\varsigma}{\varepsilon} + \varepsilon  \text{ by Lemma (\ref{iso-iso-2-lemma}).}
  \end{align*}

  This concludes our proof.  
\end{proof}

We therefore can now prove that we can dispense with iso-isos and simply use the identity of $[0,\infty)$ in the definition of the spectral propinquity, to the cost of obtaining an equivalent distance on metric spectral triples. Thus, all computations may as well be done with this version of the propinquity.

\begin{theorem}\label{id-thm}
  If, for any two metric spectral triples $(\A_1,\Hilbert_1,\Dirac_1)$ and $(\A_2,\Hilbert_2,\Dirac_2)$, we let $\spectralpropinquity{\ast}((\A_1,\Hilbert_1,\Dirac_1),(\A_2,\Hilbert_2,\Dirac_2))$ be
  \begin{multline*}
    \inf\Bigg\{\frac{\sqrt{2}}{2}, \varepsilon > 0 : \exists \text{ tunnel $\tau$ from }(\A_1,\Hilbert_1,\Dirac_1)\text{ to }(\A_2,\Hilbert_2,\Dirac_2) \\ \text{such that }\tunneldispersion{\tau}{(U_1^t)_{0\leq t\leq \frac{1}{\varepsilon}},(U_2^t)_{0\leq t \leq \frac{1}{\varepsilon}}} < \varepsilon \Bigg\} \text,
  \end{multline*}
  where $\tunneldispersion{\cdot}{\cdot}$ is defined in Definition (\ref{separation-def}), then $\spectralpropinquity{\ast}$ is a distance on the space of metric spectral triples, up to unitary equivalence, which is equivalent to $\spectralpropinquity{}$; in fact:
  \begin{multline*}
    \spectralpropinquity{}((\A_1,\Hilbert_1,\Dirac_1),(\A_2,\Hilbert_2,\Dirac_2))
  \leq \spectralpropinquity{\ast}(\A_1,\Hilbert_1,\Dirac_1),(\A_2,\Hilbert_2,\Dirac_2)) \\ \leq 2\spectralpropinquity{}((\A_1,\Hilbert_1,\Dirac_1),(\A_2,\Hilbert_2,\Dirac_2)) 
\end{multline*}
for any two metric spectral triples $(\A_1,\Hilbert_1,\Dirac_1)$ and $(\A_2,\Hilbert_2,\Dirac_2)$.
\end{theorem}

\begin{proof}
  We first note that, by following the proof in \cite[Theorem 3.18]{Latremoliere18g}, we can prove that the class function $\spectralpropinquity{\ast}$ satisfies the triangle inequality, and of course, it is symmetric.
  
  By construction, we immediately conclude that:
  \begin{equation*}
    \spectralpropinquity{}((\A_1,\Hilbert_1,\Dirac_1),(\A_2,\Hilbert_2,\Dirac_2)) \leq \spectralpropinquity{\ast}((\A_1,\Hilbert_1,\Dirac_1),(\A_2,\Hilbert_2,\Dirac_2)) \text.
  \end{equation*}

  If $\spectralpropinquity{}((\A_1,\Hilbert_1,\Dirac_1),(\A_2,\Hilbert_2,\Dirac_2)) = \frac{\sqrt{2}}{2}$, then of course, our result holds. We henceforth assume that $\spectralpropinquity{}((\A_1,\Hilbert_1,\Dirac_1),(\A_2,\Hilbert_2,\Dirac_2))<\frac{\sqrt{2}}{2}$.
  
  Let $\varepsilon \in \left( \spectralpropinquity{}((\A_1,\Hilbert_1,\Dirac_1),(\A_2,\Hilbert_2,\Dirac_2)), \frac{\sqrt{2}}{2}\right)$. By Definition (\ref{spectral-propinquity-def}), there exists an $\varepsilon$-iso-iso $\varsigma=(\varsigma_1,\varsigma_2) \in \UIsoMono{\varepsilon}{[0,\infty)}$ and a tunnel $\tau$ from $(\A_1,\Hilbert_1,\Dirac_1)$ to $(\A_2,\Hilbert_2,\Dirac_2)$ such that $\tunnelmodmagnitude{\tau,\varsigma_1,\varsigma_2}{\varepsilon} \leq \varepsilon$. In fact, we will use the notations introduced in Hypothesis (\ref{working-hyp}) for our tunnel $\tau$, with the extra assumption given here on its $\varepsilon$-magnitude.
  
  By Lemma (\ref{reach-correction-lemma}), since $\tunnelmodreach{\tau,\varsigma_1,\varsigma_2}{\varepsilon} \leq \tunnelmodmagnitude{\tau,\varsigma_1,\varsigma_2}{\varepsilon} \leq \varepsilon$, and since $(\varsigma_1,\varsigma_2) \in \UIsoMono{\varepsilon}{[0,\infty)}$, we conclude that
  \begin{equation*}
    \tunnelsep{\tau}{(U_1^t)_{t \in \left[0, \frac{1}{\varepsilon}\right]}, (U_2^t)_{t \in \left[0, \frac{1}{\varepsilon}\right]}}\leq \varepsilon + \tunnelmodreach{(\tau,\varsigma,\varkappa)}{\varepsilon} \leq 2 \varepsilon \text.
  \end{equation*}

  Of course, the extents of $(\tau,\varsigma,\varkappa)$ is equal to the extent of $\tau$, by definition.
  
  Therefore, $\tunneldispersion{\tau}{(U_1^t)_{t \in \left[0, \frac{1}{\varepsilon}\right]}, (U_2^t)_{t \in \left[0, \frac{1}{\varepsilon}\right]}} \leq 2 \varepsilon$. By construction, it then follows that $\spectralpropinquity{\ast}((\A_1,\Hilbert_1,\Dirac_1),(\A_2,\Hilbert_2,\Dirac_2)) \leq 2 \varepsilon$.

  We therefore conclude, as $\varepsilon \in \left( \spectralpropinquity{}((\A_1,\Hilbert_1,\Dirac_1),(\A_2,\Hilbert_2,\Dirac_2)), \frac{\sqrt{2}}{2}\right)$ is arbitrary:
  \begin{equation*}
    \spectralpropinquity{\ast}((\A_1,\Hilbert_1,\Dirac_1),(\A_2,\Hilbert_2,\Dirac_2)) \leq 2 \spectralpropinquity{}((\A_1,\Hilbert_1,\Dirac_1),(\A_2,\Hilbert_2,\Dirac_2)) \text.
  \end{equation*}

  This concludes our proof.
\end{proof}

In sight of Theorem (\ref{id-thm}), the following notation seems natural
\begin{notation}
  If $\tau$ is a tunnel between two metric spectral triples, and if $\varepsilon > 0$, then we write $\tunnelreach{\tau}{\varepsilon}$ and $\tunnelmodmagnitude{\tau}{\varepsilon}$ for, respectively, $\tunnelreach{\tau,(\mathrm{id},\mathrm{id})}{\varepsilon}$ and $\tunnelmodmagnitude{\tau,(\mathrm{id},\mathrm{id})}{\varepsilon}$, where $\mathrm{id}$ is the identity of $[0,\infty)$.
\end{notation}

\bigskip

While any metric spectral triple $(\A,\Hilbert,\Dirac)$ induces an action $U:t\in\R\mapsto \exp(i t \Dirac)$ of $\R$ by unitaries of $\Hilbert$, the spectral propinquity, in its original form, and in the equivalent form from Theorem (\ref{id-thm}), only involves the restriction of this action to the monoid $[0,\infty)$. This is by design: there is only one possible isometric monoid morphism of $[0,\infty)$, namely, the identity, which is part of why the spectral propinquity is indeed a metric, up to unitary equivalence \cite[Theorem 4.4]{Latremoliere18g}. It is nonetheless natural to ask what can be said about the action for negative times, as well. This is the matter of the next theorem of this section.

\begin{theorem}\label{sym-thm}
  If we assume Hypothesis (\ref{working-hyp}), then for all $r \in \left[0,\frac{1}{\varepsilon}\right]$, writing $J(r) \coloneqq \left[-r,\frac{1}{\varepsilon}-r\right]$, we conclude:
  \begin{equation*}
    \tunnelsep{\tau}{(U_1^t)_{t \in J(r)},(U_2^t)_{t \in J(r)}} = \tunnelsep{\tau}{(U_1^t)_{t \in [0,\frac{1}{\varepsilon}]},(U_2^t)_{t \in [0,\frac{1}{\varepsilon}]}} \text,
  \end{equation*}
  and therefore:
  \begin{equation*}
    \tunneldispersion{\tau}{(U_1^t)_{t \in J(r)},(U_2^t)_{t \in J(r)}} = \tunneldispersion{\tau}{(U_1^t)_{t \in [0,\frac{1}{\varepsilon}]},(U_2^t)_{t \in [0,\frac{1}{\varepsilon}]}} \text,
  \end{equation*}
where $\tunnelsep{\cdot}{\cdot}$ and $\tunneldispersion{\cdot}{\cdot}$ are defined in Definition (\ref{separation-def}).  
\end{theorem}

\begin{proof}
Let $r \in \left[0, \frac{1}{\varepsilon}\right] > 0$ and $\rho = \tunnelsep{\tau}{(U_1^t)_{t \in [0,\frac{1}{\varepsilon}]},(U_2^t)_{t \in [0,\frac{1}{\varepsilon}]}}$. Let $\{ j,k \} = \{ 1,2 \}$.

Let $\xi \in \Hilbert_j$ such that $\CDN_j(\xi)\leq 1$. Now, set $\xi' = U_j^{r} \xi$. Note that $\CDN_j(\xi')\leq 1$. By Definition (\ref{reach-def}) of the reach, there exists $\eta \in \dom{\Dirac_k}$ with $\CDN_k(\eta)\leq 1$, such that
\begin{equation*}
  \sup_{\substack{\omega\in\mathscr{J} \\ \TDN(\omega)\leq 1}} \sup_{0\leq t \leq \frac{1}{\varepsilon}} \left|\inner{\xi'}{U_j^t\Pi_j(\omega)}{\Hilbert_j} - \inner{\eta}{U_k^t\Pi_k(\omega)}{\Hilbert_k} \right| \leq \rho \text.
\end{equation*}

Let $t\in\left[ -r, \frac{1}{\varepsilon}-r \right]$ and $\omega\in\dom{\TDN}$ with $\TDN(\omega)\leq 1$. We then conclude:
\begin{align*}
  \big|\inner{\xi }{U_j^t \Pi_j(\omega)}{\Hilbert_j}
  &- \inner{U_k^{-r}\eta}{U_k^t\Pi_k(\omega)}{\Hilbert_k}\big| \\
  &= \left| \inner{U_j^{r}\xi}{U_j^{t+r}\Pi_j(\omega)}{\Hilbert_j} - \inner{\eta}{U_k^{t+r}\Pi_k(\omega)}{\Hilbert_k} \right| \\
  &\leq \rho \text{ since $0\leq t+r \leq \frac{1}{\varepsilon}$.}
\end{align*}

Since $\CDN_k(U_k^{-r}\eta)\leq 1$ as well, we therefore have found $\zeta \coloneqq U_k^{-r}\eta$ with $\CDN_k(\zeta)\leq 1$, such that for all $\omega\in\dom{\TDN}$ with $\TDN(\omega)\leq 1$, and for all $t \in \left[-r,\frac{1}{\varepsilon}-r\right]$:
\begin{equation*}
  \left| \inner{\xi}{U_j^t \Pi_j(\omega)}{\Hilbert_j} - \inner{\eta}{U_k^t \Pi_k(\omega)}{\Hilbert_k} \right|\text.
\end{equation*}
Since $\xi$ was arbitrary with $\CDN_j(\xi)\leq 1$, we thus have shown, by Definition (\ref{separation-def}):
\begin{equation*}
  \tunnelsep{\tau}{(U_1^t)_{t \in [-r,\frac{1}{\varepsilon}-r]},(U_2^t)_{t \in [-r,\frac{1}{\varepsilon}-r]}} \leq \rho \text.
\end{equation*}

The converse inequality is proven in the same manner.
\end{proof}

In particular, we devise yet another equivalent form of the spectral propinquity.
\begin{corollary}\label{sym-cor}
  If, for any two metric spectral triples $(\A_1,\Hilbert_1,\Dirac_1)$ and $(\A_2,\Hilbert_2,\Dirac_2)$, for any tunnel $\tau$ from $(\A_1,\Hilbert_1,\Dirac_1)$ to $(\A_2,\Hilbert_2,\Dirac_2)$, and for all $\varepsilon > 0$, we set $J(\frac{1}{\varepsilon}) \coloneqq \left[\frac{-1}{2\varepsilon},\frac{1}{2\varepsilon}\right]$ and:
  \begin{equation*}
    \tunnelmodsymmagnitude{\tau}{\varepsilon} \coloneqq \tunneldispersion{\tau}{(U_1^t)_{t \in J(\frac{1}{\varepsilon})}, (U_2^t)_{t \in J(\frac{1}{\varepsilon})}} \text,
  \end{equation*}
  then
  \begin{multline*}
    \spectralpropinquity{\ast}((\A_1,\Hilbert_1,\Dirac_1),(\A_2,\Hilbert_2,\Dirac_2)) = \inf\Big\{ \frac{\sqrt{2}}{2} , \varepsilon > 0 : \exists \text{ tunnel }\tau \text{ from }(\A_1,\Hilbert_1,\Dirac_1) \\ \text{ to }(\A_2,\Hilbert_2,\Dirac_2) \text{ such that } \tunnelmodsymmagnitude{\tau}{\varepsilon} \leq \varepsilon \Big\} \text.
  \end{multline*}
\end{corollary}

\begin{proof}
  We apply Theorem (\ref{sym-thm}) with $r = \frac{1}{2\varepsilon}$.
\end{proof}

In passing, we record the following useful consequence of Corollary (\ref{sym-cor}).
\begin{corollary}\label{minus-cor}
  If a sequence $(\A_n,\Hilbert_n,\Dirac_n)_{n\in\N}$ of metric spectral triples converges to a metric spectral triple $(\A_\infty,\Hilbert_\infty,\Dirac_\infty)$ for the spectral propinquity, then the sequence $(\A_n,\Hilbert_n,-\Dirac_n)_{n\in\N}$ converges to $(\A,\Hilbert,-\Dirac_\infty)$.
\end{corollary}

\begin{proof}
  By Corollary (\ref{sym-cor}), noting that $\exp(it(-\Dirac_n)) = \exp(i(-t) \Dirac_n)$ for all $n\in\N\cup\{\infty\}$ and for all $t \in \R$, we conclude that $(\A_n,\Hilbert_n,-\Dirac_n)_{n\in\N}$ converges to $(\A_\infty,\Hilbert_\infty,-\Dirac_\infty)$.
\end{proof}

\bigskip

The core mechanism which enables us to prove various results about the propinquity, spectral or otherwise, are certain particular set-valued functions between {\qcms s} and {\gMVB s}, called \emph{target sets} \cite{Latremoliere16c,Latremoliere18d,Latremoliere18g}. When working with {\gMVB s}, however, a different choice of set-valued maps appear natural as well: if target sets are related to the notion of extent, then another choice is related to the notion of reach. We now elucidate the relation between these set-valued maps, as it will play a role in our proofs in this paper.

We begin by recalling the notion of target sets associated with a tunnel.

\begin{definition}[{\cite[Definition 3.18]{Latremoliere18d}}]\label{targetset-def}
  For each $j\in\{1,2\}$, let $\mathds{M}_j = \left(\module{M}_j,\CDN_j,\A_j,\SLip_j,\B_j,\Lip_j\right)$ be a metrical C*-correspondence.
  
  Let $\tunnel{\tau}{\mathds{M}_1}{(\Pi_1,\pi_1,\theta_1)}{\mathds{J}}{(\Pi_2,\pi_2,\theta_2)}{\mathds{M}_2}$ be a tunnel from $\mathds{M}_1$ to $\mathds{M}_2$, where
  \begin{equation*}
    \mathds{J} \coloneqq (\module{J},\TDN,\alg{E},\Lip_{\alg{E}},\D,\Lip_\D) \text.
  \end{equation*}

    For all $a\in\dom{\Lip_1}$, and for all $l\geq \Lip_1(a)$, we define the \emph{target set} of $a$ for $\tau$, controlled by $l$, as:
  \begin{equation*}
    \targetsettunnel{\tau}{a}{l} \coloneqq \left\{ \pi_2(d) \in \B_2 : d\in\dom{\Lip_\D},\pi_1(d) = a, \Lip_\D(d) \leq l \right\} \text.
  \end{equation*}

  For all $\xi \in \dom{\CDN_1}$, and for all $l\geq \CDN_1(\xi)$, we define the \emph{target set} of $\xi$ for $\tau$, controlled by $l$, as:
  \begin{equation*}
    \targetsettunnel{\tau}{\xi}{l} \coloneqq \left\{ \Pi_2(\omega) \in \module{M}_2 : \omega\in\dom{\TDN}, \Pi_1(\omega) = \xi, \TDN(\omega)\leq l \right\} \text.
  \end{equation*}
  
\end{definition}

The \emph{extent} \cite{Latremoliere13,Latremoliere13b,Latremoliere14,Latremoliere15,Latremoliere16c,Latremoliere18c,Latremoliere18g} of a tunnel controls the behavior of its target sets, by measuring, in essence, how far from giving an actual morphism the target set map is. We refer to \cite[Proposition 3.19, Proposition 3.20, Proposition 3.21]{Latremoliere18g} for the basic properties of target sets, as controlled by the extent, which we will need in this work. 

On the other hand, when working with C*-correspondences built over Hilbert spaces, the reach of a tunnel seems to also provide us with a means to define a set valued map, which a priori is distinct from the target set map. Notably, there is no requirement in the definition of the reach that vectors be related by means of a target set. The next lemma explains this situation in more details.

\begin{lemma}\label{s-set}
  For each $j\in\{1,2\}$, let $\mathds{M}_j = \left(\Hilbert_j,\CDN_j,\A_j,\SLip_j,\C,0\right)$ be a metrical C*-correspondence, where $\Hilbert_j$ is a Hilbert space.
  
  Let $\tunnel{\tau}{\mathds{M}_1}{(\Pi_1,\pi_1,\theta_1)}{\mathds{J}}{(\Pi_2,\pi_2,\theta_2)}{\mathds{M}_2}$ be a tunnel from $\mathds{M}_1$ to $\mathds{M}_2$, with $\mathds{J} \coloneqq (\module{J},\TDN,\alg{E},\Lip_{\alg{E}},\D,\Lip_\D)$.

  Let:
  \begin{multline*}
    \rho \coloneqq \Haus{\Kantorovich{\TDN}}\Big( \left\{\inner{\xi}{\Pi_1(\cdot)}{\Hilbert_1} : \xi\in\dom{\CDN_1},\CDN_1(\xi)\leq 1\right\},  \\  \left\{ \inner{\eta}{\Pi_2(\cdot)}{\Hilbert_2} : \eta\in\dom{\CDN_1},\CDN_1(\eta)\leq 1\right\} \Big)\text.
  \end{multline*}
  
  For all $\xi \in \dom{\CDN_1}$, if $l\geq \CDN_1(\xi)$, if $\eta \in \targetsettunnel{\tau}{\xi}{l}$ and $\zeta\in s(\xi,l)$, where
  \begin{multline*} 
    s(\xi,l) = \Big\{ \psi \in \dom{\CDN_2} : \CDN_2(\psi)\leq l, \\ \Kantorovich{\TDN}(\inner{\xi}{\Pi_1(\cdot)}{\Hilbert_1},\inner{\psi}{\Pi_2(\cdot)}{\Hilbert_2}) \leq \rho \CDN_1(\xi) \Big\}\text,
  \end{multline*}
  then
  \begin{equation*}
    \Kantorovich{\TDN}(\inner{\eta}{\Pi_2(\cdot)}{\Hilbert_2}, \inner{\zeta}{\Pi_2(\cdot)}{\Hilbert_2}) = \sup_{\substack{\omega\in\dom{\TDN} \\ \TDN(\omega)\leq 1}}\left|\inner{\eta-\zeta}{\Pi_2(\omega)}{\Hilbert_2}\right| \leq 4 H l \tunnelextent{\tau} \text.
  \end{equation*}
\end{lemma}

\begin{remark}
  The definition of $\rho$ in Lemma (\ref{s-set}) can be rewritten as:
  \begin{multline*}
    \rho = \max_{\{j,k\}=\{1,2\}} \sup_{\substack{\xi \in \dom{\CDN_j} \\ \CDN_j(\xi)\leq 1}} \inf_{\substack{\eta\in\dom{\CDN_k} \\ \CDN_k(\eta)\leq 1}} \sup_{\substack{\omega\in\dom{\TDN} \\ \TDN(\omega)\leq 1}} \\
    \left|\inner{\xi}{\Pi_j(\omega)}{\Hilbert_j} - \inner{\eta}{\Pi_k(\omega)}{\Hilbert_k}\right| \text.
  \end{multline*}
\end{remark}

\begin{proof}
  Let $\psi \in \dom{\TDN}$ such that $\Pi_1(\psi)=\xi$ and $\Pi_2(\psi)=\eta$, while $\TDN(\psi) = \CDN(\xi)$. Moreover, let $\omega \in \dom{\TDN}$ with $\TDN(\omega)\leq 1$.

  By assumption on the tunnel $\tau$, we first note that $\pi_1$ and $\pi_2$, as unital *-morphisms onto $\C$, are states of $\D$. Then we note that, for all $d \in \dom{\D}$ with $\Lip_\D(d) \leq 1$, we have
  \begin{equation*}
    |\pi_1(d) - \pi_2(d)| \leq \Kantorovich{\Lip_\D}(\pi_1,\pi_2) \leq \tunnelextent{\tau} \text.
  \end{equation*}

  On the other hand, note that $\inner{\xi}{\Pi_1(\omega)}{\Hilbert_1} = \inner{\Pi_1(\psi)}{\Pi_1(\omega)}{\Hilbert_1} = \pi_1(\inner{\psi}{\omega}{\module{J}})$, ad similarly $\inner{\eta}{\Pi_2(\omega)}{\Hilbert_2} = \pi_2(\inner{\psi}{\omega}{\module{J}})$, since $(\Pi_1,\pi_1)$ and $(\Pi_2,\pi_2)$ are Hilbert $C^\ast$-module morphisms.

  By the Leibniz property, we then conclude that
  \begin{equation*}
    \max\left\{ \Lip_\D(\Re\inner{\psi}{\omega}{\module{J}}), \Lip_\D(\Im\inner{\psi}{\omega}{\module{J}}) \right\} \leq H \TDN(\psi)\TDN(\omega) \leq H l \text.
  \end{equation*}

  Therefore
  \begin{align*}
    \left|\inner{\xi}{\Pi_1(\omega)}{\Hilbert_1} - \inner{\eta}{\Pi_2(\omega)}{\Hilbert_2}\right|
    &= \left|\pi_1(\inner{\psi}{\omega}{\module{J}}) - \pi_2(\inner{\psi}{\omega}{\module{J}})\right| \\
    &\leq \left|\pi_1(\Re\inner{\psi}{\omega}{\module{J}}) - \pi_2(\Re\inner{\psi}{\omega}{\module{J}})\right| \\
    &\quad + \left|\pi_1(\Im\inner{\psi}{\omega}{\module{J}}) - \pi_2(\Im\inner{\psi}{\omega}{\module{J}})\right| \\
    &\leq 2 H l \tunnelextent{\tau} \text.
  \end{align*}

  Therefore:
  \begin{align*}
    \left|\inner{\eta-\zeta}{\Pi_2(\omega)}{\Hilbert_2}\right|
    &\leq \left|\inner{\eta}{\Pi_2(\omega)}{\Hilbert_2} - \inner{\xi}{\Pi_1(\omega)}{\Hilbert_1}\right| \\
    &\quad + \left|\inner{\xi}{\Pi_1(\omega)}{\Hilbert_1} - \inner{\zeta}{\Pi_2(\omega)}{\Hilbert_2} \right| \\
    &\leq 2 H l \tunnelextent{\tau} + l \rho \text.
  \end{align*}
  By \cite[Proposition 3.13]{Latremoliere18g}, we also have that $\rho \leq 2 H \tunnelextent{\tau}$. Therefore, our theorem is proven.
\end{proof}

A consequence of Lemma (\ref{s-set}) is that, if $\xi \in \dom{\CDN_1}$ and $\CDN_1(\xi)\leq 1$, then
\begin{equation*}
  \Haus{\norm{\cdot}{\B}}\left(s(\xi),\targetsettunnel{\tau}{\xi}{1}\right) \leq 4 H  \tunnelmodsymmagnitude{\tau}{\varepsilon} \text.
\end{equation*}

\bigskip

We can recover the analogues of Hilbert $C^\ast$-module morphism properties for the sets $s(\xi,l)$ from Lemma (\ref{s-set}) --- though at times a direct proof is easy; for instance it is immediate that
\begin{equation*}
  \forall\xi,\zeta \quad \forall l \geq \max\{\CDN_1(\xi),\CDN_1(\zeta)\} \quad \forall t \in \C \quad t s(\xi,l) + s(\zeta,l) \subseteq s(t \xi+\zeta,(1+|t|)l)\text.
\end{equation*}
We also remark that the set $s(\xi,l)$ is not generalizable as is to general metrical C*-correspondences: the situation when working with Hilbert spaces is special, since choosing a modular pseudo-state over a Hilbert space is the same as choosing a vector in the space, while in general, a modular pseudo-state involves both an element of the module and a state of the underlying C*-algebra acting as scalars on the right.

For our purpose, a useful corollary of Lemma (\ref{s-set}) is as follows.
\begin{corollary}\label{inner-cor}
  We use the hypothesis of Lemma (\ref{s-set}). For all $\xi,\xi' \in \dom{\CDN_1}$, and for all $l\geq \max\{\CDN_1(\xi),\CDN_1(\xi')\}$, if $\eta\in s(\xi,l)$ and $\eta'\in s(\xi',l)$, then:
  \begin{equation*}
    \left|\inner{\eta}{\eta'}{\Hilbert_2}\right| \leq 9 H l^2 \tunnelextent{\tau} + \left|\inner{\xi}{\xi'}{\Hilbert_1}\right| \text.
  \end{equation*}
\end{corollary}

\begin{proof}
  Let $\eta \in \targetsettunnel{\tau}{\xi}{l}$ and $\eta'\in \targetsettunnel{\tau}{\xi'}{l)}$. We let $\omega \in \Pi_1^{-1}(\{\xi\})$, $\omega'\in\Pi_1^{-1}(\xi')$ such that $\TDN(\omega)\leq l$, $\TDN(\omega')\leq l$, and $\Pi_2(\omega)=\eta$, $\Pi_2(\omega')=\eta'$.

  Note that $\pi_2(\inner{\omega}{\omega'}{\mathscr{J}}) = \inner{\Pi_2(\omega)}{\Pi_2(\omega')}{\Hilbert_2} = \inner{\eta}{\eta'}{\Hilbert_2}$.

  Now, let $\theta\in\R$ be any argument of the complex number $\inner{\eta}{\eta'}{\Hilbert_2}$. Now, we note that
  \begin{align*}
    \inner{\exp(i\theta)\eta}{\eta'}{\Hilbert_2}
    &= \exp(-i\theta)\inner{\eta}{\eta'}{\Hilbert_2} \in \R \\
    &= \Re\inner{\exp(i\theta)\eta}{\eta'}{\Hilbert_2} \text,
  \end{align*}
  while $\Im\inner{\exp(i\theta)\eta}{\eta'}{\Hilbert_2} = 0$.

  On the other hand:
  \begin{equation*}
    \Lip_\D(\Re\inner{\exp(i\theta)\omega}{\omega'}{\mathscr{J}}) \leq H \TDN(\exp(i\theta)\omega)\TDN(\omega') \leq H l^2 \text.
  \end{equation*}
  By linearity, we  have
  \begin{equation*}
    \pi_2(\Re\inner{\exp(i\theta)\omega}{\omega'}{\mathscr{J}}) = \Re \inner{\exp(i\theta)\eta}{\eta'}{\Hilbert_2} = \exp(-i\theta)\inner{\eta}{\eta'}{\Hilbert_2} \text,
  \end{equation*}
  and similarly, $\pi_1(\Re\inner{\omega}{\omega'}{\mathscr{J}}) = \Re\inner{\xi}{\xi'}{\Hilbert_1}$.
  
  Therefore, $\exp(-i\theta)\inner{\eta}{\eta'}{\Hilbert_2} \in \targetsettunnel{\tau}{\Re\inner{\xi}{\xi'}{\Hilbert_1}}{H l^2}$.

  By \cite{Latremoliere14,Latremoliere18d}, we then conclude
  \begin{align*}
    \left|\inner{\eta}{\eta'}{\Hilbert_2}\right|
    &= \left|\exp(-i\theta)\inner{\eta}{\eta'}{\Hilbert_2}\right| 
      \leq \left|\Re\inner{\xi}{\xi'}{\Hilbert_1}\right| + H l^2 \tunnelextent{\tau}  \\
    &\leq \left|\inner{\xi}{\xi'}{\Hilbert_1}\right| + H l^2 \tunnelextent{\tau}  \text. 
  \end{align*}

  By assumption, $\CDN_2(\eta)\leq l$,$\CDN_2(\eta')\leq l$, $\CDN_2(\zeta)\leq l$ and $\CDN_2(\zeta')\leq l$. Now, using Lemma (\ref{s-set}), if $\zeta\in s(\xi)$ and $\zeta' \in s(\xi')$, then we compute
  \begin{align*}
    \left|\inner{\zeta}{\zeta'}{\Hilbert_2}\right|
    &\leq \left|\inner{\eta}{\eta'}{\Hilbert_2}\right| + \left|\inner{\zeta-\eta}{\eta'}{\Hilbert_2}\right| + \left|\inner{\zeta}{\zeta'-\eta'}{\Hilbert_2}\right| \\
    &\leq \left|\inner{\eta}{\eta'}{\Hilbert_2}\right| + 8 H l^2 \tunnelmodmagnitude{\tau}{\varepsilon}  \\
    &\leq H l^2 \tunnelextent{\tau} + \left|\inner{\xi}{\xi'}{\Hilbert_1}\right| + 8 H l^2 \tunnelextent{\tau} \\
    &\leq 9 H l^2 \tunnelextent{\tau} + \left|\inner{\xi}{\xi'}{\Hilbert_1}\right| \text,
  \end{align*}
  where the last inequality follows from \cite{Latremoliere18c}.

  The reasoning is similar for $\Im$ in place of $\Re$.
\end{proof}
We will use Corollary (\ref{inner-cor}) when dealing with the continuity of multiplicities of eigenvalues of Dirac operators with respect to the spectral propinquity.

\bigskip

Theorems (\ref{id-thm}) and (\ref{sym-thm}) provide the most useful expressions for the spectral propinquity (up to equivalence of metrics), and we now employ these new forms to establish the continuity of the continuous functional calculus for the Dirac operators of metric spectral triples with respect to the spectral propinquity.

\section{Convergence of the bounded continuous functional calculus for metric spectral triples}

We now prove that convergence in the spectral propinquity implies the convergence, in an appropriate sense, of the bounded continuous functional calculus of the associated Dirac operators. Our convergence will use the separation introduced in Definition (\ref{separation-def}) and the operational propinquity in Definition (\ref{oppropinquity-def}). It will be helpful to use the following notation for families when working with these numbers.
\begin{notation}
  We of course regard any $n$-tuple $(a_1,\ldots,a_n)$ as family indexed by $\{1,\ldots,n\}$, and we identify $1$-tuples with their unique member without further mention. We thus can make sense of the separation, dispersion, and operational propinquity between pairs of single operators, in the sense of Definitions (\ref{separation-def}) and (\ref{oppropinquity-def}). More generally, if $A$ is some family indexed by $J$ and $(a_1,\ldots,a_n)$ is some $n$-tuple, then $(A,a_1,\ldots,a_n)$ is meant for the family indexed by the disjoint union $J\coprod\{1,\ldots,n\}$, whose restriction to $J$ is $A$ and whose restriction to $\{1,\ldots,n\}$ is $(a_1,\ldots,a_n)$. 
\end{notation}

As in the previous section, it is helpful to group the framework for this section within one hypothesis, which we will use repeatedly.

\begin{hypothesis}\label{second-working-hyp}
  Let $(\A_n,\Hilbert_n,\Dirac_n)_{n\in\N}$ be a sequence of metric spectral triples, which converge to a metric spectral triple $(\A_\infty,\Hilbert_\infty,\Dirac_\infty)$ for the spectral propinquity. Up to truncating our sequence, we will assume that
  \begin{equation*}
    \spectralpropinquity{\ast}((\A_n,\Hilbert_n,\Dirac_n),(\A_\infty,\Hilbert_\infty,\Dirac_\infty)) < \frac{\sqrt{2}}{2}
  \end{equation*}
  for all $n \in \N$, where $\spectralpropinquity{\ast}$ was defined in Theorem (\ref{id-thm}).

  For each $n\in\N\cup\{\infty\}$, we write $\mcc{\A_n}{\Hilbert_n}{\Dirac_n} = (\Hilbert_n,\CDN_n,\A_n,\Lip_n,\C,0)$. We also set $U_n : t \in \R \mapsto \exp(i t \Dirac_n)$.
  
  Moreover, for each $n\in\N$, let $\mu_n > \spectralpropinquity{\ast}((\A_n,\Hilbert_n,\Dirac_n),(\A_\infty,\Hilbert_\infty,\Dirac_\infty))$ (where $\spectralpropinquity{\ast}$ is defined in Theorem (\ref{id-thm})), and let $\tau_n$ be a tunnel from $(\A_n,\Hilbert_n,\Dirac_n)$ to $(\A_\infty,\Hilbert_\infty,\Dirac_\infty)$, such that
  \begin{itemize}
  \item $\tunnelmodsymmagnitude{\tau_n}{\mu_n} \leq \mu_n$ (where $\tunnelmodsymmagnitude{\tau_n}{\mu_n}$ was defined in Corollary (\ref{sym-cor})),
  \item $\lim_{n\rightarrow\infty} \mu_n = 0$.
  \end{itemize}

  For each $n\in\N$, we write
  \begin{equation*}
    \tunnel{\tau_n}{(\A_n,\Hilbert_n,\Dirac_n)}{(\Pi_n,\pi_n,\theta_n)}{\mathds{J}_n}{(\Psi_n,\psi_n,\sigma_n)}{(\A_\infty,\Hilbert_\infty,\Dirac_\infty)}\text,
  \end{equation*}
  with
  \begin{equation*}
    \mathds{J}_n \coloneqq (\mathscr{J}_n, \TDN_n, \D_n, \TLip_n, \alg{E}_n, \SLip_n) \text.
  \end{equation*}

  It will be convenient to write $C_n = \left[ \frac{-1}{2\mu_n}, \frac{1}{2\mu_n}\right]$, for all $n\in\N$.
\end{hypothesis}

By Theorem (\ref{sym-thm}), we observe the following key property under Hypothesis (\ref{second-working-hyp}):
\begin{equation*}
  \forall n \in \N \quad \tunnelsep{\tau_n}{(U_n^t)_{t\in C_n}, (U_\infty^t)_{t\in C_n}} \leq \tunneldispersion{\tau_n}{(U_n^t)_{t\in C_n}, (U_\infty^t)_{t\in C_n}} \leq \mu_n \text,
\end{equation*}
with $\tunnelsep{\cdot}{\cdot}$ and $\tunneldispersion{\cdot}{\cdot}$ defined in Definition (\ref{separation-def}).

We start with a small and easy digression, as a simple illustration for the more involved methods to follow.

\begin{notation}
  For all $x\in \R$, let $E_x : t \in \R \mapsto \exp(i t x)$; of course $E_x$ is an element of the C*-algebra $C_b(\R)$ of $\C$-valued, bounded, continuous functions over $\R$.
\end{notation}

\begin{proposition}\label{Bohr-prop}
  Assume Hypothesis (\ref{second-working-hyp}). If $f \in C_b(\R)$ is a Bohr almost periodic function, i.e. $f$ is in the norm closure of the linear span of $\{ E_x : x \in \R \}$ in $C_b(\R)$, then,
  \begin{equation*}
   \lim_{n\rightarrow\infty} \tunnelsep{\tau_n}{ ( (U_n^t)_{t\in C_n},f(\Dirac_n)),( (U_\infty^t)_{t\in C_n},f(\Dirac_\infty) ) } = 0 \text.
  \end{equation*}
\end{proposition}

\begin{proof}
  Let $\varepsilon > 0$. There exists a finite subset $F \subseteq \R$ of $\R$, and $v : F \rightarrow \C$, such that $\norm{f - \sum_{x \in F} v(x) E_x}{C_b(X)} < \frac{\varepsilon}{3}$. Write $P = \sum_{x\in F} v(x) E_x$. Let $N \in \N$ such that, for all $n\geq N$,
  \begin{equation*}
    \mu_n < \min\left\{ \varepsilon(3(1+\sum_{x\in F}|v(x)|)^{-1}, (2\max\{|x|,1:x\in F\})^{-1} \right\}\text.
  \end{equation*}
  In particular:
  \begin{equation*}
    \tunnelmodsymmagnitude{\tau_n}{\mu_n} \leq \frac{\varepsilon}{3(1+\sum_{x\in F}|v(x)|)} \text,
  \end{equation*}
  and $F \subseteq \left[-\frac{1}{2\mu_n},\frac{1}{2\mu_n}\right] = C_n$.
    
  Let $\xi \in \dom{\Dirac_n}$ with $\CDN_n(\xi)\leq 1$. By definition of the covariant reach, there exists $\eta\in \dom{\Dirac_\infty}$ such that $\CDN_\infty(\eta)\leq 1$, and, for all $t \in [-\frac{1}{2\mu_n},\frac{1}{2\mu_n}]$, and for all $\omega\in\dom{\TDN_n}$ with $\TDN_n(\omega)\leq 1$, we have
  \begin{equation*}
    \left| \inner{U_n^t \xi}{\Pi_n(\omega)}{\Hilbert_n} - \inner{U_\infty^t\eta}{\Psi_n(\omega)}{\Hilbert_\infty} \right| \leq \mu_n < \frac{\varepsilon}{3(1+\sum_{x\in F}|v(x)|)} \text.
  \end{equation*}
  We emphasize that, for our choice of $\xi$, the choice of $\eta$ is given by the definition of the $\mu_n$-reach, independently of $f$.

  We then compute, for all $\omega\in \mathscr{J}_n$ with $\TDN_n(\omega)\leq 1$, the following sequence of inequalities:
  \begin{align*}
    \Bigg| & \inner{f(\Dirac_n)\xi}{\Pi_n(\omega)}{\Hilbert_n} - \inner{f(\Dirac_\infty)\eta}{\Psi_n(\omega)}{\Hilbert_\infty}\Bigg| \\
           &\leq \Bigg| \inner{(f-P)(\Dirac_n)\xi}{\Pi_n(\omega)}{\Hilbert_n} \Bigg| + \Bigg| \inner{P(\Dirac_n)\xi}{\Pi_n(\omega)}{\Hilbert_n} - \inner{P(\Dirac_\infty)\eta}{\Psi_n(\omega)}{\Hilbert_\infty}\Bigg| \\
           &\quad + \Bigg| \inner{(f-P)(\Dirac_\infty)\xi}{\Psi_n(\omega)}{\Hilbert_\infty} \Bigg| \\
           &\leq \frac{2\varepsilon}{3} + \sum_{x\in F} |v(x)| \left|\inner{U_n^x \xi}{\Pi_n(\omega)}{\Hilbert_n} - \inner{U_\infty^x\eta}{\Psi_n(\omega)}{\Hilbert_\infty}\right| \\
           &\leq \frac{2 \varepsilon}{3} + \sum_{x\in F} |v(x)|\frac{\varepsilon}{3(1+\sum_{x\in F}|v(x)|)} \leq \varepsilon \text.
  \end{align*}

  The proof is similar when we switch the roles of $(\A_n,\Hilbert_n,\Dirac_n)$ and $(\A_\infty,\Hilbert_\infty,\Dirac_\infty)$. This concludes our argument.
\end{proof}

In particular, under Hypothesis (\ref{second-working-hyp}), if $f \in C_b(\R)$ is Bohr almost periodic, Proposition (\ref{Bohr-prop}) proves that $\lim_{n\rightarrow\infty} \oppropinquity{}(f(\Dirac_n),f(\Dirac_\infty)) = 0$ (where $\oppropinquity{}$ is defined in Definition (\ref{oppropinquity-def})). However, the statement given by Proposition (\ref{Bohr-prop}) is stronger, and the uniformity implied by this statement will be a helpful tool in this section.

\bigskip

Our next step is to prove that the restriction of the continuous functional calculus for Dirac operators to $C_0(\R)$ converges in a manner similar to what we established with Proposition (\ref{Bohr-prop}) for Bohr almost periodic functions. We begin by expliciting a simple relation between the continuous calculus, restricted to Schwarz functions, and our mode of convergence for spectral triples.

\begin{definition}\label{Schwarz-space-def}
The \emph{Schwarz space} $\mathcal{S}(\R)$ is the space of all infinitely differentiable functions $f$ from $\R$ to $\C$ such that
\begin{equation*}
  \forall k,n \in \N \quad \lim_{x\rightarrow \pm \infty} (1+|x|^k)f^{(n)}(x) = 0 \text,
\end{equation*}
where $f^{(n)}$ is the $n$-th derivative of $f$ over $\R$.
\end{definition}

Let $f \in \mathcal{S}(\R)$, and let $\widehat{f} : x \in \R \mapsto \int_{\R} f(t) \exp(-i t x) \, dt$ be the Fourier transform of $f$. By \cite[Lemma IX.1]{ReedSimon2}, the function $\widehat{f}$ is in $\mathcal{S}(\R)$, while by \cite[Theorem IX.1]{ReedSimon2}, the following identity holds:
\begin{equation*}
  \forall t \in \R \quad f(t) = \int_\R \widehat{f}(x) \exp(i t x) \, dx \text.
\end{equation*}

Let now $\Dirac$ be a possibly unbounded self-adjoint operator defined from a dense subspace $\dom{\Dirac}$ of a Hilbert space $\Hilbert$, to $\Hilbert$. The \emph{continuous functional calculus} for the self-adjoint operator $\Dirac$ is a *-morphism $\Phi : C_b(\R) \rightarrow \B(\Hilbert)$ from the C*-algebra $C_b(\R)$ of bounded $\C$-valued continuous functions over $\R$, to the C*-algebra $\B(\Hilbert)$ of bounded linear operators on $\Hilbert$, such that in particular, $U^t \coloneqq \Phi(E_{t})$ is a unitary operator of $\B(\Hilbert)$, for all $t \in \R$, with the property that $U^{t+s} = U^t U^s$ for all $t,s\in \R$, and of course, $U^0$ is the identity of $\Hilbert$. The function $t \in \R \mapsto U^t$ is continuous for the strong operator topology, and thus it is in this topology that the following computation is made.

As $\Phi$ is a unital *-morphism, we conclude that, for all $f \in \mathcal{S}(\R)$ and for all $\xi \in \Hilbert$,
\begin{align*}
  f(\Dirac)\xi \coloneqq \Phi(f)\xi
  &= \Phi\left(t\in\R \mapsto \int_\R \widehat{f}(x) \exp(i t x) \, dx\right)\xi \\
  &= \int_\R \widehat{f}(x) \Phi(E_x)\xi \, dx\\
  &= \int_\R \widehat{f}(x) U^x\xi \, dx \\
  &= \left(\int_\R \widehat{f}(x)U^x\, dx\right)  \xi \text.
\end{align*}

\bigskip

We can thus derive the following lemma.
\begin{lemma}\label{cutoff-lemma}
  Let $(\A,\Hilbert,\Dirac)$ be a spectral triple, and, for all $t \in \R$, let $U^t \coloneqq \exp(it\Dirac)$. If $f \in \mathcal{S}(\R)$, then for all $\varepsilon > 0$, there exists $M > 0$ such that, if $K\geq M$, then
  \begin{equation*}
    \opnorm{f(\Dirac)-\int_{[-K,K]} \widehat{f}(t) U^t \, dt}{}{\Hilbert} < \varepsilon \text.
  \end{equation*}
\end{lemma}

\begin{proof}
  Let $\varepsilon > 0$. By Definition (\ref{Schwarz-space-def}), there exists $M > 0$ such that, if $|x| > M$, then $|\widehat{f}(x)|\leq \frac{\varepsilon}{\pi(1+x^2)}$. Therefore, if $K\geq M$, then, for all $\xi \in \Hilbert$ with $\norm{\xi}{\Hilbert}\leq 1$:
\begin{align*}
  \norm{f(\Dirac)\xi-\int_{[-K,K]} \widehat{f}(t) U^t\xi \, dt}{\Hilbert}
  &= \norm{\int_{\{ x\in \R : |x| > K \}} \widehat{f}(t) U^t\xi \, dt}{\Hilbert} \\
  &\leq \int_{\{ x\in\R: |x|>M \}} |\widehat{f}(t)| \opnorm{U^t}{}{\Hilbert} \, dt \\
  &\leq \frac{\varepsilon}{\pi} \int_{\{ x\in\R : |x|>M \}} \frac{dt}{1+t^2} \\
  &\leq \frac{\varepsilon}{\pi} \left(\frac{\pi}{2}-\arctan(M) + \arctan(-M) - (-\frac{\pi}{2}) \right) \\
  &=\frac{\varepsilon}{\pi}\left(\pi-2\arctan(M)\right) \leq \frac{\varepsilon}{\pi} \pi = \varepsilon\text,
\end{align*}
noting that, since $M>0$, we used the fact that $\arctan(M) > 0$ as well. Our lemma follows.
\end{proof}

We thus establish the continuity of the restriction to $C_0(\R)$ of the continuous functional calculus for spectral triples, with respect to the operational propinquity.

\begin{theorem}\label{C0-thm}
  If we assume Hypothesis (\ref{second-working-hyp}), then
  \begin{equation*}
    \forall f \in C_0(\R) \quad \lim_{n\rightarrow\infty} \tunnelsep{\tau_n}{((U_n^t)_{t\in C_n},f(\Dirac_n)),((U_\infty^t)_{t \in C_n},f(\Dirac_\infty))} = 0 \text.
  \end{equation*}
  In particular, $\lim_{n\rightarrow\infty} \oppropinquity{}(f(\Dirac_n),f(\Dirac_\infty)) = 0 $ for all $f \in C_0(\R)$, in the sense of Definition (\ref{oppropinquity-def}).
\end{theorem}

\begin{proof}
  Let  $f \in \mathcal{S}(\R)$. Let $\varepsilon \in (0,1)$. Let $M$ be given by Lemma (\ref{cutoff-lemma}) for $\frac{\varepsilon}{3}$ in place of $\varepsilon$; note that we can assume $M\geq 1$ with no loss of generality. Let $K = \max\left\{\norm{\widehat{f}}{C_0(\R)},1 \right\}$, where
  \begin{equation*}
    \widehat{f}:x\in \R \mapsto \int_{\R} f(t) \exp(-it x) \, dt \text.
  \end{equation*}

By assumption, there exists $N\in\N$ such that, if $n\geq N$, then $\mu_n < \frac{\varepsilon}{6 M K}$, so $\tunnelmodsymmagnitude{\tau_n}{\mu_n} < \frac{\varepsilon}{6 M K}$ and, in particular:
\begin{equation*}
  \spectralpropinquity{\ast}((\A_n,\Hilbert_n,\Dirac_n),(\A_\infty,\Hilbert_\infty,\Dirac_\infty)) < \mu_n < \frac{\varepsilon}{6 M K} \text.
\end{equation*}

Let $\xi \in \Hilbert_n$ with $\CDN_n(\xi)\leq 1$. By Corollary (\ref{sym-cor}), there exists $\eta_n \in \Hilbert_\infty$ with $\CDN_\infty(\eta_n)\leq 1$, such that, for all $\omega\in\dom{\TDN_n} \subseteq\mathscr{J}_n$ with $\TDN_n(\omega)\leq 1$, and for all $t \in \left[-\frac{3 M K}{\varepsilon},\frac{3 M K}{\varepsilon}\right]$, we have
\begin{equation*}
  \left|\inner{U_n^t\xi}{\Pi_n(\omega)}{\Hilbert_n} - \inner{U_\infty^t\eta_n}{\Psi_n(\omega)}{\Hilbert_\infty}\right| < \frac{\varepsilon}{6MK}  \text.
\end{equation*}

We emphasize that, for any $\xi$ as above, the choice of $\eta_n$ is from the definition of the reach, and does not depend on $f$ --- so far, we only used $f$ to decide on $N$. We also note that $\frac{\varepsilon}{6 M K} \leq \varepsilon$.

By construction, $[-M,M]\subseteq \left[-\frac{3 M K}{\varepsilon},\frac{3 M K}{\varepsilon}\right]$ (since $K\geq 1$, $\varepsilon \in (0,1)$, so $\frac{3K}{\varepsilon}>1$). Therefore:
\begin{align*}
  \Big| &\inner{f(\Dirac_n)\xi}{\Pi_n(\omega)}{\Hilbert_n} - \inner{f(\Dirac_\infty)\eta_n}{\Psi_n(\omega)}{\Hilbert_\infty}\Big| \\
  &\leq \left| \inner{\left(f(\Dirac_n)-\int_{[-M,M]} \widehat{f}(t) U_n^t \, dt\right)\xi}{\Pi_n(\omega)}{\Hilbert_n} \right| \\
  &\quad + \left| \inner{\int_{[-M,M]} \widehat{f}(t) U_n^t \xi \, dt}{\Pi_n(\omega)}{\Hilbert_n} - \inner{\int_{[-M,M]} \widehat{f}(t) U_\infty^t \eta_n \, dt}{\Psi_n(\omega)}{\Hilbert_\infty}\right| \\
  &\quad + \left| \inner{\left(\int_{[-M,M]} \widehat{f}(t) U_\infty^t \, dt - f(\Dirac_\infty)\right)\eta_n}{\Psi_n(\omega)}{\Hilbert_\infty} \right| \\
  &\leq \frac{\varepsilon}{3} + \int_{[-M,M]} \left|\widehat{f}(t)\right| \left|\inner{U_n^t\xi}{\Pi_n(\omega)}{\Hilbert_n} - \inner{U^t\eta_n}{\Psi_n(\omega)}{\Hilbert_\infty}\right| \, dt + \frac{\varepsilon}{3} \\
  &\leq \frac{\varepsilon}{2} + 2 M K \frac{\varepsilon}{6 M K} + \frac{\varepsilon}{3} = \varepsilon \text.
\end{align*}

The proof above is similar when the roles of $\Hilbert_n$ (for $n\geq N$) and $\Hilbert_\infty$ are reversed. Thus, for all $n\geq N$, we have $\tunnelsep{\tau_n}{((U_n^t)_{t\in C_n},f(\Dirac_n)),((U_\infty^t)_{t\in C_n},f(\Dirac_\infty))} < \varepsilon$, as claimed, in the special case where $f \in \mathcal{S}(\R)$.

Now, let $f \in C_0(\R)$. Since the space $\mathcal{S}(\R)$ is dense in $C_0(\R)$, there exists $g \in \mathcal{S}(\R)$ such that $\norm{f-g}{C_0(X)} < \frac{\varepsilon}{3}$. Since $g \in \mathcal{S}(\R)$, using the work done in the first part of this proof, there exists $N\in\N$ such that, if $n\geq N$, then
\begin{equation*}
  \tunnelsep{\tau_n}{(((U_n^t)_{t\in C_n},g(\Dirac_n)),(U_\infty^t)_{t\in C_n},g(\Dirac_\infty))} < \frac{\varepsilon}{3}\text.
\end{equation*}

Let $n\geq N$. As the functional calculus is a *-morphism, we then conclude that, for all $\xi \in \Hilbert_n$ with $\CDN_n(\xi)\leq 1$, for all $\eta\in \Hilbert_\infty$ with $\CDN_\infty(\eta)\leq 1$, and for all $\omega\in\mathscr{J}$ with $\TDN_n(\omega)\leq 1$, we have:
\begin{align}\label{Schwarz-to-C0-eq}
  \Big| &\inner{f(\Dirac_n)\xi}{\Pi_n(\omega)}{\Hilbert_n} - \inner{f(\Dirac_\infty)\eta}{\Psi_n(\omega)}{\Hilbert_\infty}\Big| \\
        &\leq \left|\inner{(f(\Dirac_n)-g(\Dirac_n))\xi}{\Pi_n(\omega)}{\Hilbert_n}\right| \nonumber \\
        &\quad + \left|\inner{g(\Dirac_n)\xi}{\Pi_n(\omega)}{\Hilbert_n} - \inner{g(\Dirac_\infty)\eta}{\Psi_n(\omega)}{\Hilbert_\infty}\right| \nonumber \\
        &\quad + \left|\inner{(g(\Dirac_\infty)-f(\Dirac_\infty))\eta}{\Psi_n(\omega)}{\Hilbert_\infty}\right| \nonumber \\
        &\leq 2\norm{f-g}{C_0(X)} + \left|\inner{g(\Dirac_n)\xi}{\Pi_n(\omega)}{\Hilbert_n} - \inner{g(\Dirac_\infty)\eta}{\Psi_n(\omega)}{\Hilbert_\infty}\right| \nonumber \\
  &\leq \frac{2\varepsilon}{3} + \left|\inner{g(\Dirac_n)\xi}{\Pi_n(\omega)}{\Hilbert_n} - \inner{g(\Dirac_\infty)\eta}{\Psi_n(\omega)}{\Hilbert_\infty}\right| \text. \nonumber
\end{align}

By our choice of $g$ and $N$, if $n\geq N$, and if $\xi\in\Hilbert_n$, with $\CDN_n(\xi)\leq 1$, then there exists $\eta\in\Hilbert_\infty$ with $\CDN_\infty(\eta) \leq 1$, such that
\begin{equation*}
  \forall t \in C_n \quad \left|\inner{U_n^t\xi}{\Pi_n(\omega)}{\Hilbert_n} - \inner{U_\infty^t \eta}{\Psi_n(\omega)}{\Hilbert_\infty}\right| < \frac{\varepsilon}{3}\text,
\end{equation*}
and
\begin{equation*}
  \left|\inner{g(\Dirac_n)\xi}{\Pi_n(\omega)}{\Hilbert_n} - \inner{g(\Dirac_\infty)\eta}{\Psi_n(\omega)}{\Hilbert_\infty}\right| < \frac{\varepsilon}{3}\text.
\end{equation*}
Thus, using Equation (\ref{Schwarz-to-C0-eq}), we conclude that
\begin{equation*}
  \left| \inner{f(\Dirac_n)\xi}{\Pi_n(\omega)}{\Hilbert_n} - \inner{f(\Dirac_\infty)\eta}{\Psi_n(\omega)}{\Hilbert_\infty}\right| < \varepsilon \text.
\end{equation*}

A similar argument holds with the roles of $\Hilbert_n$ and $\Hilbert_\infty$ switched. Thus, our result holds for $f \in C_0(\R)$, i.e.
\begin{equation*}
  \lim_{n\rightarrow\infty} \tunnelsep{\tau_n}{((U_n^t)_{t\in C_n},f(\Dirac_n)),((U_\infty^t)_{t\in C_n},f(\Dirac_\infty))} = 0 \text.
\end{equation*}

Therefore, our theorem is proven.
\end{proof}

As an application, for instance, we can state a result about the convergence of the resolvent of Dirac operators implied by the convergence of spectral triples for the spectral propinquity.

\begin{corollary}
  Let $\lambda \in \C\setminus\R$. If $(\A_n,\Hilbert_n,\Dirac_n)_{n\in\N}$ converges to $(\A_\infty,\Hilbert_\infty,\Dirac_\infty)$ for the spectral propinquity, and if $\resolvent{\Dirac_n}{\lambda} \coloneqq (\Dirac_n - \lambda)^{-1}$ is the resolvent of $\Dirac_n$ at $\lambda$ for all $n\in\N\cup\{\infty\}$, then
  \begin{equation*}
    \lim_{n\rightarrow\infty} \oppropinquity{}(\resolvent{\Dirac_n}{\lambda},\resolvent{\Dirac_\infty}{\lambda}) = 0\text.
  \end{equation*}
\end{corollary}

\begin{proof}
  This is an immediate corollary of Theorem (\ref{C0-thm}), with $f : t \in \R \mapsto \frac{1}{t - \lambda} \in C_0(\R)$.
\end{proof}

An important consequence of Theorem (\ref{C0-thm}) will be proven in the next section, and is arguably the core result of this paper: namely, the continuity result about the spectra of metric spectral triples converging for the spectral propinquity. 

Before we turn to the spectrum, however, we wish to complete our study of the relation between the continuous functional calculus for spectral triples and the spectral propinquity. We have, so far, shown continuity of the bounded continuous functional calculus restricted to $C_0(\R)$ and to the Bohr almost periodic functions. We now want to extend Theorem (\ref{C0-thm}) to the whole bounded, continuous functional calculus. This is somewhat more involved, and this is the subject of the following theorem. The reader who wishes to see our result on the continuity of spectra can skip this result, which will not be used for that purpose.

\begin{theorem}\label{Cb-thm}
  If $(\A_n,\Hilbert_n,\Dirac_n)_{n\in \N}$ is a sequence of metric spectral triples, converging to a metric spectral triple $(\A_\infty,\Hilbert_\infty,\Dirac_\infty)$ for the spectral propinquity, and if $f \in C_b(\R)$ is a bounded, $\C$-valued continuous function over $\R$, then
  \begin{equation*}
    \lim_{n\rightarrow\infty} \oppropinquity{}(f(\Dirac_n),f(\Dirac_\infty)) = 0\text,
  \end{equation*}
  where $\oppropinquity{}$ is defined in Definition (\ref{oppropinquity-def}).
\end{theorem}

\begin{proof}
  Once more, we will use the notation of Hypothesis (\ref{second-working-hyp}).
  
  Let $f \in C_b(\R)$. Our result is obvious if $f = 0$, so we assume $f\neq 0$, and in fact, without loss of generality, we assume that $\norm{f}{C_b(\R)} = 1$. Let $\varepsilon > 0$. Let $(v_k)_{k\in\N}$ be a sequence of non-negative functions in $C_0(\R)$ such that, for all $k\in\N$, we have $0\leq v_k \leq 1$, and $(v_k)_{k\in\N}$ converges pointwise to $1$ on $\R$.

  First, let $\xi \in \dom{\CDN_\infty}$ with $\CDN_\infty(\xi)\leq 1$. Using the Borel functional calculus for the self-adjoint operator $\Dirac_\infty$ \cite[Theorem VIII.5, d]{ReedSimon}, the sequence $(v_k(\Dirac_\infty)\xi)_{k\in\N}$ converges, in the norm topology of $\Hilbert_\infty$, to $\xi$. Therefore, there exists $K\in\N$ such that, if $k\geq K$, then $\norm{v_k(\Dirac_\infty)\xi-\xi}{\Hilbert_\infty} < \frac{\varepsilon}{2}$.

  Now, the function $g \coloneqq v_K f$ lies in $C_0(\R)$. Therefore, by Theorem (\ref{C0-thm}), there exists $N_1\in \N$ such that, if $n\geq N_1$, then
  \begin{equation*}
    \max_{h\in\{g,v_K\}}\tunnelsep{\tau_n}{((U_n^t)_{t\in C_n},h(\Dirac_n)),((U_\infty^t)_{t\in C_n},h(\Dirac_\infty))} \leq \frac{\varepsilon}{2} \text.
  \end{equation*}
  
  Let $n\geq N_1$. Noting that $U_n^0$ is the identity of $\Hilbert_n$, there exists $\eta_n\in\dom{\Dirac_n}$ such that
  \begin{equation*}
    \max_{h \in \{1,g,v_K\}}\sup_{\substack{\omega\in\dom{\TDN_n} \\ \TDN_n(\omega)\leq 1}}\left|\inner{h(\Dirac_n)\eta_n}{\Pi_n(\omega)}{\Hilbert_n} - \inner{h(\Dirac_\infty)\xi}{\Psi_n(\omega)}{\Hilbert_\infty}\right| < \frac{\varepsilon}{2} \text.
  \end{equation*}

  Let $\omega\in\dom{\TDN_n}$, with $\TDN_n(\omega)\leq 1$. We then compute:
  \begin{align*}
    \Big| & \inner{f(\Dirac_n)(v_K(\Dirac_n))\eta_n)}{\Pi_n(\omega)}{\Hilbert_n} - \inner{f(\Dirac_\infty)\xi}{\Psi_n(\omega)}{\Hilbert_\infty}\Big| \\
          &\leq \left|\inner{g(\Dirac_n)\eta_n}{\Pi_n(\omega)}{\Hilbert_n} - \inner{g(\Dirac_\infty)\xi}{\Psi_n(\omega)}{\Hilbert_\infty}\right| \\
          &\quad + \norm{f(1-v_K)(\Dirac_\infty)\xi}{\Hilbert_\infty} \norm{\omega}{\mathscr{J}_n} \\
          &\leq \frac{\varepsilon}{2} + \norm{f}{C_b(\R)}\norm{\xi-v_K(\Dirac_\infty)\xi}{\Hilbert_\infty} \leq \varepsilon \text.
  \end{align*}
  Similarly,
  \begin{equation*}
    \left|\inner{v_K(\Dirac_n)\eta_n}{\Pi_n(\omega)}{\Hilbert_n} - \inner{\xi}{\Psi_n(\omega)}{\Hilbert_\infty}\right| \leq \varepsilon \text.
  \end{equation*}

  On the other hand, $\Dirac_nv_K(\Dirac_n)(\eta_n) = v_K(\Dirac_n)\Dirac_n\eta_n$, so
  \begin{equation*}
    \CDN_n(v_K\eta_n) \leq \norm{v_K}{C_0(\R)} \CDN_n(\eta_n) \leq 1 \text.
  \end{equation*}

  Therefore, we have shown that, for all $\varepsilon >0$, there exists $N\in\N$ such that, if $n\geq N$, for all $\xi \in \dom{\Dirac_\infty}$ with $\CDN_\infty(\xi)\leq 1$, there exist $\zeta_n \in \dom{\Dirac_n}$ with $\CDN_n(\zeta_n)\leq 1$, such that
  \begin{equation}\label{Cb-thm-eq-0-0}
    \sup_{\substack{\omega\in\dom{\TDN_n}\\ \TDN_n(\omega)\leq 1}} \left|\inner{f(\Dirac_\infty)\xi}{\Pi_n(\omega)}{\Hilbert_\infty} - \inner{f(\Dirac_n)\zeta_n}{\Psi_n(\omega)}{\Hilbert_n}\right| < \varepsilon \text,
  \end{equation}
  and
  \begin{equation}\label{Cb-thm-eq-0-1}
    \sup_{\substack{\omega\in\dom{\TDN_n}\\ \TDN_n(\omega)\leq 1}} \left|\inner{\xi}{\Pi_n(\omega)}{\Hilbert_\infty} - \inner{\zeta_n}{\Psi_n(\omega)}{\Hilbert_n}\right| < \varepsilon \text.
  \end{equation}

  The proof of the other inequality is more complicated.

  Let $\mathscr{B} = \left\{ \xi \in \dom{\Dirac_\infty} : \CDN_\infty(\xi) \leq 1 \right\}$. By Definition (\ref{mcc-def}), the set $\mathscr{B}$ is compact in $\Hilbert_\infty$. Let $F \subseteq \mathscr{B}$ be a finite, $\frac{\varepsilon}{7}$-dense subset of $\mathscr{B}$. Since $F$ is finite, there exists $K \in \N$ such that, if $k\geq K$, then
  \begin{equation}\label{vk-choice-eq}
    \max_{\xi \in F} \norm{v_k(\Dirac_\infty) \xi - \xi}{\Hilbert_\infty} < \frac{\varepsilon}{7} \text,
  \end{equation}
  again, since $(v_k(\Dirac_\infty))_{k\in\N}$ converges, in the strong operator topology, to the identity.

  Once again, we write $g\coloneqq v_K f \in C_0(\R)$.

  Now, by assumption, there exists $N_2 \in \N$ such that, if $n\geq N_2$, then $\mu_n \leq \frac{\varepsilon}{7}$. By Theorem (\ref{C0-thm}), there exists $N_3\in \N$ such that, if $n\geq N_3$, we have
  \begin{equation}\label{Cb-thm-eq-1}
    \tunnelsep{\tau_n}{(U_n^0,v_K(\Dirac_n),g(\Dirac_n)), (U_\infty^0,v_K(\Dirac_\infty),g(\Dirac_\infty))} < \frac{\varepsilon}{7} \text.
  \end{equation}

  Let $n \geq \max\{N_2,N_3\}$. Denote by $1$ the constant function equal to $1$ on $\R$, so that $1(\Dirac_n)$ is the identity $U_n^0$ of $\Hilbert_n$, for all $n \in \N\cup\{\infty\}$.

  Therefore, for each $\xi \in F$, by Expression (\ref{Cb-thm-eq-1}), there exists $t(\xi) \in \dom{\Dirac_n}$ with $\CDN_n(t(\xi)) \leq 1$, and
  \begin{equation}\label{Cb-thm-eq-2}
    \sup_{\substack{h \in \{1,g,v_K\} \\ \omega\in\dom{\TDN_n} \\ \TDN_n(\omega)\leq 1}} \max\Bigg\{ \left| \inner{h(\Dirac_\infty)\xi}{\Psi_n(\omega)}{\Hilbert_\infty} - \inner{h(\Dirac_n)t(\xi)}{\Pi_n(\omega)}{\Hilbert_n}\right| \Bigg\}
    < \mu_n < \frac{\varepsilon}{7} \text.
  \end{equation}
  We now prove that $t(F)$ is $\frac{3\varepsilon}{7}$-dense in the closed unit ball of $\CDN_n$ for the following metric, introduced in \cite[Definition 3.24]{Latremoliere16c}. For any $\omega,\omega'\in \Hilbert_n$, we set
  \begin{equation*}
    \KantorovichMod{\CDN_n}(\omega,\omega')\coloneqq \sup\left\{ \norm{\omega-\omega'}{\Hilbert_n} : \CDN_n(\omega)\leq 1 \right\} \text.
  \end{equation*}

  Let now $\eta\in\dom{\Dirac_n}$ with $\CDN_n(\eta)\leq 1$.  Let $r(\eta) \in \dom{\Dirac_\infty}$ with $\CDN_\infty(r(\eta))\leq 1$, and
  \begin{equation*}
    \sup_{\substack{h \in \{1,g,v_K\} \\ \omega\in\dom{\TDN_n} \\ \TDN_n(\omega)\leq 1}} \left| \inner{h(\Dirac_\infty) r(\eta)}{\Psi_n(\omega)}{\Hilbert_\infty} - \inner{h(\Dirac_n) \eta}{\Pi_n(\omega)}{\Hilbert_n}\right| < \mu_n < \frac{\varepsilon}{7} \text.
  \end{equation*}

  Let $\xi \in F$ such that $\norm{r(\eta)-\xi}{\Hilbert_\infty} < \frac{\varepsilon}{7}$. We then compute, for all $\omega\in\dom{\TDN_n}$ with $\TDN_n(\omega)\leq 1$,
  \begin{align}
    \Big|& \inner{t(\xi)-\eta}{\Pi_n(\omega)}{\Hilbert_n}\Big| \nonumber \\
          &\leq \left|\inner{t(\xi)}{\Pi_n(\omega)}{\Hilbert_n}-\inner{\xi}{\Psi_n(\omega)}{\Hilbert_\infty}\right|  \nonumber \\
          &\quad + \left|\inner{\eta}{\Pi_n(\omega)}{\Hilbert_n}-\inner{r(\eta)}{\Psi_n(\omega)}{\Hilbert_\infty} \right| \nonumber \\
          &\quad + \left|\inner{\xi-r(\eta)}{\Psi_n(\omega)}{\Hilbert_\infty}\right| \nonumber \\
          &\leq 2\frac{\varepsilon}{7} + \norm{\xi-r(\eta)}{\Hilbert_\infty} 
          \leq \frac{3}{7}\varepsilon \text. \label{tximinuseta-eq}
  \end{align}

  Thus, $t(F)$ is $\frac{3\varepsilon}{7}$-dense in $\left\{\zeta\in\dom{\Dirac_n} : \CDN_n(\zeta)\leq 1\right\}$ for the metric $\KantorovichMod{\CDN_n}$ (since by definition of $\Pi_n$, $\zeta\in\dom{\CDN_n}$ and $\CDN_n(\zeta)\leq 1$ if, and only if, there exists $\omega \in \dom{\TDN_n}$ with $\TDN_n(\omega)\leq 1$ and $\Pi_n(\omega)=\zeta$).

  Let now $\eta\in\dom{\Dirac_n}$, with $\CDN_n(\eta)\leq 1$. There exists $\xi \in F$ such that $\KantorovichMod{\CDN_n}(t(\xi),\eta) < \frac{3\varepsilon}{7}$. By construction, for all $\omega\in\dom{\TDN_n}$ with $\TDN_n(\omega)\leq 1$,
  \begin{equation*}
    \sup_{h\in\{1,g,v_K\}}\left|\inner{h(\Dirac_n)t(\xi)}{\Pi_n(\omega)}{\Hilbert_n} - \inner{h(\Dirac_\infty)\xi}{\Psi_n(\omega)}{\Hilbert_\infty}\right| \leq \frac{\varepsilon}{7} \text,
  \end{equation*}
  so, in particular, using Expression (\ref{tximinuseta-eq}),
  \begin{equation}\label{Cb-thm-eq-3}
    \left|\inner{\eta}{\Pi_n(\omega)}{\Hilbert_n} - \inner{\xi}{\Psi_n(\omega)}{\Hilbert_\infty}\right| \leq \frac{4\varepsilon}{7} \text.
  \end{equation}

  Now, we note that if $\omega\in\dom{\TDN_n}$ with $\TDN_n(\omega)\leq 1$, then
  \begin{align*}
    \CDN_n(\overline{f}(\Dirac_n)\Pi_n(\omega))
    &= \norm{\overline{f}(\Dirac_n)\Pi_n(\omega)}{\Hilbert_n} + \norm{\Dirac_n\overline{f}(\Dirac_n)\Pi_n(\omega)}{\Hilbert_n} \\
    &= \norm{\overline{f}(\Dirac_n)\Pi_n(\omega)}{\Hilbert_n} + \norm{\overline{f}(\Dirac_n)\Dirac_n\Pi_n(\omega)}{\Hilbert_n} \\
    &\leq \norm{f}{C_0(\R)} \CDN_n(\Pi_n(\omega)) \leq \norm{f}{C_0(\R)} \leq 1 \text.
  \end{align*}

  Therefore, there exists $\zeta\in\dom{\TDN_n}$ with $\TDN_n(\zeta)\leq 1$ and $\Pi_n(\zeta) = f(\Dirac_n)\Pi_n(\omega)$ as $\Pi_n$ is a modular quantum isometry. Therefore:
  \begin{align}
    \left|\inner{f(\Dirac_n)(\eta-t(\xi))}{\Pi_n(\omega)}{\Hilbert_n}\right|
    &=\left|\inner{\eta-t(\xi)}{\overline{f}(\Dirac_n)\Pi_n(\omega)}{\Hilbert_n}\right| \label{fD-eta-txi-eq} \\
    &= \left|\inner{\eta-t(\xi)}{\Pi_n(\zeta)}{\Hilbert_n}\right| \nonumber \\
    &\leq \frac{3 \varepsilon}{7} \text{ by Eq. (\ref{tximinuseta-eq}).} \nonumber
  \end{align}

  Similarly, and keeping the same notation, we also note that:
  \begin{align}
    \Big| &\inner{(f(\Dirac_n)-v_Kf(\Dirac_n)) t(\xi)}{\Pi_n(\omega)}{\Hilbert_n}\Big| \label{f-fvk-txi-eq}\\
          &= \Big| \inner{(1-v_K)(\Dirac_n) t(\xi)}{\Pi_n(\zeta)}{\Hilbert_n}\Big| \nonumber \\
          &\leq \frac{2\varepsilon}{7} + \Big| \inner{(1-v_K)(\Dirac_\infty)\xi}{\Psi_n(\zeta)}{\Hilbert_\infty}\Big| \text{ using Eq. (\ref{Cb-thm-eq-2}),} \nonumber \\
          &\leq \frac{2\varepsilon}{7} + \frac{\varepsilon}{7} = \frac{3\varepsilon}{7} \text{ using Eq. (\ref{vk-choice-eq}).} \nonumber
  \end{align}
  
  We now conclude, for all $\omega\in\dom{\TDN_n}$ with $\TDN_n(\omega)\leq 1$,
  \begin{align*}
    \Big| &\inner{f(\Dirac_n)\eta}{\Pi_n(\omega)}{\Hilbert_n} - \inner{fv_K (\Dirac_\infty)\xi}{\Psi_n(\omega)}{\Hilbert_\infty} \Big| \\
          &\leq \left|\inner{f(\Dirac_n)(\eta-t(\xi))}{\Pi_n(\omega)}{\Hilbert_n}\right| \\
          &\quad+ \Big| \inner{f(\Dirac_n)t(\xi)}{\Pi_n(\omega)}{\Hilbert_n} - \inner{f v_K(\Dirac_\infty)\xi}{\Psi_n(\omega)}{\Hilbert_\infty} \Big| \\
          &\leq \underbracket[1pt]{\frac{3 \varepsilon}{7}}_{\text{by Eq. (\ref{fD-eta-txi-eq})}} + \Big| \inner{(f(\Dirac_n)-v_Kf(\Dirac_n)) t(\xi)}{\Pi_n(\omega)}{\Hilbert_n}\Big|  \\
          &\quad + \Bigg| \inner{\underbracket[1pt]{f v_K (\Dirac_n)t(\xi)}_{=g(\Dirac_n)t(\xi)}}{\Pi_n(\omega)}{\Hilbert_n} - \inner{\underbracket[1pt]{f v_K(\Dirac_\infty)\xi}_{=g(\Dirac_\infty)\xi}}{\Psi_n(\omega)}{\Hilbert_\infty} \Bigg| \\
          &\leq \frac{3\varepsilon}{7} + \underbracket[1pt]{\frac{3\varepsilon}{7}}_{\text{by Eq. (\ref{f-fvk-txi-eq})}}  + \underbracket[1pt]{\frac{\varepsilon}{7}}_{\text{by Eq.  (\ref{Cb-thm-eq-2})}} = \varepsilon \text.
  \end{align*}

  Again, we note that $f v_K(\Dirac_\infty) \xi = f(\Dirac_\infty) (v_K(\Dirac_\infty)\xi)$ and $\CDN_\infty(v_K(\Dirac_\infty)\xi) \leq 1$.

  Hence, we now have shown that, if $n\geq \max\{N_2,N_3\}$, then
  \begin{multline*}
    \sup_{\substack{\eta\in\dom{\CDN_n}\\ \CDN_n(\eta)\leq 1}} \inf_{\substack{\xi \in \dom{\CDN_\infty} \\ \CDN_\infty(\xi)\leq 1}} \sup_{\substack{\omega\in\dom{\TDN_n} \\ \TDN_n(\omega)\leq 1}} \\ \left| \inner{f(\Dirac_n)\eta}{\Psi_n(\omega)}{\Hilbert_n} - \inner{f(\Dirac_\infty)\xi}{\Pi_n(\omega)}{\Hilbert_\infty} \right| \leq \varepsilon \text.
  \end{multline*}

  Together with Expression (\ref{Cb-thm-eq-3}), we thus have proven that, for all $n \geq \max\{N_2,N_3\}$:
  \begin{multline}\label{Cb-thm-eq-4}
    \sup_{\substack{\eta\in\dom{\CDN_n}\\ \CDN_n(\eta)\leq 1}} \inf_{\substack{\xi \in \dom{\CDN_\infty} \\ \CDN_\infty(\xi)\leq 1}} \sup_{\substack{h\in\{1,f\} \\ \omega\in\dom{\TDN_n} \\ \TDN_n(\omega)\leq 1}} \\ \left| \inner{h(\Dirac_n)\eta}{\Psi_n(\omega)}{\Hilbert_n} - \inner{h(\Dirac_\infty)\xi}{\Pi_n(\omega)}{\Hilbert_\infty} \right| \leq \varepsilon \text.
  \end{multline}

  Our theorem is thus proven by choosing $n\geq \max\{N_1,N_2,N_3\}$, thanks to Expressions (\ref{Cb-thm-eq-0-0}), (\ref{Cb-thm-eq-0-1}) and (\ref{Cb-thm-eq-4}).
\end{proof}

\bigskip

We record that, in general, and unsurprisingly, there is no extension of Theorem (\ref{Cb-thm}) to the Borel functional calculus.
\begin{example}
  Let $\Hilbert = \ell^2(\Z)$, and, for each $m\in\N$, let $\Dirac_m$ be the closure of the linear operator which sends $e_n$ to $\left(n+\frac{1}{m+1}\right) e_n$ for all $n\in\Z$, where:
  \begin{equation*}
    e_n(z) = \begin{cases} 1 \text{ if $z=n$,} \\ 0 \text{ otherwise.} \end{cases}
  \end{equation*}
  Similarly, let $\Dirac$ be the closure of the operator defined by sending $e_n$ to $n e_n$, for all $n\in\Z$.

  Denoting $\T = \left\{ z \in \C : |z| = 1 \right\}$, we note that $(C(\T),\Hilbert,\Dirac_m)$ is a metric spectral triple on $C(\T)$, which converges to $(C(\T),\Hilbert,\Dirac)$ for the spectral propinquity, as shown in \cite[Section 5.1]{Latremoliere18g}.

  Now, let
  \begin{equation*}
    f : t \in \R \mapsto \begin{cases}
      1 \text{ if $t = 1$, }\\
      0 \text{ otherwise.}
    \end{cases}
  \end{equation*}

  Now, $f(\Dirac)$ is the projection on $e_1$, yet $f(\Dirac_m) = 0$ for all $m \in \N$. Note that $\norm{e_1}{\Hilbert} + \norm{\Dirac e_1}{\Hilbert} = 2$, and $\inner{f(\Dirac) e_1}{e_1}{\Hilbert} = 1$, yet, $\inner{f(\Dirac_m)\eta}{\xi}{\Hilbert} = 0$ for all $\eta,\xi \in \Hilbert$. So it is impossible, no matter what tunnel we may choose, to get $\lim_{n\rightarrow\infty} \oppropinquity{}(f(\Dirac_n),f(\Dirac)) = 0$.
\end{example}

\bigskip

We conclude this section by involving the *-representation of the C*-algebra in a spectral triple. Specifically, we prove the following. We will use the notion of target sets, from Definition (\ref{targetset-def}).

\begin{theorem}
  Assume Hypothesis (\ref{second-working-hyp}). If $f \in C_b(\R)$, if $a\in \A_\infty$ with $\Lip_\infty(a) < \infty$, if $l \geq \Lip_\infty(a)$, and if, for each $n \in \N$, we let $b_n \in \targetsettunnel{\tau_n}{a}{l}$, then
  \begin{equation*}
    \lim_{n\rightarrow\infty} \oppropinquity{}(b_n f(\Dirac_n), a f(\Dirac_\infty) ) = 0
  \end{equation*}
  and
  \begin{equation*}
    \lim_{n\rightarrow\infty} \oppropinquity{}(f(\Dirac_n) b_n, f(\Dirac_\infty) a) = 0 \text.
  \end{equation*}
\end{theorem}

\begin{proof}
  Let $d_n \in \dom{\TLip_{n}}$ with $\TLip_{n}(d_n)\leq l$, $\Psi_n(d_n) = a$ and $\Pi_n(d_n) = b_n$. Since $\Psi_n$ and $\Pi_n$ are morphism maps, we then have, for all $\omega\in\module{J}_n$,
  \begin{equation*}
    a\Pi_n(\omega) = \Pi_n(d_n \omega) \text{ and }b_n\Psi_n(\omega) = \Psi_n(d_n \omega) \text.
  \end{equation*}
  By \cite{Latremoliere13b,Latremoliere14,Latremoliere15}, we note that:
  \begin{equation*}
    \norm{d_n}{\D_n} \leq \norm{a}{\A_\infty} + \tunnelextent{\tau_n} l \text.
  \end{equation*}
  
  Therefore, by the modular Leibniz property,  for all $\omega\in\dom{\TDN_n}$ with $\TDN_n(\omega)\leq 1$, the following holds
  \begin{align*}
    \TDN_n(d_n\omega_n)
    &\leq G(\norm{d_n}{\D} + \TLip_n(d_n)) \TDN_n(\omega) \\
    &\leq G(\norm{a}{\A_\infty} + l \tunnelextent{\tau_n} + l) \\
    &\leq G(\norm{a}{\A_\infty} + l(\tunnelextent{\tau_n}+1)) \text.
  \end{align*}
  Let $k > G(\norm{a}{\A_\infty} + l(\tunnelextent{\tau_n}+1))$.
  
  Since $f \in C_b(\R)$, by Theorem (\ref{Cb-thm}), we conclude that
  \begin{equation*}
    \lim_{n\rightarrow\infty} \oppropinquity{}(f(\Dirac_n),f(\Dirac_\infty)) = 0 \text.
  \end{equation*}

  We now prove the first inequality, where we multiply a function of our Dirac operators on the left.
  
  Let $\varepsilon > 0$. There exists $N\in \N$ such that, if $n\geq N$, then $\oppropinquity{}(f(\Dirac_n),f(\Dirac_\infty)) < \frac{\varepsilon}{k}$.

  Let $n\geq N$. Let $\xi \in \dom{\Dirac_\infty}$ with $\CDN_\infty(\xi)\leq 1$. There exists $\eta \in \dom{\Dirac_n}$, with $\CDN_n(\eta)\leq 1$, such that
  \begin{equation*}
    \sup_{\substack{\omega\in\dom{\TDN_n} \\ \TDN_n(\omega)\leq 1}} \left|\inner{f(\Dirac_n)\eta}{\Pi_n(\omega)}{\Hilbert_n} - \inner{f(\Dirac_\infty)\xi}{\Psi_n(\omega)}{\Hilbert_\infty}\right| < \frac{\varepsilon}{k} \text.
  \end{equation*}

  Therefore,
  \begin{align*}
    \sup_{\substack{\omega\in\dom{\TDN_n} \\ \TDN_n(\omega)\leq 1}}
    & \left|\inner{b_n f(\Dirac_n)\eta}{\Pi_n(\omega)}{\Hilbert_n} - \inner{a f(\Dirac_\infty)\xi}{\Psi_n(\omega)}{\Hilbert_\infty}\right| \\
    &= \sup_{\substack{\omega\in\dom{\TDN_n} \\ \TDN_n(\omega)\leq 1}} \left|\inner{f(\Dirac_n)\eta}{b_n \Pi_n(\omega)}{\Hilbert_n} - \inner{f(\Dirac_\infty)\xi}{a \Psi_n(\omega)}{\Hilbert_\infty}\right| \\
    &= \sup_{\substack{\omega\in\dom{\TDN_n} \\ \TDN_n(\omega)\leq 1}} \left|\inner{f(\Dirac_n)\eta}{\Pi_n(d_n\omega)}{\Hilbert_n} - \inner{f(\Dirac_\infty)\xi}{\Psi_n(d_n\omega)}{\Hilbert_\infty}\right| \\
    &< \varepsilon \text.
  \end{align*}

  The reasoning is similar when the roles of $\Hilbert_n$ and $\Hilbert_\infty$ are reversed. Thus
  \begin{equation*}
    \lim_{n\rightarrow\infty} \oppropinquity{}(b_n f(\Dirac_n)b_n,a f(\Dirac_\infty)) = 0 \text.
  \end{equation*}

  We turn to the other convergence result. Let $\xi \in \dom{\CDN_\infty}$ with $\CDN_\infty(\xi) = 1$. Thus, $\CDN_\infty(a\xi) \leq C$ for any $C > G(\norm{a}{\A_\infty}+\Lip_\infty(a))\CDN_\infty(\xi) = G(\norm{a}{\A_\infty}+\Lip_\infty(a))$. Let $N' \in \N$ such that, if $n\geq N'$, then
  \begin{equation*}
    \tunnelsep{\tau_n}{f(\Dirac_n),f(\Dirac_\infty)} < \frac{\varepsilon}{\max\{2 C,8 H \}} \text. 
  \end{equation*}

  As before, for all $n\in\N$, we also choose $d_n \in \dom{\TLip_n}$ such that $\Pi_n(d_n) = a$, $\Psi_n(d_n) = b_n$, and $\TLip_n(d_n) \leq l$.

  Let $n\geq N'$. Therefore, there exists $\eta\in\dom{\CDN_n}$ with $\CDN_n(\eta)\leq C$, such that
  \begin{equation*}
    \sup_{\substack{\omega\in\dom{\TDN_n} \\ \TDN_n(\omega)\leq 1}} \left|\inner{f(\Dirac_\infty) a \xi}{\Pi_n(\omega)}{\Hilbert_\infty} - \inner{f(\Dirac_n)\eta}{\Psi_n(\omega)}{\Hilbert_n}\right| < \frac{\varepsilon}{2} \text.
  \end{equation*}

  Let $\psi \in \targetsettunnel{\tau_n}{\xi}{1}$. Let $\zeta \in \module{J}_n$ such that $\TDN_n(\zeta)\leq 1$, $\Pi_n(\zeta)=\xi$ and $\Psi_n(\zeta)=\psi$. Now $a\xi = \Psi_n(d_n \zeta)$, and $b_n\psi = \Pi_n(d_n \zeta)$ --- i.e. $b_n\psi \in \targetsettunnel{\tau_n}{a\xi}{C}$. By Lemma (\ref{s-set}), since $\eta\in s(a\xi,1)$ (using the notation of Lemma (\ref{s-set})), we conclude $\norm{b_n\psi - \eta}{\Hilbert_n} < \frac{\varepsilon}{2}$.

  Therefore, for all $\omega\in\dom{\TDN_n}$ with $\TDN_n(\omega)\leq 1$, we compute:
  \begin{align*}
    \Big|\inner{f(\Dirac_\infty) a \xi}{\Psi_n(\omega)}{\Hilbert_\infty}
    &- \inner{f(\Dirac_n) b_n \psi}{\Pi_n(\omega)}{\Hilbert_n} \Big| \\
    &\leq \Big|\inner{f(\Dirac_\infty) a \xi}{\Psi_n(\omega)}{\Hilbert_\infty} - \inner{f(\Dirac_n) \eta}{\Pi_n(\omega)}{\Hilbert_n} \Big| \\
    &\quad + \Big|\inner{f(\Dirac_n) (\eta-b_n \psi)}{\Pi_n(\omega)}{\Hilbert_n}  \Big| \\
    &\leq \frac{\varepsilon}{2} + \frac{\varepsilon}{2} \text.
  \end{align*}

  Once more, the reasoning applies equally well when the roles of $\Hilbert_n$ and $\Hilbert_\infty$ are reversed.
  
  Therefore $\lim_{n\rightarrow\infty} \oppropinquity{}(f(\Dirac_n)b_n,f(\Dirac_\infty)a) = 0$, as claimed, and our proof is complete.
\end{proof}

\bigskip

We now turn to the applications of our results so far to the spectral properties of Dirac operators in metric spectral triples, which converge under the spectral propinquity.

\section{Convergence of the Spectrum of Metric Spectral Triples}

This section contains the titular result of this paper: convergence in the spectral propinquity implies convergence of the spectrum. This result is based on Theorem (\ref{C0-thm}).

\begin{notation}
  The spectrum of a (possibly unbounded) operator $\Dirac$ on a Hilbert space $\Hilbert$ will be denoted by $\spectrum{\Dirac}$; thus $\spectrum{\Dirac}$ is the complement of the resolvent set of $\Dirac$, defined as the set of $\lambda\in\C$ such that $\Dirac-\lambda$ has an inverse which is bounded.
\end{notation}

\begin{theorem}\label{spectrum-thm}
  If the sequence $(\A_n,\Hilbert_n,\Dirac_n)_{n\in\N}$ of metric spectral triples converges to the metric spectral triple $(\A_\infty,\Hilbert_\infty,\Dirac_\infty)$ for the spectral propinquity, then
  \begin{equation*}
    \spectrum{\Dirac_\infty} = \left\{ \lambda\in\R : \exists (\lambda_n)_{n\in\N} \in \R^\N \quad \forall n \in \N \quad \lambda_n \in \spectrum{\Dirac_n} \text{ and } \lambda = \lim_{n\rightarrow\infty} \lambda_n \right\} \text.
  \end{equation*}
\end{theorem}

\begin{proof}
  We shall use the notation of Hypothesis (\ref{second-working-hyp}).
  
  Let $\lambda \in \spectrum{\Dirac_\infty}$.   Let $\varepsilon \in (0,1)$. As $\Dirac_\infty$ has compact resolvent, $\lambda$ is isolated in $\spectrum{\Dirac_\infty}$, so there exists $\delta\in(0,\varepsilon)$ such that $(\lambda-\delta,\lambda+\delta)\cap\spectrum{\Dirac_\infty} = \{\lambda\}$; and $\lambda$ is actually an eigenvalue of $\Dirac_\infty$. Let $f : \R\rightarrow \R$ be a continuous function such that $f(\lambda) = 1$, $\supp f \subseteq (\lambda-\delta,\lambda+\delta)$ (where $\supp f$ is the support of $f$).

  By Theorem (\ref{C0-thm}), there exists $N\in\N$ such that if $n\geq N$, then
  \begin{equation*}
    \tunnelsep{\tau_n}{f(\Dirac_n),f(\Dirac_\infty)} < \frac{\varepsilon}{3(|\lambda|+1)^2} \text.
  \end{equation*}

  Let $\xi$ be a normalized eigenvector of $\Dirac_\infty$ for $\lambda$. Note that by definition, $\CDN_\infty(\xi) = 1 + |\lambda|$. By construction $f(\Dirac_\infty)\xi = f(\lambda)\xi = \xi$.

  Let $n \geq N$. We conclude that there exists $\eta_n \in \Hilbert_n$, with $\CDN_n(\eta_n)\leq |\lambda|+1$, such that
  \begin{align}\label{spectrum-eq1}
    \sup_{\substack{\omega\in\dom{\TDN} \\ \TDN(\omega)\leq 1}}\left|\inner{f(\Dirac_n)\eta_n}{\Pi_n(\omega)}{\Hilbert_n} - \inner{f(\Dirac_\infty)\xi}{\Psi_n(\omega)}{\Hilbert_\infty}\right| < \frac{\varepsilon}{3(1+|\lambda|)} \text.
  \end{align}
  
  Since $\Psi_n$ is a modular quantum isometry, there exists $\omega\in\dom{\TDN}$ such that $\Psi_n(\omega) = \xi$, and $\TDN_n(\omega) \leq 1 + |\lambda|$. Thus $\inner{f(\Dirac_\infty)\xi}{\Psi_n(\omega)}{\Hilbert_\infty} = \inner{\xi}{\xi}{\Hilbert_\infty} = 1$, and therefore, we conclude by Expression (\ref{spectrum-eq1}):
  \begin{equation*}
    \left| \inner{f(\Dirac_n)\eta_n}{\Pi_n(\omega)}{\Hilbert_n} \right| > 1-\frac{\varepsilon}{3} \geq \frac{2}{3} \text.
  \end{equation*}
  
  Therefore, $f(\Dirac_n) \neq 0$. Thus, there exists $\lambda_n \in \spectrum{\Dirac_n}\cap(\lambda-\delta,\lambda+\delta)$; since $\delta\leq\varepsilon$, we conclude that $|\lambda-\lambda_n| < \delta \leq \varepsilon$. Thus, $\lambda$ is the limit of eigenvalues of $\Dirac_n$, as claimed.
  
  \medskip
  
  Now, let $\lambda \in \R$ be the limit of some sequence $(\lambda_n)_{n\in\N}$, where, for all $n\in\N$, the value $\lambda_n$ is an eigenvalue of $\Dirac_n$. Assume that $\lambda\notin\spectrum{\Dirac_\infty}$. As $\spectrum{\Dirac_\infty}$ is closed, there exists $\delta>0$ such that $(\lambda-\delta,\lambda+\delta)\cap\spectrum{\Dirac} = \emptyset$. Let $f : \R\rightarrow [0,1]$ be a continuous function such that $f(x) = 1$ if $|x-\lambda|<\frac{\delta}{2}$ and $f(x) = 0$ if $|x-\lambda|>\delta$.

  By construction, $f(\Dirac_\infty) = 0$. On the other hand, there exists $N\in \N$ such that, if $n\geq N$, then $|\lambda_n-\lambda| < \frac{\delta}{2}$. For each $n\geq N$, let $\xi_n$ be a normalized eigenvector $\Dirac_n$ for $\lambda_n$ (noting once more that since $\Dirac_n$ has compact resolvent, $\lambda_n$ is indeed an eigenvalue of $\Dirac_n$). Then $f(\Dirac_n)\xi_n = \xi_n$.

  By construction, $\CDN_n(\xi_n) = 1 + |\lambda_n| \leq 1+\lambda+\delta$.
  
  Now, there exists $N' \in \N$ such that, if $n\geq N'$, then
  \begin{equation*}
    \oppropinquity{}(f(\Dirac_n),f(\Dirac_\infty)) < \frac{\varepsilon}{(|\lambda|+\delta+1)^2} \text.
  \end{equation*}

  Fix $n\in\N$ with $n\geq \max\{ N, N' \}$. By Theorem (\ref{C0-thm}), there exists $\eta_n \in \dom{\Dirac_\infty}$ with $\CDN_\infty(\eta_n) \leq 1 + |\lambda| + \delta$ such that, for all $\omega \in \mathscr{J}_n$, with $\TDN_n(\omega)\leq 1$,
  \begin{align*}
    \frac{\varepsilon}{1+|\lambda|+\delta}
    &\geq \left| \inner{f(\Dirac_n)\xi_n}{\Pi_n(\omega)}{\Hilbert_n} - \inner{f(\Dirac_\infty)\eta_n}{\Psi_n(\omega)}{\Hilbert} \right| \\
    &= \left| \inner{\xi_n}{\Pi_n(\omega)}{\Hilbert_n} \right| \text.
  \end{align*}
  
  For $n\geq \max\{N,N'\}$, since $\Pi_n$ is a modular quantum isometry, there exists $\omega_n\in\mathscr{J}_n$ with $\TDN_n(\omega_n) \leq |\lambda|+1+\delta$, such that $\Pi_n(\omega_n)=\xi_n$. Therefore, (noting $\TDN_n\left(\frac{1}{1+|\lambda|+\delta}\omega_n\right)\leq 1$), we conclude:
  \begin{align*}
    \varepsilon
    &\geq \left|\inner{\xi_n}{\Pi_n(\omega_n)}{\Hilbert_n}\right| \\
    &=\left| \inner{\xi_n}{\xi_n}{\Hilbert_n} \right| = 1 \text.
  \end{align*}

  We have reached a contradiction. Hence, $\lambda\in\spectrum{\Dirac_\infty}$, and our proof is concluded.
\end{proof}

In general, of course, there may be very many ways to make convergent sequences of the type involved in Theorem (\ref{spectrum-thm}). However, there are many natural common situations where we can rephrase Theorem (\ref{spectrum-thm}) in simpler terms. An illustration of this is given by the following corollary. First, we work with self-adjoint operators with compact resolvent, so with spectrum equals to the point spectrum, and we assume, for simplicity, that our operators have at least one nonnegative eigenvalue (otherwise, we work with their opposite). To keep track of eigenvalues of various spectral triples in some sequence, we first fix an indexing scheme: we index all the eigenvalues of a self-adjoint operator (which are real) by an interval $Y$ in $\Z$ containing $0$ (i.e. a subset $Y$ of $\Z$ such that if $x < y < z \in \Z$ and $x,z\in Y$ then $y \in Y$), increasingly, so that index $0$ corresponds to the smallest nonnegative eigenvalue. So if $(\A,\Hilbert,\Dirac)$ is a spectral triple, we write $\spectrum{\Dirac}$ as $\{ \lambda_n : n \in Y \}$ for $Y \subseteq\Z$, $0\in Y$, $Y$ an interval, and $\lambda_n < \lambda_{n+1}$ for all $n\in Y$ with $n+1 \in Y$. Our scheme allows us to work with finite dimensional spectral triples. With this in mind, we obtain the following.

\begin{corollary}\label{simple-spectrum-cor}
  Let $(\A_n,\Hilbert_n,\Dirac_n)_{n\in\N}$ be a sequence of metric spectral triples converging, for the spectral propinquity, to a metric spectral triple $(\A_\infty,\Hilbert_\infty,\Dirac_\infty)$. For each $n\in\N\cup\{\infty\}$, we write $\spectrum{\Dirac_n} \coloneqq \{ \lambda_n^j : j \in Y_n \}$ where:
  \begin{enumerate}
  \item $Y_n$ is an interval in $\Z$ containing $0$,
  \item $\lambda_n^0 = \min\spectrum{\Dirac_n}\cap[0,\infty)$,
  \item $\lambda_n^j\leq\lambda_n^{j+1}$ for all $j \in Y_n$ such that $j+1 \in Y_n$.
  \end{enumerate}

  If there exists $(\delta_j)_{j\in\Z}$ in $(0,\infty)^\Z$  such that for all $n\in \N$ such that $j,j+1 \in Z_n$, then $\lambda_n^{j+1} - \lambda_n^j < \delta_j$, then:
    \begin{enumerate}
    \item if $0 \notin \spectrum{\Dirac_\infty}$, then for all $j\in Y_\infty$, $\lim_{n\rightarrow\infty} \lambda_n^j = \lambda_\infty^j$,
    \item otherwise, $\lambda_\infty^0 = 0$, and either, for all $j\in Y_\infty$, $\lim_{n\rightarrow\infty} \lambda_n^j = \lambda_\infty^j$, or for all $j\in Y_\infty$, $\lim_{n\rightarrow\infty}\lambda_n^{j-1} = \lambda_n^j$.
    \end{enumerate}
\end{corollary}

\begin{proof}
  We proceed by induction on $j$.

  Assume first that $0\notin \spectrum{\Dirac_\infty}$. By our choice of indexing, $\lambda_\infty^0 > 0$. By Theorem (\ref{spectrum-thm}), $\lambda_\infty^0 = \lim_{n\rightarrow\infty} \lambda_n^{k(n)}$ for some $k : \N \rightarrow \Z$.

  Assume first that, for all $n\in\N$, there exists $m \geq n$, such that $k(m) < 0$; thus $\lambda_m^{k(m)} \leq 0$ by assumption. In turn, this implies $\lambda_\infty^0 = 0$, which is also a contradiction. Thus, there exists $M \in \N$ such that, for all $n\geq M$, we have $k(n) \geq 0$.

  Assume now that, for all $n\in\N$, there exists $m \geq n$, such that $k(m) > 0$; thus $\lambda_m^{k(m)} > \delta_0 + \lambda_m^0$. Let $p : \N \rightarrow \N$, strictly increasing, such that $k(p(n)) > 0$ for all $n\in\N$.

  By construction, the sequence $(\lambda_{p(n)}^{k(p(n))})_{n\in\N}$, as a sub-sequence of $(\lambda_n^{k(n)})_{n\in\N}$, converges to $\lambda_\infty^0$ --- in particular, it is bounded. Consequently, $(\lambda_{p(n)}^0)_{n\in\N}$ is also bounded (below by $0$, above by an upper bound of $(\lambda_n^{k(n)})_{n\in\N}$). Thus $(\lambda_{p(n)}^0)_{n\in\N}$ converges to some $\mu \geq 0$. By Theorem (\ref{spectrum-thm}), $\mu \in \spectrum{\Dirac_\infty}$. Thus, $\mu > 0$.

  Therefore, $\lambda_\infty^0 = \lim_{n\rightarrow \infty} \lambda_{p(n)}^{k(p(m))} \geq \mu + \lambda_\infty^0 > \lambda_\infty^0$. This, too, is a contradiction. Thus, there exists $M' \in \N$ such that, if $n \geq M'$, then $k(n) \leq 0$.

  Thus, for $n\geq \max\{M,M'\}$, we conclude that $k(n) = 0$, i.e. $\lim_{n\rightarrow\infty} \lambda_n^0 = \lambda_\infty^0$.

  Assume, now, that we have proven that $\lim_{n\rightarrow\infty} \lambda_n^j = \lambda_\infty^j$ for all $j \in \{0,\ldots,N\}$, for some $N \in \N$. By Theorem (\ref{spectrum-thm}), there exists $k : \N \rightarrow \Z$ such that $\lim_{n\rightarrow\infty} \lambda_n^{k(n)} = \lambda_\infty^{N+1}$. Assume $N+1 \in Z_\infty$. We now proceed as above.

  First, if, for all $n \in \N$, there exists $n \geq n$, such that $k(m) \leq N$, then, for some strictly increasing $p : \N \rightarrow \N$, we have $k(p(m)) \leq N$, and thus $\lambda_{p(n)}^{k(p(n))} \leq \lambda_{p(n)}^N \xrightarrow{n\rightarrow\infty} \lambda_\infty^N < \lambda_\infty^{N+1}$ (using our induction hypothesis), which is a contradiction. Thus, there exists $M \in \N$ such that, if $n\geq M$, then $k(n) \geq N+1 $.

  Now, assume that there exists a strictly increasing function $q : \N \rightarrow \N$ such that $k(q(n)) > N+1$. Then, for $n\geq\max\{N_j,N_{j+1}\}$, we observe that:
  \begin{equation}\label{spectrum-cor-eq-1}
    \lambda_{q(n)}^{k(q(n))} \geq \lambda_{q(n)}^{N+2} \geq \lambda_{q(n)}^{N+1} + \delta_{N+1} \geq \lambda_{q(n)}^N + \delta_{N+1} + \delta_N > 0 \text.
  \end{equation}
  Since $(\lambda_{q(n)}^{k(q(n))})_{n\in\N}$ is convergent, it is bounded. Thus $(\lambda_{q(n)}^{N+1})_{n\in\N}$ is also bounded, and thus, it has a convergent sub-sequence $(\lambda_{q(r(n))}^{k(q(r(n)))})_{n\in\N}$, whose limit we denote by $\mu$.

  We then have, taking the limit in Expression (\ref{spectrum-cor-eq-1}):
  \begin{equation*}
    \lambda_\infty^{N+1} \geq \mu + \delta_{N+1} \geq \lambda_\infty^N + \delta_{N+1} + \delta_N > \lambda_\infty^N \text.
  \end{equation*}

  Since $(\A_{q(r(n))},\Hilbert_{q(r(n))},\Dirac_{q(r(n))})_{n\in\N}$ converges to $(\A_\infty,\Hilbert_\infty,\Dirac_\infty)$ for the spectral propinquity (as the spectral propinquity is indeed a metric), by application of Theorem (\ref{spectrum-thm}), we also have $\mu \in \spectrum{\Dirac_\infty}$. This is a contradiction, by assumption, since $\mu \in (\lambda_\infty^N,\lambda_\infty^{N+1})$. Therefore, there exists $M' \in \N$ such that $k(n) = N+1$ for all $n\geq N$.

  Therefore, we conclude, as needed, that $\lim_{n\rightarrow\infty} \lambda_n^{N+1} = \lambda_\infty^{N+1}$.

  Thus, our corollary holds for all $j \in \N$. If $Z_\infty = \Z$, we conclude our result for all $-j \in \N$, by using Corollary (\ref{minus-cor}), since $(\A_n,\Hilbert_n,-\Dirac_n)_{n\in\N}$ converges to $(\A_\infty,\Hilbert_\infty,-\Dirac_\infty)$.

  Now, if we assume instead that $0 = \lambda_\infty^0$, then by a similar reasoning, either $(\lambda_n^0)_{n\in\N}$ or $(\lambda_{n}^{-1})_{n\in\N}$ converges to $\lambda_\infty^0=0$. The induction to get our lemma is then identical as above, only differing in its starting point, $0$ or $-1$, leading to the stated result.
\end{proof}

\begin{remark}
  The assumptions of Corollary (\ref{simple-spectrum-cor}) imply that $\spectrum{\Dirac_\infty}\cap(0,\infty) \neq \emptyset$. If $\spectrum{\Dirac_\infty} \subseteq (-\infty,0]$ is infinite, then Corollary (\ref{simple-spectrum-cor}) may be applied as well, via Corollary (\ref{minus-cor}).
\end{remark}

We note that Corollary (\ref{simple-spectrum-cor}) can be extended in many ways. For instance, eigenvalues of metric spectral triples forming a convergent sequence may merge at infinity, in which case, Corollary (\ref{simple-spectrum-cor}) does not apply, but a modified version may (for instance, if only pairs of eigenvalues merge).

\bigskip

Our form of continuity of the spectrum of Dirac operators for the spectral propinquity can be strengthened, under reasonable assumptions (mostly, for bookkeeping), to prove convergence of the associated multiplicities. Now, knowing the spectrum of a self-adjoint operator with compact resolvent, together with the multiplicities of each eigenvalue, fully describe this operator, up to unitary equivalence. Of course, the spectral propinquity is a metric up to unitary equivalence of spectral triples, but the next results tell us that, in many cases, the spectral invariants of Dirac operators are actually continuous, i.e, that the spectral propinquity sees spectral information not only for spectral triples at distance zero from each other, but also for spectral triples which are close to each other.

We will use the following result to compute bounds of the dimension of certain subspaces.
\begin{lemma}\label{Gram-lemma}
  Let $\xi_1,\ldots,\xi_d$ be $d$ vectors in a Hilbert space $\Hilbert$. If there exists $\alpha>0$ such that, for all $j,k \in \{1,\ldots,d\}$, if $j\neq k$ then $|\inner{\xi_j}{\xi_k}{\Hilbert}| < \frac{\alpha}{d}$, while $\norm{\xi_j}{\Hilbert}^2 \geq \alpha$, then $\{ \xi_1,\ldots,\xi_d\}$ is linearly independent.
\end{lemma}

\begin{proof}
  Let:
  \begin{equation*}
    G = \begin{pNiceMatrix}
      \inner{\xi_1}{\xi_1}{\Hilbert} & \Cdots & \inner{\xi_1}{\xi_d}{\Hilbert} \\
      \Vdots & & \Vdots \\
      \inner{\xi_d}{\xi_1}{\Hilbert} & \Cdots & \inner{\xi_d}{\xi_d}{\Hilbert}
    \end{pNiceMatrix}
  \end{equation*}
  be the Gram matrix of the family $(\xi_j)_{j=1}^d$, and let $D$ be the diagonal matrix with entries $\norm{\xi_j}{\Hilbert}^2$ for $j \in \{1,\ldots,d\}$ --- so $D$ is the diagonal matrix obtained from $G$ by zeroing all non-diagonal elements of $G$. By assumption, $G = D(1 + M)$, for some $d\times d$ matrix $M$ whose entries  are less than $\frac{1}{d}$ (and zero on the diagonal, in particular). Therefore, the operator norm of $M$ is strictly less than $1$. Consequently, the Neumann series lemma implies that $1+M$ is invertible; since $D$ is invertible by assumption, so is $G$.

  Now, if $\sum_{j=1}^d \lambda_j \xi_j = 0$, then $G\begin{pmatrix} \lambda_1 \\ \vdots \\ \lambda_d \end{pmatrix} = 0$. Since $G$ is invertible, we conclude that $\lambda_1 = \ldots = \lambda_d = 0$, and our lemma is proven.
\end{proof}

\begin{notation}
  If $\lambda$ is an eigenvalue of an operator $\Dirac$, then the dimension of the eigenspace of $\Dirac$ for $\lambda$, also called the multiplicity of $\lambda$, is denoted by $\multiplicity{\lambda}{\Dirac}$.
\end{notation}

\begin{theorem}\label{liminf-mul-thm}
  If $(\A_n,\Hilbert_n,\Dirac_n)_{n\in\N}$ is a sequence of metric spectral triples which converges to a metric spectral triple $(\A_\infty,\Hilbert_\infty,\Dirac_\infty)$ for the spectral propinquity, if $\lambda\in\spectrum{\Dirac_\infty}$, and if there exists $\delta>0$ and $N\in \N$ such that $(\lambda-\delta,\lambda+\delta)\cap\spectrum{\Dirac_n}$ is a singleton $\{\lambda_n\}$ for all $n\geq N$, then we assert:
  \begin{equation*}
    \liminf_{n\rightarrow\infty} \multiplicity{\lambda_n}{\Dirac_n} \geq \multiplicity{\lambda}{\Dirac_\infty} \text.
  \end{equation*}
\end{theorem}

\begin{proof}
  Once more, we will use the notation of Hypothesis (\ref{second-working-hyp}).
  
  Let $\lambda \in \spectrum{\Dirac_\infty}$, and assume that some $\delta>0$, and for some $N\in\N$, we have $\spectrum{\Dirac_n}\cap(\lambda-\delta,\lambda+\delta) = \{\lambda_n\}$, for all $n\geq N$. Up to shrinking $\delta$, we can assume that $(\lambda-\delta,\lambda+\delta)\cap\spectrum{\Dirac_\infty} = \{\lambda\}$ (since $\spectrum{\Dirac_\infty}$ is discrete). It follows from Theorem (\ref{spectrum-thm}) that $(\lambda_n)_{n\in\N}$ converges to $\lambda$. Without loss of generality, assume from now on that $N = 0$.

  Let now $(\xi_1,\ldots,\xi_m)$ be an orthonormal basis in the spectral subspace of $\Dirac_\infty$, associated with $\lambda$, and with $m\geq 1$; thus
  \begin{equation*}
    \multiplicity{\lambda}{\Dirac_\infty} = m \text.
  \end{equation*}
  Note that $\Dirac_\infty$ has a compact resolvent, so $m \in \N\setminus\{0\}$.

  We also record that by construction, $\CDN_\infty(\xi_j) = 1 + |\lambda|$ for all $j \in \{1,\ldots,m\}$.

  Let now $f \in C_0(\R)$ such that $f(\lambda)=1$, $f(\R) = [0,1]$, and $f(x) = 0$ if $|x-\lambda|\geq \delta$. Thus, $f(\Dirac_\infty)$ is the spectral projection of $\Dirac_\infty$ on the eigenspace of $\lambda$; in particular, note that $f(\Dirac_\infty) \xi_j = \xi_j$ for all $j\in\{1,\ldots,m\}$.

  Let $\varepsilon \in \left(0,\frac{1}{m+1}\right)$. There exists $N \in \N$ such that, if $n\geq N$, then
  \begin{equation*}
    \tunnelmodsymmagnitude{\tau_n}{\mu_n} \leq \mu_n \leq \frac{\varepsilon}{9H(1+|\lambda|)^2} \text,
  \end{equation*}
  and
  \begin{equation*}
    \tunnelsep{\tau_n}{((U_n^t)_{t\in C_n}, f(\Dirac_n)),((U_\infty^t)_{t\in C_n},f(\Dirac_\infty))} \leq \mu_n \text.
  \end{equation*}

  Let $n\geq N$. Fix $j\in \{1,\ldots,m\}$. There exists $\eta_j \in \dom{\Dirac_n}$ such that $\CDN_n(\eta_j)\leq 1+|\lambda|$, and (since $9 H > 2$)
  \begin{equation}\label{mul-eq-1}
    \sup_{\substack{\omega\in\dom{\TDN_n} \\ \TDN_n(\omega)\leq 1}}\left|\inner{f(\Dirac_\infty)\xi_j}{\Pi_n(\omega)}{\Hilbert_n} - \inner{f(\Dirac_n)\eta_j}{\Psi_n(\omega)}{\Hilbert_\infty}\right| < \frac{\varepsilon}{2(1+|\lambda|)} \text.
  \end{equation}

  Now,
  \begin{align*}
    \CDN_n(f(\Dirac_n)\eta_j)
    &= \norm{f(\Dirac_n)\eta_j}{\Hilbert_n} + \norm{\Dirac_n f(\Dirac_n)\eta_j}{\Hilbert_n} \\
    &=  \norm{f(\Dirac_n)\eta_j}{\Hilbert_n} + \norm{f(\Dirac_n) \Dirac_n\eta_j}{\Hilbert_n} \\
    &\leq \opnorm{f(\Dirac_n)}{}{\Hilbert_n} \left(\norm{\eta_j}{\Hilbert_n} + \norm{\Dirac_n\eta_j}{\Hilbert_n} \right) \\
    &\leq 1\cdot \CDN_n(\eta_j) \leq 1+|\lambda| \text.
  \end{align*}

  By Corollary (\ref{inner-cor}), we then conclude:
  \begin{align*}
    \left|\inner{f(\Dirac_n)\eta_j}{f(\Dirac_n)\eta_k}{\Hilbert_n}\right|
    &\leq (9H)(1+|\lambda|)^2 \mu_n + \left| \inner{f(\Dirac_\infty)\xi_j}{f(\Dirac_\infty)\xi_k}{\Hilbert_\infty} \right| \\
    &\leq \varepsilon + \left| \inner{\xi_j}{\xi_k}{\Hilbert_\infty} \right| \\
    &= \begin{cases}
      1 + \varepsilon \text{ if $j=k$,} \\
      \varepsilon \text{ if $j\neq k$.}
    \end{cases}
  \end{align*}

  Corollary (\ref{inner-cor}) also implies that
  \begin{equation*}
    1 = \norm{\xi}{\Hilbert_\infty}^2 = \left|\inner{\xi_j}{\xi_j}{\Hilbert_\infty}\right| \leq \left|\inner{f(\Dirac_n)\eta_j}{f(\Dirac_n)\eta_j}{\Hilbert_n}\right| + \varepsilon
  \end{equation*}
  and thus $\left|\inner{f(\Dirac_n)\eta_j}{f(\Dirac_n)\eta_j}{\Hilbert_n}\right| \geq 1-\varepsilon$ for all $j\in \{1,\ldots,m\}$.

  Now, $\frac{1-\varepsilon}{m} - \varepsilon = \frac{1-(m+1)\varepsilon}{m} > 0$, so we can apply Lemma (\ref{Gram-lemma}) to conclude that $\left(f(\Dirac_n)\eta_1,\ldots,f(\Dirac_n)\eta_m\right)$ is linearly independent.
 As $f(\Dirac_n)$ is the projection on the eigenspace of $\Dirac_n$ for $\lambda_n$, we conclude that this eigenspace has dimension at least $m$, as long as $n\geq N$. This proves our result.
\end{proof}

In particular, we record that
\begin{corollary}
  If a sequence $(\A_n,\Hilbert_n,\Dirac_n)_{n\in\N}$ of metric spectral triples converges to a metric spectral triple $(\A_\infty,\Hilbert_\infty,\Dirac_\infty)$ for the spectral propinquity, such that:
  \begin{itemize}
  \item $\exists \delta>0 \quad \forall n\in \N\cup\{\infty\} \quad \forall \lambda\in\spectrum{\Dirac_n} \quad (\lambda-\delta,\lambda+\delta)\cap\spectrum{\Dirac_n} = \{\lambda\}$,
  \item $\forall n \in \N \quad \forall \lambda \in \spectrum{\Dirac_n} \quad \multiplicity{\lambda}{\Dirac_n} = 1$,
  \end{itemize}
  then $\multiplicity{\lambda}{\Dirac_\infty} = 1$ for all $\lambda\in\spectrum{\Dirac_\infty}$.
\end{corollary}

To get the converse inequality, a sufficient condition is given by requiring that the multiplicities of eigenvalues do not grow too fast, as controlled by the spectral propinquity. In particular, the converse holds when the multiplicities are bounded.

\begin{theorem}
  If $(\A_n,\Hilbert_n,\Dirac_n)_{n\in\N}$ is a sequence of metric spectral triples converging, for the spectral propinquity, to a metric spectral triple $(\A_\infty,\Hilbert_\infty,\Dirac_\infty)$, and
  \begin{enumerate}
  \item if $\lambda\in\spectrum{\Dirac_\infty}$,
  \item there exists $\delta>0$ and $N\in\N$ such that, for all $n\geq N$, the intersection $\spectrum{\Dirac_n}\cap(\lambda-\delta,\lambda+\delta)$ is a singleton, denoted by $\{\lambda_n\}$,
  \item if $(\multiplicity{\lambda_n}{\Dirac_n})_{n\in\N}$ converges in $\N$ --- i.e., is eventually constant,
  \end{enumerate}
  then
  \begin{equation*}
    \lim_{n\rightarrow\infty} \multiplicity{\lambda_n}{\Dirac_n} = \multiplicity{\lambda}{\Dirac_\infty} \text.
  \end{equation*}
\end{theorem}

\begin{proof}
  In this proof, we continue using the same notation as in Hypothesis (\ref{second-working-hyp}).
  
  We already have
  \begin{equation*}
    \liminf_{n\rightarrow\infty} \multiplicity{\lambda_n}{\Dirac_n} \geq \multiplicity{\lambda}{\Dirac_\infty} \text.
  \end{equation*}

  It is thus sufficient to prove that
    \begin{equation*}
      \limsup_{n\rightarrow\infty} \multiplicity{\lambda_n}{\Dirac_n} \leq \multiplicity{\lambda}{\Dirac_\infty} \text.
    \end{equation*}

    Let $N_1\in\N$ and $m \in \N$ such that $\multiplicity{\lambda_n}{\Dirac_n} = m$ for all $n\geq N$.
    
    Let $\varepsilon \in \left(0,\frac{1}{m+1}\right)$.  We also choose $f \in C_0(\R)$ such that $f(\R) = [0,1]$, $f(\lambda) = 1$, and $f(t) > 0$ if and only if $t \in (\lambda-\delta,\lambda+\delta)$.

    As in the proof of Theorem (\ref{liminf-mul-thm}), we choose $N_2 \in \N$ such that, if $n\geq N_2$, then
  \begin{equation*}
    \tunnelsep{\tau_n}{((U_n^t)_{t\in C_n}, f(\Dirac_n)),((U_\infty^t)_{t\in C_n},f(\Dirac_\infty))} \leq \frac{\varepsilon}{(9 H)(1+|\lambda|)^2}\text,
  \end{equation*}
  and
  \begin{equation*}
    \mu_n \leq \frac{\varepsilon}{9 H(1+|\lambda|)^2} \text.
  \end{equation*}
    
  Let $n\geq \max\{N,N_1,N_2\}$. Let $\{ \eta_1,\ldots,\eta_m \}$ be an orthonormal basis for the eigenspace of $\Dirac_n$ associated with the eigenvalue $\lambda_n$ (so $m=\multiplicity{\lambda_n}{\Dirac_n}$). As in Theorem (\ref{liminf-mul-thm}), there exists $\xi_1,\ldots,x_m$ in $\dom{\Dirac_\infty}$, with $\CDN_\infty(\xi_j)\leq 1 + |\lambda|$, such that, for all $\omega\in\mathscr{J}_n$ with $\TDN_n(\omega)\leq 1$, and for all $h \in \{ 1, f \}$,
  \begin{equation*}
    \left|\inner{h(\Dirac_n) \eta_j}{\Pi_n(\omega)}{\Hilbert_n} - \inner{h(\Dirac_\infty)\xi_j}{\Psi_n(\omega)}{\Hilbert_\infty}\right| < \frac{\varepsilon}{(1+|\lambda|)} \text.
  \end{equation*}

  The same argument as in the proof of Theorem (\ref{liminf-mul-thm}) based on Lemma (\ref{Gram-lemma}) shows that $\{\xi_1,\ldots,\xi_m\}$ are linearly independent, and thus $f(\Dirac_\infty)$ is a projection on a space of dimension at least $m$ --- as well as the spectral projection on the eigenspace of $\Dirac_\infty$ for $\lambda$. This concludes our proof.
\end{proof}

\bigskip

In particular, we conclude that we have found natural sufficient conditions for the convergence of certain traces of convergent metric spectral triples.
\begin{corollary}\label{spectral-action-cor}
  Let $(\A_n,\Hilbert_n,\Dirac_n)_{n\in\N}$ be a sequence of metric spectral triples, converging, for the spectral propinquity, to $(\A_\infty,\Hilbert_\infty,\Dirac_\infty)$. Let $f \in C_0(\R)$. Write $\spectrum{\Dirac_\infty}=\left\{ \lambda_\infty^j : j \in \Z \right\}$, using the same convention as Corollary (\ref{simple-spectrum-cor}). If:
  \begin{enumerate}
  \item for each $\lambda\in\spectrum{\Dirac_\infty}$, there exists $\delta_\lambda> 0$ and $N_\lambda\in \N$ such that, if $n\geq N_\lambda$, then $(\lambda-\delta_\lambda,\lambda+\delta_\lambda)\cap\spectrum{\Dirac_n}$ is a singleton, denoted by $\{\lambda_n^j\}$, (setting $\lambda_n^j=\infty$ when $n < N$),
  \item for each $j \in \Z$, the sequence $(\multiplicity{\lambda_n^j}{\Dirac_n})_{n\in\N}$ is convergent (i.e. eventually constant),
  \item there exists some sequence $(x_n)_{n\in\Z} \in \R^\Z$ with $(\multiplicity{\lambda_\infty^n}{\Dirac_\infty} x_n)_{n\in\N} \in \ell^1(\Z)$ such that, for all $n\in\N$ and $j \in \Z$, we have $f(\lambda_n^j) \leq x_n$ (with $f(\infty) = 0$),
  \end{enumerate}
  then
  \begin{equation*}
    \lim_{n\rightarrow\infty} \mathrm{Trace}(f(\Dirac_n)) = \mathrm{Trace}(f(\Dirac_\infty)) \text.
  \end{equation*}
\end{corollary}

\begin{proof}
  We use the notation of Corollary (\ref{simple-spectrum-cor}) --- if $\spectrum{\Dirac_\infty}\cap[0,\infty) = \emptyset$, we replace $\Dirac_\infty$ with $-\Dirac_\infty$. If we write $m_j = \multiplicity{\lambda_j}{\Dirac_\infty}$ for all $j\in\Z$, then, by the dominated convergence theorem, and by continuity of $f$, we conclude:
  \begin{align*}
    \mathrm{Trace}(f(\Dirac_n))
    &= \sum_{j \in \Z} \multiplicity{\lambda_n^j}{\Dirac_n} f(\lambda_n^j) \\
    &\xrightarrow{n\rightarrow\infty} \sum_{j\in\Z} m_j f(\lambda_\infty^j) \\
    &= \mathrm{Trace}(f(\Dirac_\infty)) \text,
  \end{align*}
  as claimed.
\end{proof}

In particular, it is common to define the \emph{action functional} of a spectral triple $(\A,\Hilbert,\Dirac)$ as $\mathrm{Trace}\left(f\left(\frac{1}{S}\Dirac\right)\right)$ for some smooth, positive, even, decreasing function $f$ and positive real number $S$ --- this functional is referred to as the spectral action of the spectral triple \cite{Connes96,Connes96b,Marcolli18}. Corollary (\ref{spectral-action-cor}) gives some sufficient conditions for \emph{the convergence of the spectral actions}. Spectral actions are used as the basis for the physical content of a spectral triple, and are seen as defined over a space of spectral triples --- whose subspace of metric spectral triples we now see is well behaved as a metric space for the spectral propinquity. We will explore this matter deeper in a sub-sequence work.

As a reminder of the discussion at the beginning of this work, we note that, if we were to define convergence of spectral triples simply in terms of their spectra and associated multiplicity function, we would get a much weaker form of convergence than the spectral convergence --- which, we just saw, imply convergence of this spectral data. Indeed, the spectral propinquity also accounts for the quantum, Connes' metric induced on the underlying C*-algebra of a spectral triple, which also must converge in the sense of the Gromov-Hausdorff propinquity \cite{Latremoliere13,Latremoliere13b,Latremoliere14} if we are to conclude convergence of the spectral triples for the spectral propinquity. In fact, the quantum metric is not an add-on, but rather a key concept, in our construction, since the extent of any tunnel between metric spectral triples is computed entirely using the Connes' metric on the state spaces of the various, involved, C*-algebras.

\bigskip

We conclude this paper with an application of our work, bringing our work in \cite{Latremoliere21a} on the convergence of certain spectral triples over fuzzy tori to a spectral triple over the classical torus.

Specifically, for each $n \in \N$, we define the two unitaries:
\begin{equation}\label{Clock-Shift-eq}
  U_n =
  \begin{pNiceMatrix}
    1 & & & \\ 
    & \exp\left(\frac{2i\pi}{n}\right) & & \\
    & & \Ddots & \\
    & & & \exp\left(\frac{2i(n-1)\pi}{n}\right)
  \end{pNiceMatrix} \text{ and }
    V_n =
  \begin{pNiceMatrix}
    0      & 1 & 0      & \Cdots &  0 \\
    \Vdots & \Ddots & \Ddots & \Ddots  & \Vdots \\
           &        &        &         & 0\\
    0      & \Cdots &        &         0 & 1 \\
    1      & 0      & \Cdots &           & 0
  \end{pNiceMatrix}
\end{equation}

The C*-algebra $\A_n \coloneqq C^\ast(U_n,V_n)$ is the C*-algebra of all $n\times n$ matrices, and it is universal in the sense that, if $u,v$ are two unitaries in any unital  C*-algebra $\A$, such that $u v = \exp\left(\frac{2i\pi}{n}\right) v u$ and $u^n = v^n = 1$, then there exists a unique *-monomorphism $\pi : \A_n\rightarrow \A$ such that $\pi(U_n) = u$ and $\pi(V_n) = v$. In particular, $U_n V_n = \exp\left(\frac{2i \pi}{n}\right) V_n U_n$ and $U_n^n = V_n^n = 1$.

Let $\gamma_1,\gamma_2,\gamma_3,\gamma_4$ be the usual $4\times 4$ gamma matrices. Writing $\Hilbert_n$ for the Hilbert space whose underlying vector space if $\A_n$, endowed with the inner product $a,b \in \A_n \mapsto \mathrm{Trace}(a^\ast b)$, we constructed the spectral triple $(\A_n,\Hilbert_n\otimes\C^4,\Dirac_n)$, where, for all $a\in \Hilbert_n$:
\begin{equation*}
  \Dirac_n (a\otimes e) = \frac{n}{2\pi}\left([\Re U_n, a]\otimes \gamma_1 e + [\Im U_n, a]\otimes\gamma_2 e + [\Re V_n, a]\otimes\gamma_3 e + [\Im V_n, a]\otimes\gamma_4 e \right) \text.
\end{equation*}

Now, we define $\partial_1$ and $\partial_2$ as the closures of the operators defined, for all finitely supported $\xi \in \ell^2(\Z^2)$:
\begin{equation*}
  \partial_1\xi : (m_1,m_2) \in \Z^2 \mapsto m_1 \xi(m_1,m_2) \text{ and }\partial_1\xi : (m_1,m_2) \in \Z^2 \mapsto m_2 \xi(m_1,m_2) \text.
\end{equation*}

Let $C(\T^2)$ be the C*-algebra of $\C$-valued, continuous functions over the $2$-torus $\T^2 = \{(z_1,z_2) \in \C : |z_1| = |z_2| = 1 \}$, acting on $\ell^2(\Z^2)$. We define
\begin{equation*}
  \Dirac_\infty (\xi\otimes e)  = \Re U \partial_2\xi \otimes\gamma_1 e + \Im U \partial_2 \xi \otimes \gamma_2 e + \Re V \partial_1 \xi \otimes \gamma_3 e + \Im V \partial_4 \xi\otimes \gamma_4 e \text.
\end{equation*}

We proved in \cite{Latremoliere21a}:
\begin{equation*}
  \lim_{n\rightarrow\infty} \spectralpropinquity{}((\A_n,\Hilbert_n\otimes\C^4,\Dirac_n),(\A_\infty,\Hilbert_\infty\otimes\C^4,\Dirac_\infty)) = 0\text.
\end{equation*}
This proof does not depend on the computation of the spectrum of the various spectral triples involved. By Theorem (\ref{spectrum-thm}), we can now conclude:
\begin{equation*}
  \spectrum{\Dirac_\infty} = \left\{ \lambda\in \R : \forall n \in \N \quad \exists \lambda_n \in \spectrum{\Dirac_n} \quad \lim_{n\rightarrow\infty} \lambda_n = \lambda \right\} \text.
\end{equation*}

Now, in \cite{Schreivogl13}, the spectrum $\spectrum{\Dirac_n}$ was computed. If we write $[x]_n = \frac{\sin(\frac{2\pi x}{n})}{\sin(\frac{2\pi}{n})}$, then
\begin{multline*}
  \spectrum{\Dirac_n} = \Big\{ \pm\Big(-([1-m]_n^2 + [m]_n^2 + [1+k]_n^2 + [k]_n^2) \\ \pm \sqrt{([1-m]_n^2 + [m]_n^2 +[1+k]_n^2+[k]_n^2)^2 - ([1/2-m]_n^2+[1/2+k]_n^2 - 2[1/2]^2)^2)}\Big)^{\frac{1}{2}} \\ : m,k \in \N \Big\}
\end{multline*}

We thus conclude, using our work, that the informally stated conclusion that the spectrum of $\Dirac_\infty$ in \cite{Schreivogl13} is indeed:
\begin{equation*}
  \spectrum{\Dirac_\infty} = \left\{ \pm\sqrt{m^2+n+n^2+1-m \pm \sqrt{2m^2-2m+2n+2n^2+1}} : n,m \in \N \right\} \text. 
\end{equation*}
This result is now not just the limit of spectra of spectral triples on fuzzy tori, as computed in \cite{Schreivogl13}, but actually the spectrum of the limit spectral triple on the torus.

\providecommand{\bysame}{\leavevmode\hbox to3em{\hrulefill}\thinspace}
\providecommand{\MR}{\relax\ifhmode\unskip\space\fi MR }

\providecommand{\MRhref}[2]{
  \href{http://www.ams.org/mathscinet-getitem?mr=#1}{#2}
}
\providecommand{\href}[2]{#2}

\vfill

\end{document}